\newif\ifArXiV

\ArXiVtrue %ArXiV style
\ifArXiV
\documentclass{article}
\else
\documentclass[authoryear,preprint]{elsarticle}
\fi

\usepackage{fullpage}

\usepackage{lmodern}
\usepackage[english]{babel}

\usepackage{amsmath,amssymb,amsthm,graphicx}
\usepackage{enumitem}
\usepackage{mathtools}
\usepackage{tikz}
\usepackage{hyperref}
\usepackage{subcaption}
\usepackage{xspace}
\usepackage{multicol}
\usepackage{placeins} % for \FloatBarrier = \clearpage without \newpage
\usepackage{xfrac}
\usepackage{marvosym}

\makeatletter
\def\ps@pprintTitle{%
	\let\@oddhead\@empty
	\let\@evenhead\@empty
	\def\@oddfoot{\centerline{\thepage}}%
	\let\@evenfoot\@oddfoot}
\makeatother

\newtheorem{theorem}{Theorem}
\newtheorem{definition}[theorem]{Definition}

\newtheorem{lemma}[theorem]{Lemma}
\newtheorem{corollary}[theorem]{Corollary}
\newtheorem{proposition}[theorem]{Proposition}

\theoremstyle{remark}

\newtheorem{examplex}[theorem]{Example}

\DeclareMathOperator{\st}{s.t.}
\newcommand{\obj}{z}
\newcommand{\loc}{J}
\newcommand{\cus}{I}
\newcommand{\myx}[1]{\ensuremath{x^{(#1)}}}
\DeclareMathOperator*{\supp}{supp}

\newcommand{\mytag}[1]{(\hypertarget{#1}{{#1}})}
\newcommand{\myref}[1]{\textnormal{(\hyperlink{#1}{#1})}}

\newcommand{\BC}{B\&C\xspace}
\newcommand{\pCP}{$p$CP\xspace}
\newcommand{\pnCP}{$p$NCP\xspace}
\newcommand{\pSCP}{$p$SCP\xspace}
\newcommand{\apCP}{$\alpha$N$p$CP\xspace}
\newcommand{\paCP}{$p\alpha$CCP\xspace}

\newcommand{\tsplib}{\texttt{TSPlib}\xspace}
\newcommand{\pmed}{\texttt{pmedian}\xspace}

\ifArXiV
\usepackage[affil-it]{authblk}
\usepackage[authoryear]{natbib}
\newenvironment{frontmatter}{}{}
\newenvironment{keyword}{\small \textbf{Keywords:}}{}
\let\address\affil
\fi

\begin{document}
	
\begin{frontmatter}
	\title{Investigating mixed-integer programming approaches for the  
	\texorpdfstring{$p$}{p}-\texorpdfstring{$\alpha$}{a}-closest-center problem}
	
	\ifArXiV
	\author[1]{Elisabeth Gaar\thanks{elisabeth.gaar@uni-a.de}}
	\author[1]{Sara Joosten\thanks{sara.joosten@uni-a.de}}
	\author[2]{Markus Sinnl\thanks{markus.sinnl@jku.at}}
	\affil[1]{Institute of Mathematics, University of Augsburg, Augsburg, 
	Germany}
	\affil[2]{Institute of Business Analytics and Technology Transformation/JKU 
	Business School, Johannes Kepler University Linz, Austria}			
	%\date{\today}
	\date{}
	\maketitle
	
	\else
	
	\author[unia]{Elisabeth Gaar}
	\ead{elisabeth.gaar@uni-a.de}
	\author[unia]{Sara Joosten}
	\ead{sara.joosten@uni-a.de}
	\author[jku]{Markus Sinnl}
	\ead{markus.sinnl@jku.at}
	\address[unia]{Institute of Mathematics, University of Augsburg, Augsburg, 
	Germany}
	\address[jku]{Institute of Business Analytics and Technology 
	Transformation/JKU Business School, Johannes Kepler University Linz, 
	Austria}
			
	\fi

	\begin{abstract}
	In this work, we introduce and study the~$p$-$\alpha$-closest-center 
	problem 
	(\paCP), which 
	is a generalization of the~$p$-second-center problem, a recently emerged 
	variant of the classical 
	(discrete)~$p$-center problem. In the~\paCP, 
	we are given a discrete set of customers, a discrete set of 
	potential facility locations, distances between each customer and 
	potential facility location as well as two integers~$p$ and~$\alpha$. The 
	goal is 
	to open facilities at~$p$ of the potential facility locations, such that 
	the 
	maximum~$\alpha$-distance between each customer and the open 
	facilities is minimized. The~$\alpha$-distance of a customer is 
	defined as the sum of distances from the customer to 
	its~$\alpha$ closest open facilities.
	If~$\alpha$ is set to one, 
	the \paCP is the~$p$-center 
	problem, and for~$\alpha$ being
	two, the~$p$-second-center problem is 
	obtained, for which the only existing algorithm in literature is a variable 
	neighborhood search (VNS). 
	
	We first prove relationships between the optimal objective function values 
	for different variants of the $p$-center problem to set the \paCP in 
	context 
	with 
	existing problems. 
	We then present four mixed-integer programming (MIP) formulations for 
	the~\paCP 
	and strengthen them by adding valid and optimality-preserving
	inequalities. 
	We also conduct a polyhedral study to prove relationships between 
	the linear programming relaxations of our MIP formulations.
	Moreover, we present iterative procedures for lifting some valid 
	inequalities
	to improve initial lower bounds on the optimal objective function value of 
	the~\paCP and characterize the best lower bounds obtainable by this 
	iterative lifting approach. 
	 
	Based on our theoretical findings, we develop a branch-and-cut algorithm 
	(\BC) to solve 
	the~\paCP exactly. We improve its performance by a starting and a primal 
	heuristic, variable fixings and separating inequalities. 
	In our computational study, we investigate the effect of the various 
	ingredients of our \BC  on benchmark instances from related literature. 
	Our \BC is 
	able to prove optimality for 17 of the 40 instances from the work on the 
	VNS heuristic.
	  
	\end{abstract}
\end{frontmatter}

\section{Introduction}\label{sec:introduction}

The (discrete) \emph{$p$-center problem (\pCP)} is a fundamental problem in location science (see, e.g., the book-chapter \citet{calik2019p}) that was first introduced in \citet{hakimi1965}.
In the \pCP, we are given a set of customer locations~$I$, a set of 
potential facility locations~$J$, an integer~$p<|J|$ and distances~$d_{ij}\geq0$ 
between each customer location $i\in I$ and potential facility location 
$j\in J$. The goal is to open~$p$~facilities in such a way that the maximum 
distance of any customer location to its closest open facility is 
minimized\footnote{We note that in literature, the problem is also sometimes 
defined using only one set of locations that represents the customer 
locations and potential facility locations at the same time. Note, that both 
versions can be transformed into another.}.
The problem is NP-hard as proven by \citet{kariv1979algorithmic} and one of its main application areas is the placement of emergency services and relief actions in humanitarian crisis as described in \citet{jia2007modeling,calik2013double} and \citet{lu2013robust}. Motivated by the fact that in such settings, facilities may fail for various reasons, different variants of the \pCP such as 
the \emph{$\alpha$-neighbor-$p$-center problem (\apCP)} defined in \citet{krumke1995generalization} 
and the \emph{$p$-next-center problem (\pnCP)} introduced by \citet{albareda2015centers}
have emerged to deal better with the issue of failing facilities. 
Another variant motivated by this issue is the \emph{$p$-second-center problem (\pSCP)}. It was first mentioned in \cite{lopez2019grasp} and only recently studied in \citet{ristic2023solving}, where the authors
present a heuristic solution approach. At the end of their work, they also propose that a generalization of this problem should be studied, 
which they call the \emph{$p$-$\alpha$-closest-center problem}. In this paper we follow this suggestion and start by giving its definition. 

\begin{definition}[The~$p$-$\alpha$-closest-center problem]
	Let~$I$ be a set of customer locations and~$J$ be a set of potential 
	facility locations with distances~$d_{ij}\geq 0$ between customer 
	locations~$i\in I$ and potential facility locations~$j \in J$. 
	Let $n=|I|$, $m=|J|$ and let~$\alpha$ and~$p$ be two integers with~$1 \leq 
	\alpha \leq p<m$. A feasible solution to the~\emph{$p$-$\alpha$-closest-center 
	problem (\paCP)} consists of a subset~$P \subseteq J$ of open facilities 
	with~$|P|=p$. Given a feasible solution~$P$ and a customer 
	location~$i \in I$, the $\alpha$-distance~$d_\alpha(P,i)$ is defined as 
	\begin{align*}
		d_\alpha(P,i)=\min_{\substack{A\subseteq P\\|A|=\alpha}} \sum_{j \in A} d_{ij}.
	\end{align*}
	The objective function value~$f_\alpha(P)$ of a feasible solution~$P$ is defined as  
	\begin{align*}
		f_\alpha(P)=\max_{i \in I} d_\alpha(P,i).
	\end{align*}
	Using these definitions, the \paCP can be formulated as
	\begin{align*}
		\min_{\substack{P \subseteq J\\ |P|=p}} f_\alpha(P). 
	\end{align*}
\end{definition}

For the sake of brevity, we also write customers instead of customer 
locations and facilities instead of potential facility locations in the remainder 
of this paper.
Note, that for~$\alpha=1$, the \paCP is exactly the \pCP, and for~$\alpha=2$ we obtain the \pSCP. 
The motivation behind the \paCP is that in case of failing facilities, the closest 
open facility for some customer in an optimal solution of the \pCP may 
fail, and thus this customer must travel to another open facility. As the \pCP 
only considers the closest open facility for each customer in the 
objective 
function, this other open facility may be far from it. The \paCP deals 
with 
this issue by considering the distance to $\alpha$ facilities in the objective 
function by minimizing the average distance from a customer to its 
$\alpha$ 
closest open facilities. 
Therefore, it minimizes 
the average distance this
customer needs to travel if its~$\beta$~closest open facilities fail over all $\beta \in \{0, 1, \ldots, \alpha -1\}$.
As our main contribution in this paper, we investigate the \paCP both from a theoretical and a practical side.

However, we also compare the \paCP to other variants of the \pCP to set it in context. Thus, we next discuss some of them in more detail. 
The~\apCP is defined as 
\[
	\min_{\substack{P\subseteq V\\|P|=p}} \max_{i \in V\setminus P} \min_{\substack{A\subseteq P\\ |A| = \alpha}} \max_{j\in A} d_{ij}.
\]
It aims to open a set of facilities $P$ such that the maximum distance between a 
customer and its $\alpha$-closest open facility is minimized.
The idea behind it is that the~$\alpha-1$ closest facilities to a 
customer may fail. Then, the customer needs to travel to the 
$\alpha$-closest facility, which therefore should not be too far away.
We note that unlike the setting in the~\paCP, the~\apCP has a single set 
$V$ of locations as input. The set of customer locations
is then defined as the set of elements where no facility is opened, i.e., $I = 
V\setminus P$ depends on the feasible solution~$P$.

A different approach is modeled by the~\pnCP.
The idea is that if the closest open facility to a customer has failed, 
the customer may only notice this when she has arrived at the failed facility. 
In 
this case, it can be useful if the next open facility is near the facility the 
customer has just arrived. 
Therefore, the \pnCP aims to minimize the maximum sum of distances from a 
customer location to its closest open facility and from this closest open 
facility to the closest open alternative facility. It can be formally defined as
\begin{align*}
	\min_{\substack{P\subseteq V\\|P|=p}} \max_{i\in I} \left( \min_{j\in P} 
	d_{ij} + \min_{\substack{k\in P\\ j'\in \arg \min_{j\in P} d_{ij}\\k\neq 
	j'}} d_{j'k} \right).
\end{align*}
Note, that the~\pnCP does not require the~$\alpha$~parameter, but compared to the~\paCP it additionally needs distances between potential facility locations as input. 

For the remainder of this paper, we use the following notation. 
For a problem, a mixed-integer program (MIP) or a linear program (LP) $F$, we denote the optimal objective function value of $F$ by $\nu(F)$. Moreover, we denote the LP-relaxation of a MIP ($F$) by ($F$-R).
Additionally, let~$\mathcal{J}^\alpha = \{A\subseteq J: |A| = \alpha\}$ be the set of all subsets of $J$ of cardinality $\alpha$. 
Then, we define $d_{iA} = \sum_{j\in A} d_{ij}$ and $d_{iA}^{\max} = \max \{d_{ij}: j\in A\}$ for all~$i\in \cus$ and~$A\in \mathcal{J}^\alpha$.
Furthermore, let $D^\alpha=\{d_{iA}: i\in I, A \in \mathcal{J}^\alpha\}$ be the set of all possible~$\alpha$-distances.

\subsection{Contribution and outline}

This work formally defines and studies the \paCP. We embed this new problem into the context of existing \pCP variants by proving relationships between their optimal objective function values. 
We further introduce 
four different mixed-integer programming (MIP) formulations of the \paCP, which are the very first for this problem and also for the \pSCP. 
We present several valid inequalities and so-called optimality-preserving inequalities, i.e., inequalities that may cut off feasible solutions, but that do not change the optimal objective function value. We conduct a polyhedral study on all formulations and study their semi-relaxations, in which only some of the binary variables are relaxed. 

For three of our MIP formulations, we also present valid inequalities that incorporate a lower bound on the optimal objective function of the \paCP, i.e., lifted inequalities. Adding these inequalities iteratively yields a procedure to improve such lower bounds similar to the approaches in \citet{gaar2022scaleable} and \citet{gaar2023exact} for the \pCP and the \apCP, respectively.
We prove convergence of this iterative procedure and give a characterization of the best obtainable lower bound as a solution of a set cover problem variant for some of the MIP formulations. Moreover, we compare the three best obtainable lower bounds in this way with each other and prove that two of them coincide.

We develop a branch-and-cut algorithm based on our theoretical results, which is the first algorithm to solve the~\paCP and the~\pSCP exactly.  
We introduce a starting and a primal heuristic as well as variable fixing procedures and study their effects on the algorithms performance. 
Our computational tests are made on~93~different instances from the literature, 
of which we could solve~52~instances to optimality. 
We compare our results to the heuristic of \citet{ristic2023solving} and prove optimality of their solutions for~17 instances.

In the remainder of this section we provide an overview of related work. In Section~\ref{sec:pcp_variants}, we discuss the relationships between several \pCP variants, including the~\apCP, the~\pnCP and the~\paCP. 
Section~\ref{sec:MIPformulations} contains our four MIP formulations of the \paCP together with strengthening inequalities and in Section~\ref{sec:polyhedral_study}, a polyhedral study is conducted.
Section~\ref{sec:lifting_D3} contains the lifted inequalities and the iterative procedure for our third and fourth MIP formulation. Section~\ref{sec:lifting_D1} describes a similar approach for our first MIP formulation, which is much more involved for this formulation compared to that for the third and fourth formulation.
In Section~\ref{sec:implementation}, we present our branch-and-cut algorithm based on our first MIP formulation and some implementation details. The results of our computational experiments are discussed in Section~\ref{sec:computational} and Section~\ref{sec:conclusion} provides closing remarks.

\subsection{Literature review}
\label{sec:litreview}

\cite{lopez2019grasp} are the first to mention the \pSCP as a possible new \pCP variant. However, they do not investigate it. This was only recently done by \cite{ristic2023solving}. These authors propose a variable neighborhood search (VNS) to heuristically solve the \pSCP. They also mention the NP-hardness of the problem as it is a generalization of the NP-hard \pCP.
Aside from \cite{lopez2019grasp} and \cite{ristic2023solving} we are not aware of any other work considering the \pSCP. For the related problems mentioned in the introduction, there is various existing research as we discuss next.  

\paragraph{The~$p$-center problem}
The \pCP itself was already introduced by \citet{hakimi1965} and since its introduction there has been an enormous amount of work on it, both from the exact as well as from the (meta-)heuristic side. As we are concerned with the design of exact solution algorithms in our work, we focus this overview on exact methods and refer to \citet{garcia2019survey} for heuristics and approximation algorithms.

\citet{minieka1970} presents the first exact algorithm for 
the \pCP. This approach uses the connection of the \pCP to the \emph{set cover problem (SCP)}, and aims to identify a set of open facilities of minimum cardinality such that every customer demand point is within a given radius of at least one open facility. 
Consequently, the \pCP can be solved in an iterative fashion by solving SCPs. 
Many follow-up works such as 
\citet{garfinkel1977,ilhan2001,ilhan2002,caruso2003,alKhedhairi2005,chen2009} 
and \citet{contardo2019scalable} exist.

There are also various MIP formulations for solving the \pCP to proven 
optimality. The classical textbook formulation of the \pCP 
(see, e.g., \citet{daskin2013network}) uses two sets of variables, namely facility opening variables and 
assignment variables, and its linear programming (LP)-relaxation is known to have bad bounds as shown for example in \citet{snyder2011fundamentals}. \citet{elloumi2004} present an 
alternative MIP formulation, which was later simplified in \citet{ales2018}. They show that the LP-relaxation of their formulation can give better bounds compared to the relaxation of the classical formulation. In \citet{calik2013double} another MIP formulation of the  
\pCP is developed and the authors prove that its LP-relaxation gives the same bounds as the LP-relaxation of the formulation of \citet{elloumi2004}. In \citet{gaar2022scaleable}, a further MIP formulation is obtained by projecting out the assignment variables from the classical formulation, following the Benders approach of \citet{fischetti2017redesigning} for the uncapacitated facility location problem. Moreover, the authors also present an iterative lifting scheme for the inequalities in the new formulation, which, 
given a lower bound on the optimal objective function value of the \pCP, potentially improves this lower bound. 
They prove that the best lower bounds obtainable by this procedure are the same bounds that are obtained by a semi-relaxation of \citet{elloumi2004}.

\paragraph{The~$\alpha$-neighbor-$p$-center problem}
The \apCP was first introduced and studied by the approximation-algorithms community in
\citet{krumke1995generalization,chaudhuri1998p} and \citet{khuller2000fault}. In the 2020s, various heuristic works on the problem appeared, namely a greedy randomized adaptive search procedure (GRASP) proposed by \citet{sanchez2022grasp}, a local search by \citet{mousavi2023exploiting} and a parallel VNS by \citet{chagas2024parallel}. Moreover, \citet{gaar2023exact} introduce two MIP formulations of the problem, one based on the classical textbook formulation of the \pCP and one based on the formulation of 
\citet{ales2018} of the \pCP. For the first formulation, they present valid inequalities using the lower bounds from the LP-relaxation and how these inequalities can be used in an iterative lifting-scheme akin to their approach for the \pCP in \citet{gaar2022scaleable}. For the second formulation, they present an iterative variable fixing scheme based on lower bounds from the LP-relaxation.

As described above, in the \apCP, we are given one set of locations, and for a given solution the customer locations are defined as the locations where 
no facility is opened.
However, there is some literature on a variant in which all locations are customer locations. This variant is referred 
to as the 
\emph{$\alpha$-all-neighbors-$p$-center problem} in \citet{khuller2000fault} and 
the \emph{fault-tolerant-$p$-center problem} in \citet{elloumi2004}.

\paragraph{The~$p$-next-center problem}
\citet{albareda2015centers} were the first to introduce the \pnCP. They prove the NP-hardness of this problem and propose three MIP formulations, which are based on two-index, three-index and covering variables. 
Furthermore, they compare the performance of those formulations when given to an off-the-shelf MIP solver.
The first heuristics for the \pnCP were proposed by \citet{lopez2019grasp}. They develop a VNS and a GRASP and test a hybridized algorithm as well. Later, \citet{londe2021evolutionary} propose an evolutionary algorithm for the \pnCP and \citet{ristic2023auxiliary} present a different VNS heuristic that is based on the VNS of \citet{mladenovic2003solving} for the \pCP.  

\section{Relationships between \texorpdfstring{$p$}{p}CP variants}\label{sec:pcp_variants}

It becomes clear from the introduction that the described \pCP variants are all closely related. 
In this section, we study these relationships in more detail and prove some lower and upper bounds on the optimal objective function value of one \pCP variant by the one of another \pCP variant. 

We start by showing an example of optimal solutions for the~\paCP for different values of~$\alpha$.  
We consider the instance \texttt{att48} of the \tsplib instance set from 
\citet{reinelt1991tsplib} with $I=J$. The distances in this instance are defined 
as the Euclidean distances between the points in $I=J$, which are given as 
coordinates in the two-dimensional plane. 
In Figure~\ref{fig:instance}, we
depict an optimal solution of the \paCP with $p=10$ and $\alpha \in \{1,2,3\}$, i.e., for the~\pCP, the~\pSCP and the~$p$$3$CCP. 
For this instance, all obtained optimal solutions differ.
The corresponding optimal objective function values are 1203.18, 2827.72 and 4895.52, respectively.
For a smaller example showing different optimal solutions for the \pCP, the \pSCP and the \pnCP, we refer to \citet{ristic2023solving}.

\begin{figure}[h]
	\centering
	\includegraphics[scale=0.2]{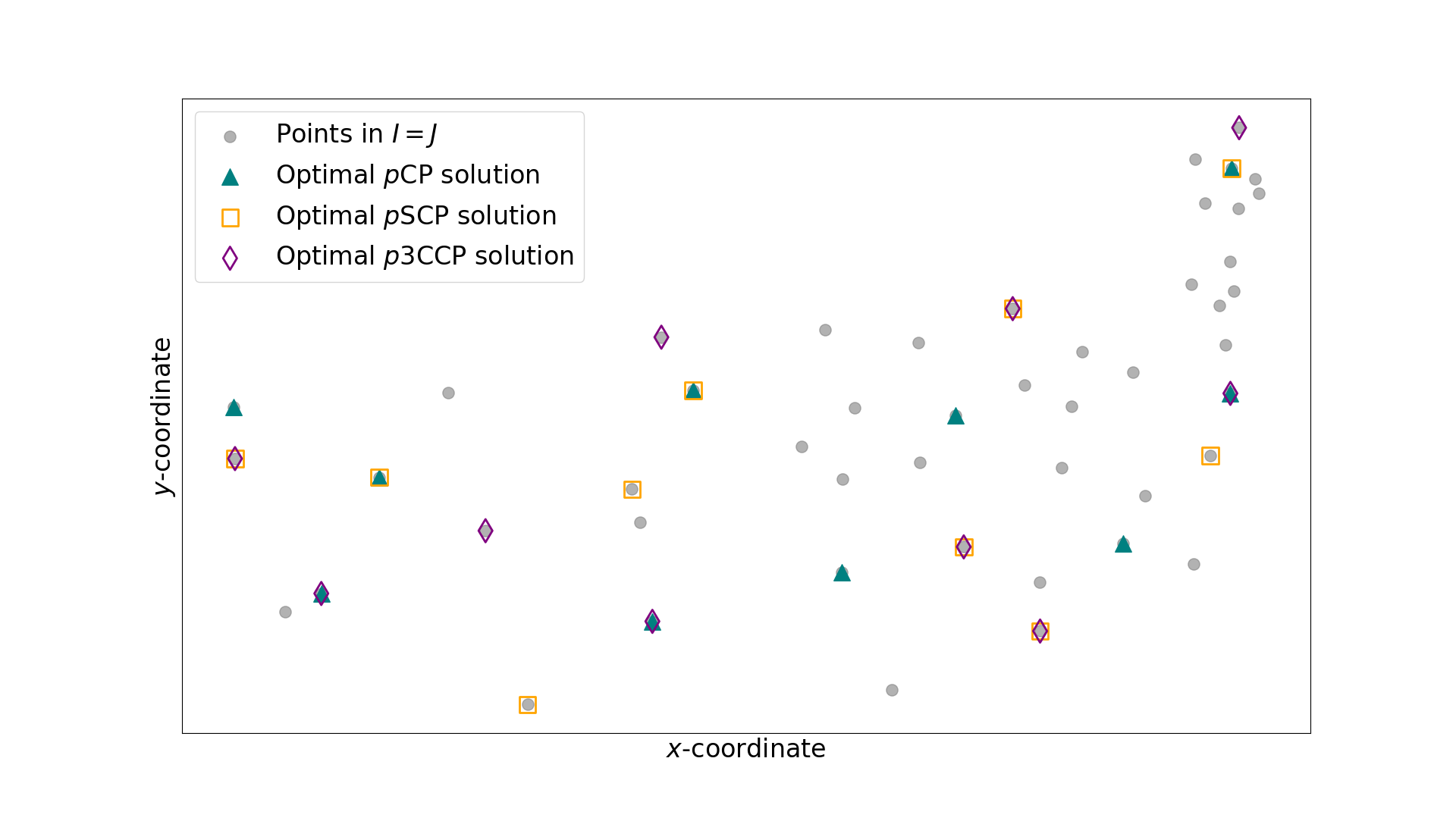}
	\caption{The locations of the \texttt{tsplib} instance 
	\texttt{att48} 
	with optimal solutions for the \pCP, the \pSCP and the $p$3CCP with $p=10$}
	\label{fig:instance}
\end{figure}

Rather than considering optimal solutions, we might also study optimal objective function values. 
The following theorem provides more detailed information about the relationships between the optimal objective function values of the \pCP, the~\apCP, the~\pnCP and the~\paCP. 
Since the \apCP is only defined for a single set representing customer 
locations and 
potential facility locations at the same time, and because for the \pnCP we need 
distances between all pairs of potential facility locations in $J$, we here assume 
$I=J$ and denote the set by $V$. 

\begin{theorem}\label{thm:pcp_relations}
	Let~$V=I=J$ be the set of customer and potential facility 
	locations, let~$p < |V|$ be an integer and let $d_{ij} \geq 0$ be the distance 
	between two locations $i,j\in V$. 
	Then, the following statements are true:  
	\begin{alignat*}{3}
		& 1. \quad & \nu(\text{\pCP}) &\leq \nu(\text{\pnCP}) && 
		\\
		& 2. & \nu(\text{\pCP}) &\leq \nu(\text{\apCP}) \qquad && \textrm{for all } \alpha \leq p \textrm{ if } d_{ii} = 0 \textrm{ for all } i\in V \\
		& 3. & \nu(\text{\apCP}) &\leq \nu(\text{\paCP})  && \textrm{for all } \alpha \leq p \\
		& 4. & \nu(\text{\apCP}) &\leq \nu(\text{$\beta$N$p$CP})  && \textrm{for all } \alpha \leq \beta \leq p \\
		& 5. & \nu(\text{\paCP}) &\leq \nu(\text{$p$$\beta$CCP})  && \textrm{for all } \alpha \leq \beta \leq p \\
		& 6. & \nu(\text{2N$p$CP}) &\leq \nu(\text{\pnCP}) && \text{if the distances satisfy the triangle} \\
		&    &                     &                       && \qquad \text{inequalities $d_{ij} \leq d_{ik} + d_{kj}$ for all $i,j,k\in V$}. 
	\end{alignat*}
\end{theorem}

\begin{proof}
	\begin{enumerate}
		\item Let $P^*$ be an optimal solution of the~\pnCP. Then, 
		\begin{align*}
			\nu(\text{\pCP}) &=
			\min_{\substack{P\subseteq V\\ |P|=p}} \ \max_{i\in V} \ \min_{j\in P} \ d_{ij} 
			\leq \max_{i\in V} \ \min_{j\in P^*} \ d_{ij} 
			= \max_{i\in V} \ ( \min_{j\in P^*} \ d_{ij} + 0)  \\
			&\leq \ \max_{i\in V} \left(\min_{j\in P^*} \ d_{ij} + \min_{\substack{k\in P^*\\ j'\in \arg \min _{j\in P^*} d_{ij},\\k\neq j'}} \ d_{j'k}\right) 
			= \nu(\text{\pnCP}).
		\end{align*} 
		
		\item Let $P^*$ be an optimal solution of the~\apCP. Then, 
		\begin{align*}
			\nu(\text{\pCP}) 
			&= \min_{\substack{P\subseteq V\\ |P|=p}} \max_{i\in V} \min_{j\in P} d_{ij}
			\leq \max_{i\in V} \min_{j\in P^*} d_{ij} 
			= \max_{i\in V\setminus P^*} \min_{j\in P^*} d_{ij} \\
			&= \max_{i\in V\setminus P^*} \min_{\substack{A\subseteq P^*\\ |A|=1}} \max_{j\in A} d_{ij} 
			\leq \max_{i\in V\setminus P^*} \min_{\substack{A\subseteq P^*\\ |A|=\alpha}} \max_{j\in A} d_{ij} 
			= \nu(\text{\apCP})
		\end{align*} 
		where the second equality holds by the assumption~$d_{ii} = 0$ for all~$i\in V$.
		
		\item Let $P^*$ be an optimal solution of the~\paCP. Then, 
		\begin{align*}
			\nu(\text{\apCP})  
			&= \min_{\substack{P\subseteq V\\ |P|=p}} \max_{i\in V\setminus P} \min_{\substack{A\subseteq P\\ |A|=\alpha}} \max_{j\in A} d_{ij}
			\leq \max_{i\in V\setminus P^*} \min_{\substack{A\subseteq P^*\\ |A|=\alpha}} \max_{j\in A} d_{ij} \\
			& \leq \max_{i\in V} \min_{\substack{A\subseteq P^*\\ |A|=\alpha}} \max_{j\in A} d_{ij} 
			\leq \max_{i\in V} \min_{\substack{A\subseteq P^*\\|A|=\alpha}} \sum_{j\in A} d_{ij}
			= \nu(\text{\paCP}).
		\end{align*}
		
		\item Let $P^*$ be an optimal solution of the~$\beta$N$p$CP. If~$\alpha \leq \beta$, then 
		\begin{align*}
			\nu(\text{\apCP}) 
			&= \min_{\substack{P\subseteq V\\ |P|=p}} \max_{i\in V\setminus P} \min_{\substack{A\subseteq P\\ |A|=\alpha}} \max_{j\in A} d_{ij} 
			\leq \max_{i\in V\setminus P^*} \min_{\substack{A\subseteq P^*\\ |A|=\alpha}} \max_{j\in A} d_{ij} \\
			&\leq \max_{i\in V\setminus P^*} \min_{\substack{A\subseteq P^*\\ |A|=\beta}} \max_{j\in A} d_{ij} 
			= \nu(\text{$\beta$N$p$CP}).
		\end{align*}
		
		\item  Let $P^*$ be an optimal solution of the~$p$$\beta$CP. If~$\alpha \leq \beta$, then 
		\begin{align*}
			\nu(\text{\paCP}) 
			&= \min_{\substack{P\subseteq V\\ |P|=p}} \max_{i\in V\setminus P} \min_{\substack{A\subseteq P\\|A|=\alpha}} \sum_{j\in A} d_{ij} 
			\leq \max_{i\in V\setminus P^*} \min_{\substack{A\subseteq P^*\\ |A|=\alpha}} \sum_{j\in A} d_{ij} \\
			&\leq \max_{i\in V\setminus P^*} \min_{\substack{A\subseteq P^*\\ |A|=\beta}} \sum_{j\in A} d_{ij}
			= \nu(\text{$p$$\beta$CCP}).
		\end{align*}
		
			\item Let $P^*$ be an optimal solution of the~\text{\pnCP}. Then, 
		\begin{align*}
			\nu(\text{2N$p$CP}) 
			&= \min_{\substack{P\subseteq V\\ |P|=p}} \max_{i\in V\setminus P} \min_{\substack{A\subseteq P\\ |A|=2}} \max_{j\in A} d_{ij}
			\leq \max_{i\in V\setminus P^*} \min_{\substack{A\subseteq P^*\\ |A|=2}} \max_{j\in A} d_{ij} 
			\leq \max_{i\in V} \min_{\substack{A\subseteq P^*\\ |A|=2}} \max_{j\in A} d_{ij} \\
			&= \max_{i\in V} \min_{\substack{A\subseteq P^*\\|A|=2\\A=\{j,k\}}} \max \{d_{ij}, d_{ik}\} 
			\leq \max_{i\in V} \min_{\substack{A\subseteq P^*\\ |A|=2\\A=\{j,k\}}} (d_{ij} + d_{jk})
			= \max_{i\in V} \min_{\substack{j,k\in P^*}} (d_{ij} + d_{jk}) 
			\\
			&\leq \max_{i\in V} \left(\min_{j\in P^*} d_{ij} + \min_{\substack{k\in P^*\\ j'\in \arg \min _{j\in P^*} d_{ij},\\ k\neq j'}} d_{j'k}\right) 
			= \nu(\text{\pnCP})
		\end{align*}
		where the third inequality is true because for all $i,j,k \in V$ it holds that (i) $d_{ij} \leq d_{ij}+d_{jk}$ 
		since $d_{jk} \geq 0$ and that (ii)
		$d_{ik} \leq d_{ij}+d_{jk}$ by the triangle inequality.
	\end{enumerate}
\end{proof} 

The following examples show that no relationship between the optimal objective function values of the~\apCP with $\alpha=3$ and the~\pnCP and the~\pSCP holds in general. 

\begin{examplex}\label{ex:3npcp_small}
	Let~$I = J = \{1, 2, 3, 4\}$,~$p=\alpha=3$ and $d_{1,1} = d_{2,2} = d_{3,3}= d_{4,4} = 0, d_{1,2}=d_{2,1}=d_{1,4}=d_{4,1}=d_{2,3}=d_{3,2}=d_{3,4}=d_{4,3}=1, d_{1,3}=d_{3,1}=d_{2,4}=d_{4,2}=\sqrt{2}$.
	One can verify that~$P=\{2,3,4\}$ is an optimal solution of the \apCP with optimal objective function value~$\sqrt{2}$. The solution~$P$ is also optimal for the \pnCP and the \pSCP. However, their optimal objective function values are both $2$. So for this instance $\nu(\text{\apCP}) < \nu(\text{\pnCP})$ and $\nu(\text{\apCP}) < \nu(\text{\pSCP})$ hold for $\alpha=3$.
	
	Figure~\ref{fig:3npcp_small} illustrates Example~\ref{ex:3npcp_small}. The set~$I=J$ is depicted by round nodes in Figure~\ref{fig_small:instance}. In Figures~\ref{fig_small:3NpCP},~\ref{fig_small:pNCP}, and~\ref{fig_small:pSCP}, the square nodes correspond to open facilities in an optimal solution to the~$3$N$p$CP, the~\pnCP, and the~\pSCP, respectively.
\end{examplex}

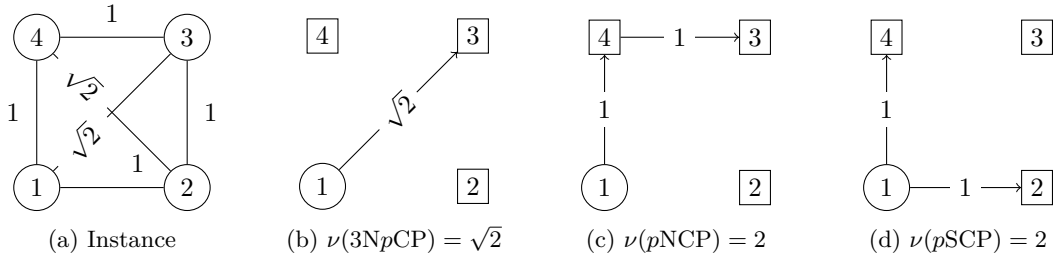
\begin{figure}[ht]
	\centering
	\begin{subfigure}{0.22\textwidth}
		\centering
		\begin{tikzpicture}[every node/.style={circle}]
			% outer square
			\node (a) at (0,0) [draw] {1};
			\node (b) at (2,0) [draw] {2};
			\node (c) at (2,2) [draw] {3};
			\node (d) at (0,2) [draw] {4};
			
			% edges with weight 1
			\draw (a) -- node[above, near end] {$1$} (b);
			\draw (b) -- node[right] {$1$} (c);
			\draw (c) -- node[above] {$1$} (d);
			\draw (d) -- node[left] {$1$} (a);
			
			% diagonals
			\draw (a) -- node[near start, sloped, fill=white] {$\sqrt{2}$} (c);
			\draw (b) -- node[near end, sloped, fill=white] {$\sqrt{2}$} (d);
		\end{tikzpicture}
		\caption{Instance}
		\label{fig_small:instance}
	\end{subfigure}
	\begin{subfigure}{0.22\textwidth}
		\centering
		\begin{tikzpicture}
			\node (a) at (0,0) [draw,circle] {1};
			\node (b) at (2,0) [draw] {2};
			\node (c) at (2,2) [draw] {3};
			\node (d) at (0,2) [draw] {4};
			
			\draw[->] (a) -- node[sloped, fill=white] {$\sqrt{2}$} (c);
		\end{tikzpicture}
		\caption{$\nu(\text{3N$p$CP}) = \sqrt{2}$}
		\label{fig_small:3NpCP}
	\end{subfigure}
	\begin{subfigure}{0.22\textwidth}
		\centering
		\begin{tikzpicture}		
			\node (a) at (0,0) [draw] {4};
			\node (b) at (2,0) [draw] {3};
			\node (c) at (0,-2) [circle, draw] {1};
			\node (d) at (2,-2) [draw] {2};
			
			\draw[->] (c) -- node[fill=white] {1} (a);
			\draw[->] (a) -- node[fill=white] {1} (b);			
		\end{tikzpicture}
		\caption{ $\nu(\text{\pnCP}) = 2$}
		\label{fig_small:pNCP}
	\end{subfigure}
	\begin{subfigure}{0.22\textwidth}
		\centering
		\begin{tikzpicture}			
			\node (a) at (0,0) [draw] {4};
			\node (b) at (2,0) [draw] {3};
			\node (c) at (0,-2) [draw, circle] {1};
			\node (d) at (2,-2) [draw] {2};
			
			\draw[->] (c) -- node[fill=white] {1} (a);
			\draw[->] (c) -- node[fill=white] {1} (d);			
		\end{tikzpicture}
		\caption{$\nu(\text{\pSCP}) = 2$}
		\label{fig_small:pSCP}
	\end{subfigure}
	\caption{Illustration of Example~\ref{ex:3npcp_small} 
	}
	\label{fig:3npcp_small}
\end{figure}

\begin{examplex}\label{ex:3npcp_big}
	Let~$I = J = \{1,2,3,4,5,6\}$, $p=4$,~$\alpha=3$ and symmetrical distances given by~$d_{1,1} = d_{2,2} = d_{3,3}= d_{4,4} = d_{5,5} = d_{6,6} = 0, d_{1,2} = d_{2,3}=d_{4,5}=d_{5,6}=1,d_{1,3}=d_{4,6}=2, d_{1,6}=d_{2,5}=d_{3,4}=4, d_{1,5}=d_{2,6}=d_{2,4}=d_{3,5}=\sqrt{17}, d_{1,4}=d_{3,6}=\sqrt{20}$.
	One can verify that~$P_1=\{1,2,4,6\}$ is an optimal solution of the \apCP with objective function value~$4$. The solution~$P_2=\{2,3,4,5\}$ is optimal for the \pnCP with an objective function value of~$2$. An optimal solution of the \pSCP is~$P_3=\{1,3,4,6\}$ with objective 
	function value $2$. Therefore, for this instance $\nu(\text{\apCP}) > \nu(\text{\pnCP})$ and $\nu(\text{\apCP}) > \nu(\text{\pSCP})$ hold for $\alpha=3$.
	
	Figure~\ref{fig:3npcp_big} illustrates Example~\ref{ex:3npcp_big}. The set~$I=J$ is again depicted by round nodes in Figure~\ref{fig_big:instance}, but for easier readability not all distances are shown. In Figures~\ref{fig_big:3NpCP},~\ref{fig_big:pNCP} and~\ref{fig_big:pSCP} the square nodes correspond again to the open facilities in an optimal solution of the~$3$N$p$CP, the~\pnCP, and the~\pSCP, respectively. 
\end{examplex}

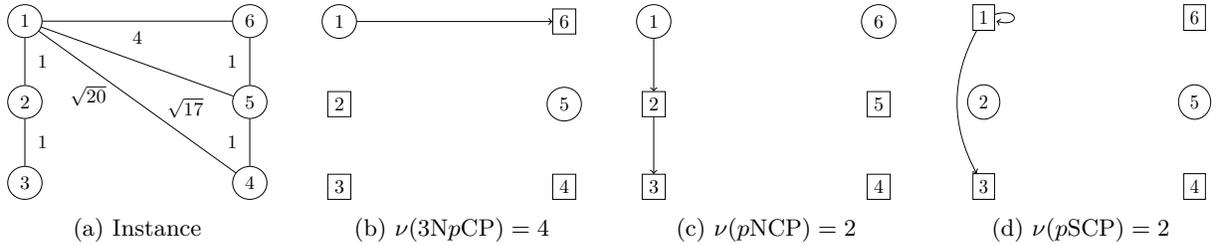
\begin{figure}[h]
	\centering
	\begin{subfigure}{0.22\textwidth}
		\centering
		\resizebox{\textwidth}{2.6cm}{
			\begin{tikzpicture}[every node/.style={circle}]
			\node (a) at (0,3) [draw] {1};
			\node (b) at (0,1.5) [draw] {2};
			\node (c) at (0,0) [draw] {3};
			\node (d) at (4,0) [draw] {4};
			\node (e) at (4,1.5) [draw] {5};
			\node (f) at (4,3) [draw] {6};
			
			\draw (a) -- node[right] {$1$} (b);
			\draw (b) -- node[right] {$1$} (c);
			\draw (d) -- node[left] {$1$} (e);
			\draw (e) -- node[left] {$1$} (f);
			
			\draw (a) -- node[below] {$4$} (f);
			\draw (a) -- node[below, near end] {$\sqrt{17}$} (e);
			\draw (a) -- node[below, near start] {$\sqrt{20}$} (d);
			
		\end{tikzpicture}
		}
		\caption{Instance}
		\label{fig_big:instance}
	\end{subfigure}
	\hspace*{0.019\textwidth}
	\begin{subfigure}{0.22\textwidth}
		\centering
		\resizebox{\textwidth}{2.6cm}{
			\begin{tikzpicture}
			\node (a) at (0,3) [draw,circle] {1};
			\node (b) at (0,1.5) [draw] {2};
			\node (c) at (0,0) [draw] {3};
			\node (d) at (4,0) [draw] {4};
			\node (e) at (4,1.5) [draw,circle] {5};
			\node (f) at (4,3) [draw] {6};
			
			\draw[->] (a) to (f);
		\end{tikzpicture}
		}
		\caption{$\nu(\text{3N$p$CP}) = 4$}
		\label{fig_big:3NpCP}
	\end{subfigure}
	\hspace*{0.019\textwidth}
	\begin{subfigure}{0.22\textwidth}
		\centering
		\resizebox{\textwidth}{2.6cm}{
		\begin{tikzpicture}		
			\node (a) at (0,3) [draw,circle] {1};
			\node (b) at (0,1.5) [draw] {2};
			\node (c) at (0,0) [draw] {3};
			\node (d) at (4,0) [draw] {4};
			\node (e) at (4,1.5) [draw] {5};
			\node (f) at (4,3) [draw,circle] {6};
			
			\draw[->] (a) to (b);
			\draw[->] (b) to (c);		
		\end{tikzpicture}
		}
		\caption{ $\nu(\text{\pnCP}) = 2$}
		\label{fig_big:pNCP}
	\end{subfigure}
	\hspace*{0.019\textwidth}
	\begin{subfigure}{0.22\textwidth}
		\centering
		\vspace{1em}
		\resizebox{\textwidth}{2.6cm}{
		\begin{tikzpicture}			
			\node (a) at (0,3) [draw] {1};
			\node (b) at (0,1.5) [draw,circle] {2};
			\node (c) at (0,0) [draw] {3};
			\node (d) at (4,0) [draw] {4};
			\node (e) at (4,1.5) [draw,circle] {5};
			\node (f) at (4,3) [draw] {6};
			
			\draw[->, bend right=30] (a) to (c);
			\path (a) edge [loop right] (a);		
		\end{tikzpicture}
		}
		\caption{$\nu(\text{\pSCP}) = 2$}
		\label{fig_big:pSCP}
	\end{subfigure}
	\caption{Illustration of Example~\ref{ex:3npcp_big} with depicting some distances 
	}
	\label{fig:3npcp_big}
\end{figure}

Examples~\ref{ex:3npcp_small} and~\ref{ex:3npcp_big} therefore show that neither $\nu(\text{3N$p$CP}) \leq \nu(\text{\pnCP})$ nor $\nu(\text{3N$p$CP}) \geq \nu(\text{\pnCP})$ and neither $\nu(\text{3N$p$CP}) \leq \nu(\text{\pSCP})$ nor $\nu(\text{3N$p$CP}) \geq \nu(\text{\pSCP})$ can hold in general. 
This observation, together with the statements in Theorem~\ref{thm:pcp_relations}, is illustrated in Figure~\ref{fig:pcp_variants}. 
There, arrows from one \pCP variant to another mean that for the same instance, the optimal objective function value of the one variant is always less or equal than the optimal objective function value of the other variant, if the requirement written on the arrow is met. Dashed lines between two variants show that there is no inequality on the optimal objective function value that holds in general.

\begin{figure}[h]
	\centering	
	\resizebox{0.8\textwidth}{6.5cm}{%
		\begin{tikzpicture}
			
			% Nodes
			\node (pCP) at (0,0) {\pCP};
			\node (2NpCP) at (0,2) {2N$p$CP};
			\node (3NpCP) at (0,4) {3N$p$CP};
			\node (aNpCP) at (0,5.5) {\apCP};
			
			\node (pNCP) at (-5.5,4) {\pnCP};
			\node (pSCP) at (5.5,4) {\pSCP};
			
			\node (paCP) at (5.5,7.5) {\paCP};
			
			\node (y0) at (-7,0) {};
			\node (y1) at (-7,8) {};
			
			% Arrows from pCP
			\draw[->] (pCP) -- (2NpCP) node[midway,right] {$d_{ii}=0$};
			\draw[->] (pCP) -- (pNCP) node[midway,sloped,above] {$d\ge 0$};
			\draw[->] (pCP) -- (pSCP) node[midway,sloped,above] {$d\ge 0$};
			
			% Arrows from 2NCP
			\draw[->] (2NpCP) -- (pNCP) node[midway,sloped,above] {$d:\Delta\text{-ineq.}$};
			\draw[->] (2NpCP) -- (3NpCP);
			\draw[->] (2NpCP) -- (pSCP) node[midway,sloped,above] {$d\ge 0$};
			
			% More arrows
			\draw[->] (3NpCP) -- (aNpCP);
			\draw[->] (aNpCP) -- (paCP) node[midway,sloped,above] {$d\ge 0$};
			\draw[->] (pSCP) -- (paCP);
			
			% Ligthning arrows
			\draw[dashed] (3NpCP) -- (pNCP) node[midway, fill=white, inner sep=1pt] {\text{\Lightning}};
			\draw[dashed] (3NpCP) -- (pSCP) node[midway, fill=white, inner sep=1pt] {\text{\Lightning}};

			\draw[->] (y0) -- (y1) node[near start, label={[align=center]left:optimal objective\\function value}] {};
		\end{tikzpicture}
	}
	\caption{Comparison of the optimal objective function values of the same instance for different variants of the~\pCP under the given assumptions}
	\label{fig:pcp_variants}
\end{figure}
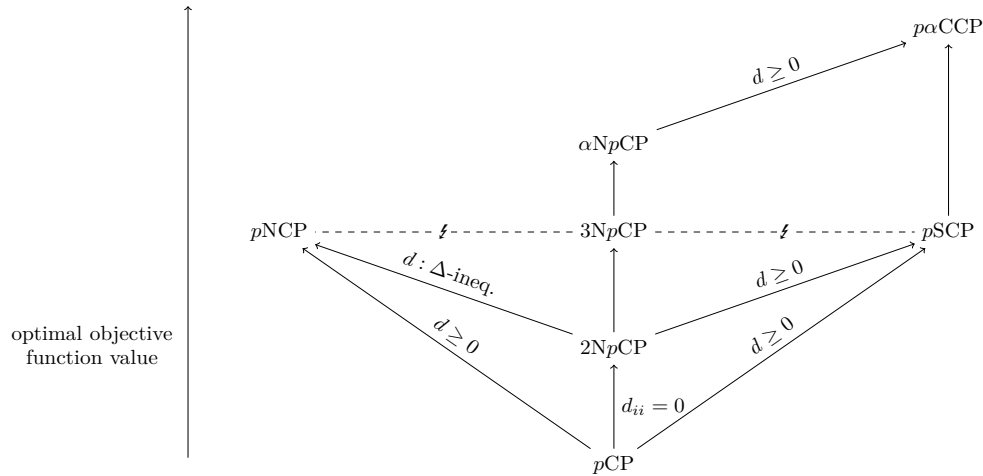

\section{Mixed-integer programming formulations of the \texorpdfstring{$p$}{p}-\texorpdfstring{$\alpha$}{a}-closest-center problem} 
\label{sec:MIPformulations}

In this section, we present several MIP formulations of the \paCP.
All of them can be viewed as generalizations of the classical textbook MIP formulation of the \pCP (see, e.g.,~\citet{daskin2013network}).   
In Section~\ref{sec:first_MIP} we present a formulation based on the same variables as the classical \pCP formulation 
as well as several optimality-preserving (in)equalities. A  
second MIP formulation based on slightly different assignment variables is considered in Section~\ref{sec:second_MIP}. 
In Section~\ref{sec:third_MIP}, we introduce two MIP formulations with assignment variables indexed by subsets of the set of potential facility locations together with valid and optimality-preserving inequalities.

\subsection{Single assignment formulation} \label{sec:first_MIP}
Our first MIP formulation of the \paCP  
uses the same binary decision variables as 
the classical textbook formulation for the \pCP. 
Therefore, let~$y_j = 1$ indicate that a facility~$j\in J$ is open and let~$y_j = 
0$ indicate that it is not. Additionally, for every customer~$i\in I$, we let 
$x_{ij}=1$ indicate that facility $j\in J$ is one of the~$\alpha$ closest open 
facilities to~$i$ and we let~$x_{ij} = 0$ mean the contrary.
Then, our first MIP formulation of the \paCP reads
\begin{subequations} \label{eq:PACA}
	\begin{alignat}{3}
		&\mathrm{\mytag{F1}} \qquad
		&\min \quad \obj \phantom{iiii} \label{pac1:z}  \\ 
		&&\st~ \sum_{j \in \loc}y_{j} & =     p   \label{pac1:sumy} \\
		&&\sum_{j \in \loc}x_{ij} & = \alpha && \forall i \in \cus  \label{pac1:sumx} \\
		&& x_{ij} & \leq  y_j  && \forall i\in \cus, \forall j\in \loc                   \label{pac1:xy} \\
		&& \sum_{j \in \loc}d_{ij}x_{ij} & \leq  z  && \forall i \in \cus                         \label{pac1:sumdx} \\
		&& x_{ij} & \in  \{0, 1\}\qquad && \forall i\in \cus, \forall j\in \loc                    \label{pac1:xbin}\\
		&& y_{j} & \in   \{0, 1\} && \forall j \in \loc                          \label{pac1:ybin}  \\
		&& z  & \in \mathbb{R}. \label{pac1:zbin}
	\end{alignat}
\end{subequations}

Constraints \eqref{pac1:sumy} and \eqref{pac1:ybin} ensure that we open 
exactly~$p$ facilities. 
Due to constraints \eqref{pac1:sumx}, each customer gets assigned 
to~$\alpha$ facilities and by constraints \eqref{pac1:xy}, such an 
assignment is only possible for open facilities.  
Moreover, because of constraints \eqref{pac1:sumdx}, the~$z$-variable must 
take at least the value of the maximum~$\alpha$-distance, i.e., the maximum 
sum of the  distances from a customer to its~$\alpha$~assigned open 
facilities, over all customers.
Since the objective function \eqref{pac1:z} is minimizing the~$z$-variable, 
any optimal solution of \myref{F1} is a feasible solution minimizing the 
maximum~$\alpha$-distance over all customers and therefore is an optimal 
solution of the \paCP.
Note, that \myref{F1} has the same number of variables and constraints as the 
classical \pCP~MIP formulation, so it has~$O(nm)$ variables, of which all but 
one are binary, and~$O(nm)$ constraints.

Next, we derive a set of \emph{optimality-preserving} inequalities, i.e., inequalities which may cut off some feasible solutions but that are satisfied by at least one optimal solution, for \myref{F1}. 
The idea is that given an upper bound $UB$ on~$\nu(\text{\paCP})$, any assignment 
from a customer~$i\in I$ to a facility~$j\in J$ with a corresponding 
distance $d_{ij} > UB$ can not be part of an optimal solution. 
This can be further generalized to sums of distances $d_{iA}$ that exceed $UB$ for 
some $i\in I$ and $A \in \mathcal{J}^\alpha$. In that case, the customer~$i$ 
can not be assigned to all of the facilities in $A$ 
in an optimal solution. 
The following theorem formalizes this idea.

\begin{theorem}\label{theo:pac1-var-fixing}
	Let~$UB$ be an upper bound on~$\nu(\text{\paCP})$ and let~$i~\in~I$, $C~\subseteq~J$ and~$\beta \in \mathbb{N}$ with $\beta \leq \alpha$. 
	If for all~$B\subseteq C$ with~$|B|=\beta$ we have that~$d_{iA} > UB$ holds for all~$A\in \mathcal{J}^\alpha$ with~$B\subseteq A$, then 
	\begin{equation}\label{pac1:optpres_all}
		\sum_{j\in C} x_{ij} \leq \beta -1	
	\end{equation}
	is an optimality-preserving inequality for~\myref{F1}, which we denote by \emph{general upper bound inequality}.
\end{theorem}

\begin{proof}
	Assume that for all $B\subseteq C$ with $|B|=\beta$ we have that $d_{iA} > UB$ holds for all $A\in \mathcal{J}^\alpha$ with $B\subseteq A$. 
	Furthermore, let~$(\bar{x},\bar{y},\bar{z})$ be an optimal solution of~\myref{F1} and let~$A_i=\{j\in J: \bar{x}_{ij}=1\}$, so clearly $A_i\in \mathcal{J}^\alpha$ holds.  Now assume that~\eqref{pac1:optpres_all} does not hold. 
	Then,~$\bar{x}_{ij} = 1$ holds for at least~$\beta$~many facilities~$j\in 
	C$, i.e., there is some~$B\subseteq C$ with~$|B|=\beta$ and~$B\subseteq A_i$. 
	As this implies $d_{iA_i} > UB$, we have $\bar{z} \geq \sum_{j\in J} d_{ij} 
	\bar{x}_{ij} = \sum_{j\in A_i} d_{ij} = d_{iA_i} > UB$. Thus, the objective 
	function value of $(\bar{x}, \bar{y}, \bar{z})$ is greater than the upper 
	bound~$UB$. As this contradicts the optimality of~$(\bar{x}, 
	\bar{y},\bar{z})$, we have shown that~\eqref{pac1:optpres_all} must hold. 
\end{proof}

Note, that the general upper bound inequalities~\eqref{pac1:optpres_all} 
are 
in fact satisfied for all optimal solutions of~\myref{F1}, i.e., adding all of 
those inequalities does not change the optimal objective function value 
of~\myref{F1}.  
Next, we derive some special cases of Theorem~\ref{theo:pac1-var-fixing} that are used in the implementation of our solution algorithm for the~\paCP that we present in Section~\ref{sec:implementation}. 

\begin{corollary}\label{cor:var_fix_d1_2}
	Let~$UB$ be an upper bound on~$\nu(\text{\paCP})$.
	Then, 
	\begin{equation}\label{pac1:optpres2}
		x_{ij} = 0 \qquad \forall i\in I, \forall j\in J: (d_{iA} > UB \ 
		\forall A\in \mathcal{J}^\alpha: j\in A) 
	\end{equation}
	are optimality-preserving equalities for~\myref{F1}, which we denote 
	by
	\emph{upper bound equalities}.
\end{corollary}

\begin{proof}
	As $UB$, $i$, $C=\{j\}$ and $\beta = 1$ satisfy the condition of Theorem~\ref{theo:pac1-var-fixing} for all~$i\in I$ and~$j\in J$ with~$d_{iA} > UB$ for all~$A\in \mathcal{J}^\alpha$ with~$j\in A$, inequalities~\eqref{pac1:optpres2} follow directly from~\eqref{pac1:optpres_all} and~\eqref{pac1:xbin}.
\end{proof}

\begin{corollary}\label{cor:var_fix_d1}
	Let~$UB$ be an upper bound on~$\nu(\text{\paCP})$. Then, 
	\begin{align*}
		x_{ij} = 0 \qquad \forall i\in I, \forall j\in J: d_{ij} > UB 
	\end{align*}
	are optimality-preserving equalities for~\myref{F1}.
\end{corollary}

We note that Corollary~\ref{cor:var_fix_d1} is a special case of Corollary~\ref{cor:var_fix_d1_2}, and thus follows directly from the latter.

\begin{corollary}\label{cor:var_fix_d1_3}
	Let~$UB$ be an upper bound on~$\nu(\text{\paCP})$. 
	For a customer~$i\in I$ and some $\beta \in \{1, \ldots, \alpha\}$ define 
	$C_i^\beta = \left\{j\in J: d_{ij} > \frac{UB}{\beta}\right\}$. 
	Then,
	\begin{align}
		\sum_{j \in C_i^\beta} x_{ij} \leq \beta-1 \qquad \forall i\in I, 
		\forall \beta\in \{1, \ldots, \alpha\} \label{pac1:optpres3} 
	\end{align}
	are optimality-preserving inequalities for~\myref{F1}, which we 
	denote 
	by \emph{simple upper bound inequalities}.
\end{corollary}

\begin{proof}
	It is easy to see that~$UB$ and $C=C_i^\beta$ satisfy the condition of Theorem~\ref{theo:pac1-var-fixing} for all $i\in I$ and $\beta \leq \alpha$. Inequalities~\eqref{pac1:optpres3} then follow from~\eqref{pac1:optpres_all}.
\end{proof}

We close this section by proposing additional optimality-preserving equalities for which we do not need to know an upper bound on $\nu(\text{\paCP})$. 

\begin{lemma}\label{lem:var_fix_d1_4}
	For a customer~$i\in I$, let $J_i^{k}$ be a set of~$k$~facilities 
	that have the highest distances to $i$, i.e., $J_i^{k} \in \arg 
	\max_{\substack{Q\subseteq J\\|Q|=k}} \{\sum_{j\in Q} d_{ij}\}$. Then 
	\begin{align}
		x_{ij} = 0 \qquad \forall i\in I, \forall j\in  J_i^{p-\alpha} 
		\label{pac1:optpres_farthest} 
	\end{align}
	are optimality-preserving equalities for~\myref{F1}, which we 
		denote 
		by \emph{remoteness equalities}.
\end{lemma}
\begin{proof}
	This follows from the fact that there is an optimal solution in which customer~$i$ is 
	not assigned to any facility in $J_i^{p-\alpha}$, because even if the~$p$ 
	facilities that are open are the $p$ furthest away from $i$, this~$i$ can be assigned 
	to the $\alpha$ closest of them.
\end{proof}

Note, that in case of ties in the distances,~a remoteness 
equality~\eqref{pac1:optpres_farthest} is only allowed to be added for at most 
one of all possible sets $J_i^{p-\alpha}$ to prevent the loss of optimal 
solutions.

\subsection{Separate assignment formulation} \label{sec:second_MIP}

In~\myref{F1}, the information of  
which assigned open facility is the~$\beta$-closest to a customer for 
each~$\beta\leq \alpha$, is not directly modeled in a solution. 
Therefore, we introduce a second MIP formulation
of the \paCP  
that also uses the distance measure~$z$ and the $y$-variables indicating 
which facilities are open, 
but~$\alpha$ different variables~$\myx{1}, \ldots, \myx{\alpha}$ modeling the 
assignment of the customers to their~$\alpha$ closest open facilities 
separately.  
To that end, let~$\myx{\beta}_{ij} = 1$ if~$j$ is the~$\beta$-closest open 
facility to a customer~$i$ and~$\myx{\beta}_{ij} = 0$ otherwise for all~$i\in 
I$,~$j\in J$ and~$\beta\in\{1, \ldots, \alpha\}$. 
Then, the \paCP can be modeled as
\begin{subequations}
	\begin{alignat}{3}
		&\mathrm{\mytag{F2}} \qquad
		& \min \quad \obj \phantom{iiii} \label{pac2:z}  \\ 
		&& \st~  \sum_{j \in \loc} y_j &= p \label{pac2:sumy} \\       
		&& \sum_{j \in \loc} \myx{\beta}_{ij} & =  1 && \forall i \in \cus, \forall \beta\in \{1,\ldots, \alpha\}\label{pac2:sumx} \\
		&& \sum_{\beta = 1}^\alpha \myx{\beta}_{ij} &\leq y_j && \forall i \in \cus, \forall j \in \loc\label{pac2:xy}\\   
		&& \sum_{j\in J} d_{ij} \myx{\beta}_{ij} &\leq \sum_{j\in J} d_{ij} \myx{\beta+1}_{ij} \qquad && \forall i\in \cus, \forall \beta \in\{1, \ldots, \alpha -1\}   \label{pac2:comparex}\\
		&& \sum_{j \in \loc} d_{ij} \left(\sum_{\beta =1}^\alpha \myx{\beta}_{ij} \right) & \leq \obj && \forall i \in \cus 
		\label{pac2:sumdx}\\
		&& \myx{\beta}_{ij} &\in  \{0,1\} \qquad&& \forall i \in \cus, \forall j \in \loc, \forall \beta\in \{1,\ldots, \alpha\}  
		\label{pac2:xbin}\\
		&& y_{j} &\in  \{0,1\} && \forall j \in \loc  \label{pac2:ybin}\\
		&& \obj & \in \mathbb{R}.
	\end{alignat}
\end{subequations}

As in~\myref{F1}, constraint~\eqref{pac2:sumy} enforces exactly~$p$ open 
facilities. Each customer is assigned to exactly one~$\beta$-closest facility 
for all~$\beta\in \{1,\ldots, \alpha\}$ by~\eqref{pac2:sumx}, and~\eqref{pac2:xy} 
make sure that those facilities are open and different from each other.
The correct assignment with respect to the distance of a customer~$i\in 
\cus$, e.g., that $j\in J$ is the closest open facility to $i$ is indicated 
by~$\myx{1}_{ij}$, is ensured by~\eqref{pac2:comparex}. 
Constraints~\eqref{pac2:sumdx} again force~$z$ to be the maximum~$\alpha$-distance 
over all customers, which is minimized by~\eqref{pac2:z}. 
Constraints~\eqref{pac2:xbin} and~\eqref{pac2:ybin} make sure that the variables~$x$ and~$y$ are binary. 

Formulation~\myref{F2} has~$O(\alpha nm)$ variables, of which all but one are binary, and~$O(nm)$ constraints. For a  fixed value of~$\alpha$, the sizes of~\myref{F1} and~\myref{F2} are therefore of the same order.  
Note, that it is straightforward to adapt the optimality-preserving inequalities discussed in Section~\ref{sec:first_MIP} to~\myref{F2}. For the sake of brevity, we do not give the details here. 

\subsection{Subset assignment formulation} \label{sec:third_MIP}

Our third MIP formulation of the \paCP is based on subsets. 
We use assignment variables~$x_{iA}$  for all~$i\in I$ and~$A\in 
\mathcal{J}^\alpha$ where~$x_{iA}=1$ if~$A$ is the subset of $J$ containing 
the~$\alpha$ closest open facilities to a customer~$i$ and~$x_{iA}=0$ 
otherwise.
The \paCP can then be formulated as
\begin{subequations}
	\begin{alignat}{3}
		&\mathrm{\mytag{F3}} \qquad
		& \min \quad \obj \phantom{iiii} \label{pac3:z}  \\ 
		&& \st~   \sum_{j \in \loc} y_j &= p \label{pac3:sumy} \\       
		&& \sum_{A\in \mathcal{J}^\alpha} x_{iA} & =  1 && \forall i \in \cus\label{pac3:sumx} \\
		&& x_{iA} &\leq y_j && \forall i \in \cus,\forall A\in \mathcal{J}^\alpha,\forall j \in A\label{pac3:xy}\\   
		&& \sum_{A\in \mathcal{J}^\alpha} d_{iA} x_{iA} & \leq \obj && \forall i \in \cus 
		\label{pac3:sumdx}\\
		&& x_{iA} &\in  \{0,1\} \qquad&& \forall i \in \cus, \forall A\in \mathcal{J}^\alpha 
		\label{pac3:xbin}\\
		&& y_{j} &\in  \{0,1\} && \forall j \in \loc  \label{pac3:ybin}\\
		&& \obj & \in \mathbb{R}.
	\end{alignat}
\end{subequations}

Again, the number of open facilities is set by~\eqref{pac3:sumy}. 
Constraints~\eqref{pac3:sumx} ensure that each customer is assigned to 
exactly one subset consisting of~$\alpha$ different facilities, which all need to 
be open by~\eqref{pac3:xy}.
The objective function value~$z$ is set to the maximum $\alpha$-distance over all 
customers by~\eqref{pac3:z} and~\eqref{pac3:sumdx}. 
Constraints~\eqref{pac3:xbin} and~\eqref{pac3:ybin} enforce~$x$ and~$y$ to be binary.
Note, that for the case~$I=J$, formulation~\myref{F3} can be viewed as a hypergraph formulation for a graph~$G = (I, \mathcal{J}^\alpha)$, i.e., the set~$I$ represents the nodes of the hypergraph and each edge consists of~$\alpha$ such nodes. 
In terms of the number of variables and constraints, formulation~\myref{F3} is usually the largest as it uses~$O(nm^\alpha)$ variables, of which all but one are binary, and~$O(\alpha nm^\alpha)$ constraints.

Next, we introduce valid inequalities for \myref{F3} that reduce the number of constraints by aggregating inequalities~\eqref{pac3:xy}. They will become important for the polyhedral study that we conduct in Section~\ref{sec:polyhedral_study}.  

\begin{lemma}\label{lem:d3_valid_1}
	The inequalities
	\begin{align}
		\sum_{\substack{A\in \mathcal{J}^\alpha\\j\in A}} x_{iA} \leq y_j
		\qquad \forall i\in I, \forall j\in J \label{pac3:valid_1}
	\end{align}
	are valid inequalities for~\myref{F3}.
\end{lemma}

\begin{proof}
	Let~$(\tilde{x},\tilde{y},\tilde{z})$ be a feasible solution of~\myref{F3} and let~$i\in \cus$ and~$j\in \loc$.  
	We have~$\tilde{y}_{j}\in \{0,1\}$ by~\eqref{pac3:ybin}.
	If~$\tilde{y}_{j} = 1$, then 
	$
	\sum_{{A\in \mathcal{J}^\alpha:j \in A}} \tilde{x}_{i A} \leq \sum_{A\in \mathcal{J}^\alpha} \tilde{x}_{i A} = 1 =  \tilde{y}_{j}
	$
	follows from~\eqref{pac3:sumx}. 
	If~$\tilde{y}_{j} = 0$, then~$\sum_{A\in \mathcal{J}^\alpha:\\j\in A} \tilde{x}_{iA} = 0 \leq \tilde{y}_j$ follows from \eqref{pac3:xy} and \eqref{pac3:xbin}.
\end{proof}

Note, that the inequalities~\eqref{pac3:valid_1} supersede constraints~\eqref{pac3:xy}. 
We therefore let \mytag{F3-V} be formulation~\myref{F3} with~\eqref{pac3:valid_1} but without \eqref{pac3:xy}. Then, \myref{F3-V} is our fourth MIP formulation of the \paCP with~$O(nm^\alpha)$~variables, of which all but one are binary, and~$O(nm)$~constraints. 

We close this section by giving optimality-preserving inequalities for~\myref{F3} and~\myref{F3-V}
which are based on the idea that each customer~$i\in I$ should only be 
assigned to a subset $A \in \mathcal{J}^\alpha$ if there is no open facility that 
is closer to $i$ then the facilities in~$A$. These inequalities can 
be viewed as an extension of the well-known \emph{closest-assignment constraints} 
from facility location literature, see, e.g., \citet{espejo2012closest} for more 
details.

\begin{theorem} 
	The inequalities
	\begin{equation}
		y_j + \sum_{\substack{A\in \mathcal{J}^\alpha\\ j\not\in A\\ d_{ij}< 
		d_{iA}^{\max} 
		}} x_{iA} \leq 1 \qquad \forall i\in I, \forall j \in J   
		\label{pac3:opt_ineq} 
	\end{equation}
	are optimality-preserving inequalities for~\myref{F3} and~\myref{F3-V}. 
\end{theorem}

\begin{proof}
	Let~$(\tilde{x},\tilde{y},\tilde{z})$ be an optimal solution of~\myref{F3} and let~$i\in I$ and~$j\in J$.  
	If~$\tilde{y}_j = 0$, then~\eqref{pac3:opt_ineq} follows from~\eqref{pac3:sumx}. 
	Therefore, let~$\tilde{y}_j = 1$ and assume~\eqref{pac3:opt_ineq} does not hold, i.e., assume there exists 
	some~$A' \in \mathcal{J}^\alpha$ with~$j\not \in A'$, $d_{ij} < d_{iA'}^{\max}$ and $\tilde{x}_{iA'} = 1$. Let $j'\in A'$ such that~$d_{ij'} = d_{iA'}^{\max}$. 
	Define~$A^* = \left(~A'~\setminus~\{j'\}~\right)~\cup~\{j\}~\in~\mathcal{J}^\alpha$ and set $x^*_{iA^*} = 1$ and $x^*_{iA} = 0$ for all $A\in \mathcal{J}^\alpha \setminus \{A^*\}$.
	Then, $d_{iA^*} < d_{iA'}$ holds because $d_{ij} < d_{ij'}$. Thus $\sum_{A\in \mathcal{J}^\alpha} d_{iA} x^*_{iA} = d_{iA^*} < d_{iA'} = \sum_{A\in \mathcal{J}^\alpha} d_{iA} \tilde{x}_{iA}  \leq \tilde{z}$.  
	It is therefore easy to see that~$(x^*, \tilde{y}, \tilde{z})$ is a feasible solution of \myref{F3} with the same objective function value as~$(\tilde{x},\tilde{y},\tilde{z})$ and satisfies inequalities~\eqref{pac3:opt_ineq}.  
	Hence, there is at least one optimal solution of~\myref{F3} satisfying~\eqref{pac3:opt_ineq}. The proof for~\myref{F3-V} can be done analogously.
\end{proof}

This closes the discussion of our four different MIP formulations of the~\paCP. 

\section{Polyhedral study}\label{sec:polyhedral_study}

In this section, we compare the previously introduced MIP formulations \myref{F1}, \myref{F2},~\myref{F3} and~\myref{F3-V} of the \paCP with respect to their LP-relaxations. 
Note that in \myref{F1}, the relaxed conditions~$x_{ij} \in [0,1]$ and~$y_j \in [0,1]$ for all~$i\in I$ and $j\in J$ are equivalent to
\begin{subequations}
\begin{align}
	x_{ij} &\geq 0  \qquad \forall i \in \cus, \forall j\in \loc \label{pac1r:xreal}\\
	y_j &\leq 1  \qquad \forall j \in \loc  \label{pac1r:yreal}
\end{align}
\end{subequations}
by constraints~\eqref{pac1:xy}.
Therefore, formulation~\mytag{F1-R} is \myref{F1} without~\eqref{pac1:xbin} and~\eqref{pac1:ybin} and with~\eqref{pac1r:xreal} and~\eqref{pac1r:yreal}.
An analogous statement is true for the MIP formulations~\myref{F2},~\myref{F3}, and \myref{F3-V} as well, with the LP-relaxations denoted by \mytag{F2-R},  \mytag{F3-R}, and \mytag{F3-V-R}, respectively. 
In particular,~\myref{F3-R} is~\myref{F3} without~\eqref{pac3:xbin} and~\eqref{pac3:ybin} and with 
\begin{subequations}
	\begin{align}
		&& x_{iA} &\geq 0 && \forall i \in \cus, \forall A\in \mathcal{J}^\alpha \label{pac3r:xreal}\\
		&& y_j &\leq 1 && \forall j \in \loc.  \label{pac3r:yreal}
	\end{align}
\end{subequations}

The following lemma shows that the LP-relaxations of~\myref{F1} and \myref{F2} are equally strong.

\begin{lemma}\label{lem:d1=d2}
	$\nu \myref{F1-R} = \nu \myref{F2-R}$.
\end{lemma}

\begin{proof}
	We start by showing~$\nu \myref{F1-R} \geq \nu \myref{F2-R}$. To do so, let~$(\bar{x}, \bar{y}, \bar{z})$ be an optimal solution of \myref{F1-R}, so~$\bar{z} = \nu \myref{F1-R}$. 
	We define~$\hat{y} = \bar{y}$,~$\hat{z} = \bar{z}$ and~$\hat{x}^{(\beta)} = \frac{1}{\alpha}\bar{x}$ for all~$\beta\in \{1, \ldots, \alpha\}$. 
	It is easy to see that~$(\hat{x},\hat{y},\hat{z})$ is a feasible solution of \myref{F2-R} with the objective function value~$\hat{z}$. 
	Hence,~$\nu \myref{F2-R} \leq \hat{z} = \bar{z} = \nu \myref{F1-R}$ holds.  
	
	Next, we show~$\nu \myref{F1-R} \leq \nu \myref{F2-R}$. For this purpose, let~$(\hat{x},\hat{y},\hat{z})$ be an optimal solution of \myref{F2-R}, so~$\hat{z} = \nu \myref{F2-R}$.	
	Define~$\bar{y}= \hat{y}$,~$ \bar{z}= \hat{z}$ and~$\bar{x} = \sum_{\beta = 1}^\alpha \hat{x}^{(\beta)}$. It is easy to see that~$(\bar{x}, \bar{y}, \bar{z})$ is a feasible solution of \myref{F1-R} with objective function value~$\bar{z}$.
	Thus,~$\nu \myref{F1-R} \leq \bar{z} = \hat{z} =  \nu \myref{F2-R}$ holds. 
\end{proof}

Unfortunately, such equality does not hold for~\myref{F1} and~\myref{F3}, 
as the following example shows.
\begin{examplex}\label{ex:d1vsd3}
	Let~$I = J = \{1, 2,3\}$,~$p=\alpha=2$, and~$d_{1,1} = d_{2,2} = d_{3,3}=0, 
	d_{1,2}=d_{2,1}=d_{2,3}=d_{3,2}=1$ and $d_{1,3}=d_{3,1}=2$ as depicted in 
	Figure~\ref{fig:ex_graph}. 
	One can verify that~$\bar{y}=\left(\bar{y}_1, \bar{y}_2, \bar{y}_3\right)=\left(\frac{2}{3}, \frac{2}{3}, \frac{2}{3}\right)$ and~$\bar{x}_{ij} = \frac{2}{3}$ for all~$i,j\in \{1, 2,3\}$ is an optimal solution of \myref{F1-R} with objective function value~$\bar{z} = 2$.
	On the other hand,~$\left(\tilde{x}, \tilde{y}, \tilde{z}\right)$ with~$\tilde{y}~=\left(\tilde{y}_1, \tilde{y}_2, \tilde{y}_3\right)=\left(\frac{2}{3}, \frac{2}{3}, \frac{2}{3}\right)$ and 
	$\tilde{x}_1 = \left(\tilde{x}_{1,\{1,2\}},\tilde{x}_{1,\{1,3\}},\tilde{x}_{1,\{2,3\}}\right) = \left(\frac{2}{3}, \frac{1}{3}, 0\right)$, 
	$\tilde{x}_2 = \left(\tilde{x}_{2,\{1,2\}},\tilde{x}_{2,\{1,3\}},\tilde{x}_{2,\{2,3\}}\right) = \left(\frac{1}{3}, \frac{2}{3}, 0\right)$ and
	$\tilde{x}_3 = \left(\tilde{x}_{3,\{1,2\}},\tilde{x}_{3,\{1,3\}},\tilde{x}_{3,\{2,3\}}\right) = \left(0, \frac{1}{3}, \frac{2}{3}\right)$ 
	is a feasible solution of \myref{F3-R} with objective function value~$\tilde{z}=\frac{5}{3} < 2$.
	So for this instance $\nu \myref{F3-R} <  \nu \myref{F1-R}$ holds.
\end{examplex}

\begin{figure}[h!]
	\centering
	\begin{tikzpicture}
		% Facilities = Customers
		\node (A) at (0, 0) [circle,draw] {1};
		\node (B) at (2, 0) [circle,draw] {2};
		\node (C) at (4, 0) [circle,draw] {3};
		
		% Distances
		\draw[-] (A) to node[below] {1} (B);
		\draw[-] (B) to node[below] {1} (C);
		\draw[-] (A) to [bend left=20] node[above] {2} (C);		
	\end{tikzpicture}
	\caption{Illustration of Example~\ref{ex:d1vsd3} 
		}
	\label{fig:ex_graph}
\end{figure}
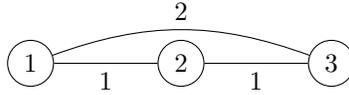

Example~\ref{ex:d1vsd3} shows that the LP-relaxation of~\myref{F3} cannot be as strong as the LP-relaxation of~\myref{F1}, since there is at least one instance for which~\myref{F3-R} yields a strictly smaller lower bound. 
One reason for this lies in the weak restriction of the assignment variables 
in~\myref{F3}. Consider a fixed customer~$i\in I$ and a fixed facility 
location~$j \in J$. Then there are~$\binom{m-1}{\alpha-1}$ many subsets 
in~$\mathcal{J}^\alpha$ containing~$j$. For each of these subsets, the 
corresponding assignment variable is bounded by~$y_j$. However, the sum of these 
subset assignment variables can be greater than~$y_j$ as it is the case in 
Example~\ref{ex:d1vsd3}. 
Indeed, for~$i=1$ and~$j=1$, we have~$\sum_{A\subseteq \{1,2,3\}: |A|=2, 1\in A} \tilde{x}_{1,A} = \tilde{x}_{1,\{1,2\}} + \tilde{x}_{1,\{1,3\}} = \frac{2}{3} + \frac{1}{3} = 1 > \frac{2}{3} = \tilde{y}_1$.
Adding the valid inequalities~\eqref{pac3:valid_1} to~\myref{F3} as it is done to obtain~\myref{F3-V}, however, cuts off the solution~$(\tilde{x},\tilde{y},\tilde{z})$, as the next example shows. 

\begin{examplex}\label{ex:d1vsd3v}
	Consider the instance of Example~\ref{ex:d1vsd3}.
	An optimal solution of \myref{F3-V-R} is 
	$\left(x^*, y^*, z^*\right)$ with~$y^*=\left(y^*_1, y^*_2, y^*_3\right)=\left(\frac{2}{3}, \frac{2}{3}, \frac{2}{3}\right)$,  
	$x^*_{i,A} = \frac{1}{3}$ for all~$i\in \{1,2,3\}$ and~$A\in \mathcal{J}^2 = \left\{\{1,2\}, \{1,3\}, \{2,3\}\right\}$
	and objective function value~$z^* = 2$.
\end{examplex}

In fact, for the instance in Example~\ref{ex:d1vsd3v} the optimal objective function value of \myref{F3-V-R} is equal to the one of~\myref{F1-R}. 
The following two lemmata show that this is no coincidence.   

\begin{lemma}\label{lem:d1<=d3v}
	$\nu\myref{F1-R} \leq \nu\myref{F3-V-R}$.
\end{lemma}

\begin{proof}
	Let~$(\tilde{x},\tilde{y},\tilde{z})$ be an optimal solution of \myref{F3-V-R}, so~$\tilde{z}=\nu \myref{F3-V-R}$.
	Define~$\bar{y}= \tilde{y}$,~$ \bar{z}= \tilde{z}$ and 
	\[
	\bar{x}_{ij} = \sum_{\substack{A\in \mathcal{J}^\alpha\\ j\in A}} \tilde{x}_{iA} 
	\]
	for all~$i\in \cus$ and~$j\in \loc$.
	Then, clearly \eqref{pac1:sumy}, \eqref{pac1:xy}, \eqref{pac1r:xreal}, \eqref{pac1r:yreal} and \eqref{pac1:zbin} are satisfied by $(\bar{x}, \bar{y}, \bar{z})$. 
	We also have
	\[
	\sum_{j\in \loc} \bar{x}_{ij} = \sum_{j\in \loc} \sum_{\substack{A\in \mathcal{J}^\alpha\\ j\in A}} \tilde{x}_{iA} = \alpha \sum_{A\in \mathcal{J}^\alpha} \tilde{x}_{iA} = \alpha
	\]
	for every~$i\in \cus$ since every~$A\in \mathcal{J}^\alpha$ contains exactly~$\alpha$ facilities~$j\in \loc$ and so, each~$\tilde{x}_{iA}$ is counted~$\alpha$ times. This shows \eqref{pac1:sumx}.
	Furthermore, we have
	\begin{align*}
		\sum_{j\in \loc} d_{ij} \bar{x}_{ij} &= \sum_{j\in \loc} d_{ij} \left( \sum_{\substack{A\in \mathcal{J}^\alpha\\j\in A}} \tilde{x}_{iA} \right) 
		= \sum_{j\in \loc} \sum_{\substack{A\in \mathcal{J}^\alpha\\j\in A}} d_{ij} \tilde{x}_{iA} 
		= \sum_{A\in \mathcal{J}^\alpha} \left( \sum_{j\in A} d_{ij} \right) \tilde{x}_{iA} 
		= \sum_{A\in \mathcal{J}^\alpha} d_{iA}  \tilde{x}_{iA} \leq \tilde{z} = \bar{z}
	\end{align*}
	for all~$i\in I$ by using~\eqref{pac3:sumdx}.
	Therefore, \eqref{pac1:sumdx} are satisfied and 
	$(\bar{x}, \bar{y}, \bar{z})$ is a feasible solution of \myref{F1-R} with objective function value~$\bar{z} = \tilde{z}$ which proves~$\nu \myref{F1-R} \leq \nu \myref{F3-V-R}$.
\end{proof}

\begin{lemma}\label{lem:d1>=d3-v}
	$\nu \myref{F1-R} \geq \nu \myref{F3-V-R}$.
\end{lemma}

\begin{proof}
	Let~$(\bar{x}, \bar{y}, \bar{z})$ be an optimal solution of~\myref{F1-R}. 
	The goal is to construct a feasible solution~$(\tilde{x}, \tilde{y}, \tilde{z})$ of~\myref{F3-V-R} with the same objective function value, i.e., with~$\tilde{z} = \bar{z}$. 
	We start by fixing a customer~$i\in I$ and defining the polytope 
	\begin{subequations}
	\begin{alignat}{3}
		P_i = \Big\{ x_i\in \mathbb{R}^{\binom{m}{\alpha}}: && \sum_{A\in \mathcal{J}^\alpha} x_{iA} &= 1, \label{poly1}\\
		&& \sum_{\substack{A\in \mathcal{J}^\alpha\\ j\in A}} x_{iA} &\leq \bar{x}_{ij} \quad \forall j\in J, \label{poly2}\\
		&& x_i &\geq 0 & \Big\}, \label{poly3}
	\end{alignat}
	\end{subequations}
	where~$x_{i} = (x_{iA})_{A\in \mathcal{J}^\alpha}$.
	We now show that~$P_i$ is non-empty. 
	For this, consider the maximization problem~$(P^i)$~=~$\max \{0: x_i\in P_i\}$ and its dual problem
	\begin{subequations}
		\begin{alignat}{3}
			&(D^i) \qquad
			&\min \phantom{i} \quad u + \sum_{j\in J}&\ \bar{x}_{ij} v_j  \label{dual:obj}\\
			&&\st~ u + \sum_{j\in A} v_j &\geq 0 \quad \forall A\in \mathcal{J}^\alpha \label{dual:constr}\\
			&& v_j &\geq 0 \quad \forall j\in J \label{dual:vbin}\\
			&& u &\in \mathbb{R}. \label{dual:u}
		\end{alignat}
	\end{subequations}
	Since~$(0,0) \in \mathbb{R}^{1+m}$ satisfies \eqref{dual:constr} to \eqref{dual:u}, the feasible region of~($D^i$) is non-empty. We show that the objective function of~($D^i$) is bounded from below. Let~$(u',v') \in \mathbb{R}^{1+m}$ be a feasible solution of~($D^i$). Let~$A' \in \arg\min \left\{\sum_{j\in A} v'_j: A\in \mathcal{J}^\alpha \right\}$ and let~$x^\prime_{i} = (x'_{ij})_{j\in J}$ with $x^\prime_{ij} = 1$ for all~$j\in A'$ and~$x^\prime_{ij} = 0$ for all $j\in J\setminus A'$. 
	Hence, $\sum_{j\in J} x^\prime_{ij} = \alpha$ and
	therefore $\sum_{j\in A'} v'_j = \sum_{j\in J} x^\prime_{ij} v'_j \leq \sum_{j\in J} \bar{x}_{ij} v'_j$ holds, because on the left and right hand-side of the inequality the weight~$\alpha$ is distributed to the smallest~$v'_j$ and any~$v'_j$, respectively.
	As a consequence, $0 \leq u' + \sum_{j\in A^\prime} v'_j \leq u' + \sum_{j\in J} \bar{x}_{ij}v'_j$ holds because of~\eqref{dual:constr} 
	and thus, the objective function value of $(u',v')$ for~($D^i$) is greater or equal to 0.
	Hence, the objective function of~($D^i$) is indeed bounded from below by 0 and thus ($D^i$) has an optimal solution, which guarantees the existence of a feasible solution of~($P^i$) by strong duality. This proves that~$P_i$ is non-empty.  	
	
	Because~$i$ was arbitrary, the polytope~$P_i$ is non-empty for all~$i\in I$. Therefore, 
	we can construct a feasible solution~$(\tilde{x}, \tilde{y}, \tilde{y})$ of~\myref{F3-V-R} as follows. Let~$\tilde{x} = (\tilde{x}_i)_{i\in I}$ with~$\tilde{x}_i \in P_i$ for all~$i\in I$,~$\tilde{y} = \bar{y}$ and~$\tilde{z} = \bar{z}$. It is easy to see that \eqref{pac3:sumy}, \eqref{pac3:sumx}, \eqref{pac3r:xreal} and \eqref{pac3r:yreal} are satisfied. 
	Additionally,~\eqref{pac3:valid_1} hold since~$\sum_{\substack{A\in \mathcal{J}^\alpha\\ j\in A}} \tilde{x}_{iA} \leq \bar{x}_{ij}$ and~$\bar{x}_{ij} \leq \bar{y}_j = \tilde{y}_j$ for all~$i \in I$ and~$j\in J$ by~\eqref{poly2} and~\eqref{pac1:xy}, respectively.
	Constraints~\eqref{pac3:sumdx} follow from~\eqref{poly2} and~\eqref{pac1:sumdx} as
	\[
	\sum_{A\in \mathcal{J}^\alpha} d_{iA} \tilde{x}_{iA} = \sum_{j\in J} d_{ij} \left( \sum_{\substack{A\in \mathcal{J}^\alpha\\ j\in A}} \tilde{x}_{iA} \right) \leq \sum_{j\in J} d_{ij} \bar{x}_{ij} \leq \bar{z}=\tilde{z}\]
	for all~$i\in I$.
	Since~$\tilde{z}=\bar{z}=\nu\myref{F1-R}$, we have~$\nu\myref{F3-V-R} \leq \nu\myref{F1-R}$, which finishes the proof.
\end{proof}

The following proposition summarizes the proven relationships between our LP-relaxations. 

\begin{proposition}\label{prop:d1-R=d3-v-R}
	$\nu \myref{F1-R} = \nu \myref{F2-R} = \nu \myref{F3-V-R} \geq \nu \myref{F3-R}$.
\end{proposition}

\begin{proof}
	Since the feasible region of~\myref{F3-V-R} is a subset of the feasible region of~\myref{F3-R}, we get~$\nu \myref{F3-V-R} \geq \nu \myref{F3-R}$. The remaining equalities follow from the Lemmata~\ref{lem:d1=d2},~\ref{lem:d1<=d3v} and~\ref{lem:d1>=d3-v}.
\end{proof}

Note, that in Example~\ref{ex:d1vsd3} we provided an instance for which $\nu\myref{F1-R} > \nu\myref{F3-R}$ holds.
Therefore, 
adding the valid inequalities~\eqref{pac3:valid_1} to~\myref{F3} as it is done to obtain~\myref{F3-V} in general really strengthens its LP-relaxation.  

We conclude this section by studying the semi-relaxations in which only the~$x$-variables are relaxed, whereas the~$y$-variables remain binary. For any MIP formulation~$(F)\in \{\myref{F1}, \myref{F2}, \myref{F3},\myref{F3-V}\}$ we denote this semi-relaxation by ($F$-Rx). For the \pCP, \citet{gaar2022scaleable} have shown that the bound obtained by solving the corresponding semi-relaxation of the classical MIP formulation 
already yields the optimal objective function value of the \pCP. 
The following proposition extends this result to our MIP formulations of the~\paCP.

\begin{proposition}\label{prop:semi_relax}
	For each~$(F) \in \{\myref{F1}, \myref{F2}, \myref{F3}, \myref{F3-V}\}$, we have
	$\nu(F$-Rx$) = \nu (F)$.
\end{proposition}

\begin{proof}
	The feasible region of ($F$) is a subset of the feasible region of ($F$-Rx), so we have~$\nu(F$-Rx$) \leq \nu (F)$ for all~$(F) \in \{\myref{F1}, \myref{F2}, \myref{F3},\myref{F3-V}\}$, as
	the \paCP is a minimization problem.
	
	Next, we show~$\nu \myref{F1} \leq \nu(\text{F1-Rx})$. 
	Let~$(x^*,y^*,z^*)$ be an optimal solution of (F1-Rx), so~$\nu(\text{F1-Rx})= z^*$. For all~$i\in \cus$, define~ the support $\supp(x_i^*) = \{j\in \loc : x^*_{ij} > 0\}$. 
	Note, that~$|\supp(x_i^*)| \geq \alpha$ by~\eqref{pac1:sumy}. Hence, we can choose~$A_i \in \arg \min \{d_{iA}: A\in \mathcal{J}^\alpha \text{ with } A \subseteq \supp(x_i^*) \}$.
	Then, set~$\bar{x}_{ij} = 1$ for all~$j\in A_i$ and~$\bar{x}_{ij}=0$ otherwise.
	Together with~$\bar{y}= y^*$ and~$\bar{z}=z^*$, it is easy to see that~$(\bar{x}, \bar{y}, \bar{z})$ is a feasible solution of \myref{F1} with objective function value~$\bar{z} = z^*$. 
	Hence,~$\nu \myref{F1} \leq \nu (\text{F1-Rx})$.
	
	By analogous arguments, $\nu(F) \leq \nu (F\text{-Rx})$ holds for all~$(F)\in \{\myref{F2}, \myref{F3}, \myref{F3-V}\}\}$ as well. 
\end{proof}

Proposition~\ref{prop:semi_relax} thus shows that relaxing the $x$-variables in the MIP formulations~\myref{F1},~\myref{F2},~\myref{F3} and~\myref{F3-V} does not change their optimal objective function values.

\section{Lifting of formulation (F3)}\label{sec:lifting_D3}

In this section, we extend the iterative lifting approach of 
\citet{gaar2022scaleable} for the \pCP to the \paCP. In particular, we present a 
set of valid inequalities 
for the MIP formulations~\myref{F3} and~\myref{F3-V} that utilizes lower bounds on $\nu(\text{\paCP})$,
and propose two iterative procedures to improve lower bounds based on these valid 
inequalities. We prove convergence results for both procedures. 
Section~\ref{sec:liftingD3basic} describes the valid inequalities and an 
iterative procedure that is based on the LP-relaxation~\myref{F3-R}, whereas in 
Section~\ref{sec:liftingD3strong}, an iterative procedure based on the stronger 
LP-relaxation~\myref{F3-V-R} is introduced.

\subsection{Lifting of basic LP-relaxation}\label{sec:liftingD3basic}

Both iterative lifting procedures are based on the following valid 
inequalities.

\begin{lemma}\label{lem:d3_lift_ineq}
	Let~$LB$ be a lower bound on $\nu(\text{\paCP})$. Then, 
	\begin{equation}\label{pac3:lift_ineq}
		\sum_{A\in \mathcal{J}^\alpha} \max\{LB, d_{iA}\} x_{iA} \leq z \qquad 
		\forall i\in I 
	\end{equation}
	are valid inequalities for \myref{F3} and~\myref{F3-V}, which we denote by
	\emph{lifted 
	inequalities}.
\end{lemma}

\begin{proof}
	Let~$(\tilde{x},\tilde{y},\tilde{z})$ be a feasible solution of \myref{F3} 
	and let~$i\in I$.  
	Because of \eqref{pac3:sumx} and \eqref{pac3:xbin} there exists some~$A_i \in \mathcal{J}^\alpha$ such that~$\tilde{x}_{iA_i} = 1$ and~$\tilde{x}_{iA} = 0~$ for all~$A\in \mathcal{J}^\alpha \setminus\{A_i\}$. Therefore,~$\sum_{A\in \mathcal{J}^\alpha} \max\{LB, d_{iA}\} \tilde{x}_{iA} = \max\{LB, d_{iA_i}\}$. We have~$d_{iA_i} \leq \tilde{z}$ by~\eqref{pac3:sumdx} and since~$LB$ is a lower bound on $\nu(\text{\paCP})$, we also have~$LB \leq \tilde{z}$. Hence,~$\sum_{A\in \mathcal{J}^\alpha} \max\{LB, d_{iA}\} \tilde{x}_{iA} = \max\{LB, d_{iA_i}\} \leq \tilde{z}$, so \eqref{pac3:lift_ineq} are valid inequalities for \myref{F3}. The proof for~\myref{F3-V} can be done analogously.
\end{proof}

We use the lifted inequalities~\eqref{pac3:lift_ineq} to obtain a new, 
possibly stronger, lower bound on $\nu(\text{\paCP})$ based on~\myref{F3} in the following way. 

\begin{proposition}\label{prop:new_d3_r}
	Let~$LB$ be a lower bound on $\nu(\text{\paCP})$ and let 
	\begin{subequations}
		\begin{alignat}{3}
			&\mathrm{\mytag{LF3}}(LB) \qquad
			&\mathcal{L}_3(LB) = \min \quad \obj \phantom{iiii} \label{pac3lb:z}  \\ 
			&& \st~ \sum_{j \in \loc} y_j &= p \label{pac3lb:sumy} \\       
			&& \sum_{A\in \mathcal{J}^\alpha} x_{iA} & =  1 && \forall i \in 
			\cus\label{pac3lb:sumx} \\  
			&& x_{iA} & \leq y_j  && \forall i\in I, \forall A\in \mathcal{J}^\alpha, \forall j\in A \label{pac3lb:xy}\\
			&& \sum_{A\in \mathcal{J}^\alpha} \max\{LB, d_{iA}\} x_{iA} & \leq \obj && \forall i \in \cus 
			\label{pac3lb:l_opt}\\
			&& x_{iA} &\geq 0 \qquad&& \forall i \in \cus, \forall A\in \mathcal{J}^\alpha
			\label{pac3lb:xbin}\\
			&& y_{j} &\leq 1 && \forall j \in \loc  \label{pac3lb:ybin}\\
			&& \obj & \in \mathbb{R}.
		\end{alignat}
	\end{subequations}
	Then,~$\mathcal{L}_3(LB)$ is a lower bound on $\nu(\text{\paCP})$ and~$\mathcal{L}_3(LB) \geq LB$ holds.
\end{proposition}

\begin{proof}
	Formulation~\myref{LF3}($LB$) can be obtained by adding the lifted 
	inequalities~\eqref{pac3:lift_ineq} and dropping~\eqref{pac3:sumdx} for each 
	$i\in I$ to the LP-relaxation \myref{F3-R} of \myref{F3}. Since 
	\eqref{pac3:lift_ineq} supersede constraints~\eqref{pac3:sumdx}, we can omit 
	the latter without changing the optimal objective function value. By 
	Lemma~\ref{lem:d3_lift_ineq}, we then see that~$\mathcal{L}_3(LB)$ is a lower 
	bound on $\nu(\text{\paCP})$.	
	Furthermore, we observe that~$\sum_{A\in \mathcal{J}^\alpha} x_{iA} = 1$ implies~$\sum_{A\in\mathcal{J}^\alpha} \max\{LB, d_{iA}\} x_{iA} \geq \sum_{A\in \mathcal{J}^\alpha} LB x_{iA} = LB$ for all~$i\in I$. Thus, every feasible solution~$(\tilde{x},\tilde{y},\tilde{z})$ of \myref{LF3}($LB$) satisfies~$\tilde{z} \geq LB$, which implies~$\mathcal{L}_3(LB) \geq LB$.
\end{proof}

Given a lower bound~$LB$ on $\nu(\text{\paCP})$, we can solve~\myref{LF3}($LB$) to obtain a new lower bound that is at least as good as $LB$. 
The following theorem gives a characterization of lower bounds~$LB$ that can not be improved anymore by solving \myref{LF3}($LB$).

\begin{theorem}\label{theo:d3_lift_conv}
	Let~$LB$ be a lower bound on $\nu(\text{\paCP})$. Then~$\mathcal{L}_3(LB)=LB$ holds if and only if there is a feasible solution for the \emph{fractional $\alpha$ set cover problem}
		\begin{subequations}
		\begin{alignat}{3}
			&\mathrm{\mytag{FASC}}(LB) \qquad
			& \min \quad \sum_{j\in J} y_j \phantom{iiii} \label{faclb:z}  \\ 
			&& \st~  \sum_{\substack{A\in \mathcal{J}^\alpha\\d_{iA}\leq LB}} u_A & \geq 1 \qquad&& \forall i\in \cus \label{faclb:sumu}\\
			&& u_A &\leq y_j &&  \forall A\in \mathcal{J}^\alpha, \forall j\in A \label{faclb:u_a}\\
			&& u_A &\geq 0 &&  \forall A\in \mathcal{J}^\alpha \label{faclb:ubin} \\
			&&  y_{j} &\leq 1 && \forall j \in \loc  \label{faclb:ybin}
		\end{alignat}
	\end{subequations}
	with objective function value at most~$p$.
\end{theorem}

\begin{proof}
	First, assume that~$\mathcal{L}_3(LB)=LB$ holds. Then, there is an optimal solution~$(\tilde{x}, \tilde{y}, \tilde{z})$ of \myref{LF3}($LB$) with~$\tilde{z}=LB$. Let~$\tilde{u} = (\tilde{u}_A)_{A\in \mathcal{J}^\alpha}$ with $\tilde{u}_A = \min\{\tilde{y}_j: j\in A\}$ for all~$A\in \mathcal{J}^\alpha$. We show that $(\tilde{y}, \tilde{u})$ is a feasible solution of~\myref{FASC}($LB$) with objective function value at most $p$.
	Clearly, constraints~\eqref{faclb:u_a},~\eqref{faclb:ubin} and~\eqref{faclb:ybin} hold for $\tilde{y}$ and~$\tilde{u}$ by definition.
	Due to \eqref{pac3lb:l_opt} and \eqref{pac3lb:sumx}, we have
	\begin{align*}
		LB &= \tilde{z} \geq \sum_{A\in \mathcal{J}^\alpha} \max\{LB,d_{iA}\} \tilde{x}_{iA} \geq \sum_{A\in \mathcal{J}^\alpha} LB \tilde{x}_{iA} = LB
	\end{align*}
	for all~$i\in I$.
	Therefore, 
	\[
		LB = \sum_{A\in \mathcal{J}^\alpha} \max\{LB,d_{iA}\} \tilde{x}_{iA} 
		= \sum_{\substack{A\in \mathcal{J}^\alpha\\d_{iA}\leq LB}} LB \tilde{x}_{iA} + \sum_{\substack{A\in \mathcal{J}^\alpha\\d_{iA}> LB}} d_{iA} \tilde{x}_{iA}
	\]
	holds for all~$i\in I$, which together with \eqref{pac3lb:sumx} and   \eqref{pac3lb:xbin}	 
	implies~$\tilde{x}_{iA} = 0$ for all~$i\in I$ and~$A\in \mathcal{J}^\alpha$ 
	 with~$d_{iA} > LB$. 
	Because of this, \eqref{pac3lb:sumx} and~\eqref{pac3lb:xy}, we also have 
	\[
		1 = \sum_{A\in \mathcal{J}^\alpha} \tilde{x}_{iA} = \sum_{\substack{A\in \mathcal{J}^\alpha\\d_{iA}\leq LB}} \tilde{x}_{iA} \leq \sum_{\substack{A\in \mathcal{J}^\alpha\\d_{iA}\leq LB}} \min\{\tilde{y}_j:~j\in A\} = \sum_{\substack{A\in \mathcal{J}^\alpha\\d_{iA}\leq LB}} \tilde{u}_A
	\]
	for all~$i\in I$.
	Thus,~$\tilde{u}$ satisfies \eqref{faclb:sumu} as well. 
	By \eqref{pac3lb:sumy}, we also have~$\sum_{j\in J} \tilde{y}_j = p$. 
	Therefore,~$(\tilde{y}, \tilde{u})$ is a feasible solution of \myref{FASC}($LB$) with objective function value~$p$, which finishes the proof of the first direction of the equivalence.
	
	For the other direction, let~$(y^*, u^*)$ be a feasible solution of 
	\myref{FASC}($LB$) with objective function value at most~$p$.	 
	If its objective function value is less than $p$,
	we construct a new feasible solution~$(\tilde{y}, \tilde{u})$ of \myref{FASC}($LB$) with objective function value $p$ by setting $\tilde{u} = u^*$ and increasing the values of~$y^*_j$ for some~$j\in J$ such that~$y^*_j \leq \tilde{y}_j \leq 1$ for all~$j \in J$ and~$\sum_{j\in J} \tilde{y}_j = p$. 
	Such~$\tilde{y}$ exists, because~$p\leq |\loc|$. 
	We note that~$\tilde{u}_A = u^*_A \leq y^*_j \leq \tilde{y}_j$ for all~$A\in \mathcal{J}^\alpha$ and~$j\in A$.
	Thus, in any case we can obtain a feasible solution~$(\tilde{y}, \tilde{u})$ of~\myref{FASC}($LB$) with objective function value~$p$.
	
	Next, we define~$\tilde{z}=LB$ and let~$i\in I$.
	Now, let 
	\[
		A_i \in \arg\min_{A^\prime\in \mathcal{J}^\alpha} \left\{d_{iA^\prime}: \sum_{\substack{A\in \mathcal{J}^\alpha\\d_{iA} <d_{iA^\prime}}} \tilde{u}_A < 1 \textrm{ and } \sum_{\substack{A\in \mathcal{J}^\alpha\\d_{iA}\leq d_{iA^\prime}}} \tilde{u}_A \geq 1 \right\}.
	\]
	Such~$A_i$ exists by \eqref{faclb:sumu} and we further know that~$d_{iA_i} \leq LB$. Let 
	\[
		\sigma_i = \frac{\left( 1 - \sum_{\substack{A\in \mathcal{J}^\alpha\\d_{iA}< d_{iA_i}}} \tilde{u}_A \right)}{\sum_{\substack{A\in \mathcal{J}^\alpha\\d_{iA}= d_{iA_i}}} \tilde{u}_A}.
	\]
	Then $0\leq \sigma_i$ because $\tilde{u} \geq 0$ and  $\sum_{\substack{A\in \mathcal{J}^\alpha\\d_{iA} <d_{iA^\prime}}} \tilde{u}_A < 1$.
	Since~$1 \leq \sum_{\substack{A\in \mathcal{J}^\alpha\\d_{iA}\leq d_{iA_i}}} \tilde{u}_A = \sum_{\substack{A\in \mathcal{J}^\alpha\\d_{iA}< d_{iA_i}}} \tilde{u}_A + \sum_{\substack{A\in \mathcal{J}^\alpha\\d_{iA}= d_{iA_i}}} \tilde{u}_A$
	implies 
	$1 - \sum_{\substack{A\in \mathcal{J}^\alpha\\d_{iA}< d_{iA_i}}} \tilde{u}_A \leq \sum_{\substack{A\in \mathcal{J}^\alpha\\d_{iA}= d_{iA_i}}} \tilde{u}_A$,
	we also have $\sigma_i\leq 1$.
	Next, we define
	\[
		\tilde{x}_{iA} = \left\{ \begin{array}{ll}
			\tilde{u}_{A} & \textrm{if } d_{iA} < d_{iA_i} \\
			\tilde{u}_{A} \sigma_i & \textrm{if } d_{iA} = d_{iA_i} \\
			0 & \textrm{if } d_{iA} > d_{iA_i} 
		\end{array}\right.
	\]
	for all~$A \in \mathcal{J}^\alpha$.
	Then,~$(\tilde{x}_{iA})_{A\in \mathcal{J}^\alpha}$ satisfies \eqref{pac3lb:sumx}, because
	\begin{align*}
		\sum_{A\in \mathcal{J}^\alpha} \tilde{x}_{iA} &=
		\sum_{\substack{A\in \mathcal{J}^\alpha\\d_{iA}<d_{iA_i}}} \tilde{u}_A + \sum_{\substack{A\in \mathcal{J}^\alpha\\d_{iA}=d_{iA_i}}} \tilde{u}_A \sigma_i
		= \sum_{\substack{A\in \mathcal{J}^\alpha\\d_{iA}<d_{iA_i}}} \tilde{u}_A + \sum_{\substack{A\in \mathcal{J}^\alpha\\d_{iA}=d_{iA_i}}} \tilde{u}_A \frac {\left( 1 - \sum_{\substack{A\in \mathcal{J}^\alpha\\d_{iA}<d_{iA_i}}} \tilde{u}_A \right)} {\sum_{\substack{A\in \mathcal{J}^\alpha\\d_{iA}= d_{iA_i}}} \tilde{u}_A} 
		= 1.
	\end{align*}
	
	Together with~\eqref{faclb:u_a}, we have~$\tilde{x}_{iA} \leq \tilde{u}_A \leq 
	\tilde{y}_j$ for all~$A\in \mathcal{J}^\alpha$ and~$j\in A$. 
	Hence,~$(\tilde{x}, \tilde{y})$ satisfies \eqref{pac3lb:xy}.
	Because of~\eqref{faclb:ybin} and~\eqref{faclb:ubin}, it is easy to see that \eqref{pac3lb:xbin} and \eqref{pac3lb:ybin} are satisfied as well. 
	Furthermore, we have 
	\begin{align*}
	\sum_{A\in \mathcal{J}^\alpha} \max\{LB, d_{iA}\} \tilde{x}_{iA} 
	&= \sum_{\substack{A\in \mathcal{J}^\alpha\\d_{iA}\leq LB}} LB 
	\tilde{x}_{iA} + \sum_{\substack{A\in \mathcal{J}^\alpha\\d_{iA}>LB}} 
	d_{iA} \tilde{x}_{iA} 
	= \sum_{A\in \mathcal{J}^\alpha} LB \tilde{x}_{iA} = LB = \tilde{z}
	\end{align*}
	where the second equality follows from~$\tilde{x}_{iA}=0$ for all~$A\in \mathcal{J}^\alpha$ with~$d_{iA} > d_{iA_i}$ and~$LB \geq d_{iA_i}$.
	Hence,~\eqref{pac3lb:l_opt} hold as well. 
	Constraints~\eqref{pac3lb:sumy} and~\eqref{pac3lb:ybin} are satisfied by our construction of~$\tilde{y}$. 
	Thus, as~$i\in I$ was arbitrary, constraints~\eqref{pac3lb:sumy} to~\eqref{pac3lb:xbin} hold for all~$i\in I$.
	Therefore,~$(\tilde{x}, \tilde{y}, \tilde{z})$ is a feasible solution of \myref{LF3}($LB$) with objective function value~$\tilde{z}=LB$, which implies~$\mathcal{L}_3(LB) \leq LB$. As~$\mathcal{L}_3(LB) \geq LB$ holds by Proposition~\ref{prop:new_d3_r}, we have~$\mathcal{L}_3(LB) = LB$, which finishes the proof.	
\end{proof}

Theorem~\ref{theo:d3_lift_conv} implies that if we have a lower bound~$LB$ with $\mathcal{L}_3(LB) = LB$, we know that $\mathcal{L}_3(LB') = LB'$ holds for all better bounds $LB'> LB$ as well, since a feasible solution of~\myref{FASC}($LB$) with objective function value at most $p$ is also a feasible solution of~\myref{FASC}($LB^\prime$) with greater radius $LB'> LB$. This motivates the definition $LB^\#_3 = \min \{LB\in \mathbb{R}: \mathcal{L}_3(LB) = LB\}$ of the smallest lower bound on $\nu(\text{\paCP})$ that does not improve by solving~\myref{LF3}($LB$).

This bound $LB^\#_3$ can be computed by the following iterative approach similar to that in \citet{gaar2022scaleable} for the~\pCP. Given an initial lower bound $LB$ on $\nu(\text{\paCP})$,  
we solve \myref{LF3}($LB$) and obtain a new lower bound~$LB' = \mathcal{L}_3(LB)$. If $LB' = LB$, we stop. Otherwise, we solve \myref{LF3}($LB'$) again to obtain a new potentially better lower bound. We repeat this until we reach the lower bound $LB^{\#}_3$ that does not improve by solving \myref{LF3}($LB^{\#}_3$). 
The following theorem gives a result on the runtime of this procedure.

\begin{theorem}\label{theo:polytime_LB3}
	For a fixed value of~$\alpha$, the lower bound~$LB^{\#}_3$ can be computed in polynomial time. 	
\end{theorem}	

\begin{proof}
	We begin by showing $LB^{\#}_3 \in D^\alpha$. 
	Assume the contrary and let $d^{\max} = \max\{d\in D^\alpha: d \leq LB^{\#}_3\}$. Then, $d^{\max} < LB^{\#}_3$.
	Now, let $(\tilde{y}, \tilde{u})$ be an optimal solution of \myref{FASC}($LB^{\#}_3$). By Theorem~\ref{theo:d3_lift_conv}, we have $\sum_{j\in J} \tilde{y}_j \leq p$. Since $\sum_{\substack{A\in \mathcal{J}^\alpha\\ d_{iA} \leq LB^{\#}_3}} \tilde{u}_{iA} = \sum_{\substack{A\in \mathcal{J}^\alpha\\ d_{iA} \leq d^{\max}}} \tilde{u}_{iA}$ for all $i\in I$, we know that $(\tilde{y},\tilde{u})$ is a feasible solution of \myref{FASC}($d^{\max}$) with $\sum_{j\in J} \tilde{y}_j \leq p$ as well. Applying Theorem~\ref{theo:d3_lift_conv} again yields $\mathcal{L}_3(d^{\max}) = d^{\max}$. Since this contradicts the definition of~$LB^{\#}_3$ as the minimal number satisfying this equality, we know that $LB^{\#}_3 \in D^\alpha$ must hold. 
	
	Now, as $\nu(\text{\paCP})$ is always an $\alpha$-distance, 
	we can increase a given lower bound 
	to the next value in~$D^\alpha$ and still have a valid lower bound.
	Then, starting from the lower bound $d^{\min} = \min \{d\in D^\alpha\}$, we obtain a new lower bound $d' = \min \{ d\in D^\alpha: d\geq \mathcal{L}_3(d^{\min})\}$ in polynomial time by solving an LP with~$O(m^\alpha)$ variables and~$O(n+\alpha m^{\alpha})$~constraints and increasing the optimal objective function value to the next $\alpha$-distance. 
	As $LB^{\#}_3 \in D^\alpha$, we need at most $|D^\alpha|\leq m^\alpha$ such iterations.
	Thus, for a fixed value of~$\alpha$,~$LB^{\#}_3$ can be computed in polynomial time.  
\end{proof}

We observe that for $\alpha=1$, the value $LB^\#_3$ is a lower bound on $\nu(\text{\pCP})$. The best lower bound that \citet{gaar2022scaleable} obtained with their approach is the smallest number~$LB$
for which there is a feasible solution of the fractional set cover problem with radius~$LB$ that uses at most~$p$ sets.
Since for $\alpha = 1$ our problem~\myref{FASC}($LB$) is equivalent to the fractional set cover problem with radius~$LB$, in this case our bound~$LB_3^\#$ is exactly the best bound $LB$ identified by \citet{gaar2022scaleable}.

\subsection{Lifting of stronger LP-relaxation}\label{sec:liftingD3strong}

Next, we apply the ideas of the previous section to the LP-relaxation~\myref{F3-V-R} of~\myref{F3-V}.

\begin{proposition}\label{prop:lift_d3_v_r}
	Let~$LB$ be a lower bound on $\nu(\text{\paCP})$ and let 
	\begin{subequations}
		\begin{alignat}{3}
			&\textnormal{\mytag{LF3-V}}(LB) \qquad
			& \mathcal{L}_{3V}(LB) =   \min \quad \obj \phantom{iiii} \label{pac3vlb:z}  \\ 
			&& \st~   \sum_{j \in \loc} y_j &= p \label{pac3vlb:sumy} \\       
			&& \sum_{A\in \mathcal{J}^\alpha} x_{iA} & =  1 && \forall i \in \cus\label{pac3vlb:sumx} \\
			&& \sum_{\substack{A\in \mathcal{J}^\alpha \\ j\in A}} x_{iA} &\leq y_j && \forall i\in I, \forall j\in J \label{pac3vlb:valid}\\
			&& \sum_{A\in \mathcal{J}^\alpha} \max\{LB, d_{iA}\} x_{iA} & \leq \obj && \forall i \in \cus 
			\label{pac3vlb:l_opt}\\
			&& x_{iA} &\geq 0 \qquad&& \forall i \in \cus, \forall A\in \mathcal{J}^\alpha
			\label{pac3vlb:xbin}\\
			&& y_{j} &\leq 1 && \forall j \in \loc  \label{pac3vlb:ybin}\\
			&& \obj & \in \mathbb{R} \label{pac3vlb:zbin}.
		\end{alignat}
	\end{subequations}
	Then, $\mathcal{L}_{3V}(LB)$ is a lower bound on $\nu(\text{\paCP})$ and $\mathcal{L}_{3V}(LB) \geq LB$ holds.
\end{proposition}

We do not give the proof of Proposition~\ref{prop:lift_d3_v_r} here, as it is 
analogous to that of Proposition~\ref{prop:new_d3_r}.
Proposition~\ref{prop:lift_d3_v_r} gives rise to an iterative procedure based on~\myref{F3-V} to 
improve lower bounds on $\nu(\text{\paCP})$ analogously to the iterative procedure based on~\myref{F3}
described in Section~\ref{sec:liftingD3basic}.
Similar to Theorem~\ref{theo:d3_lift_conv}, the following theorem gives a 
characterization of lower bounds~$LB$ that can not be improved anymore by 
solving \myref{LF3-V}($LB$).

\begin{theorem}\label{theo:d3v_lift_conv}
	Let~$LB$ be a lower bound on $\nu(\text{\paCP})$. Then $\mathcal{L}_{3V}(LB)~=~LB$ holds if and only if there is a feasible solution of 
	the fractional $\alpha$ set cover problem based on valid inequalities
	\begin{subequations}
		\begin{alignat}{3}
			&\textnormal{\mytag{FASC-V}}(LB) \qquad
			& \min \quad \sum_{j\in J} y_j \phantom{iiii} \label{fac2lb:z}  \\ 
			&& \st~  \sum_{\substack{A\in \mathcal{J}^\alpha\\d_{iA}\leq LB}} u_{iA} & \geq 1 \qquad&& \forall i\in \cus \label{fac2lb:sumu}\\
			&& \sum_{\substack{A\in \mathcal{J}^\alpha\\ j\in A}} u_{iA} &\leq y_j &&  \forall i\in I, \forall j\in J \label{fac2lb:u_a}\\
			&& u_{iA} &\geq 0 && \forall i\in I, \forall A\in \mathcal{J}^\alpha \label{fac2lb:ubin} \\
			&& y_{j} &\leq 1 && \forall j \in \loc.  \label{fac2lb:ybin} 
		\end{alignat}
	\end{subequations}
	with objective function value at most~$p$.
\end{theorem}

\begin{proof}
	First, assume~$\mathcal{L}_{3V}(LB)=LB$ holds. Then, there is an optimal solution~$(\tilde{x},\tilde{y},\tilde{z})$ of~\myref{LF3-V}($LB$) with~$\tilde{z}=LB$. Let~$\tilde{u}=\tilde{x}$.
	Analogous to the proof of Theorem~\ref{theo:d3_lift_conv}, we can show that~$1=\sum_{\substack{A\in \mathcal{J}^\alpha\\d_{iA}\leq LB}} \tilde{x}_{iA}$ for all~$i\in I$.
	Then~$1 = \sum_{\substack{A\in \mathcal{J}^\alpha\\d_{iA}\leq LB}} \tilde{u}_{iA}$ for all~$i\in I$ and constraints~\eqref{fac2lb:sumu} hold for $\tilde{u}$. Constraints~\eqref{fac2lb:u_a} to~\eqref{fac2lb:ybin} follow analogously to the proof of Theorem~\ref{theo:d3_lift_conv}. Thus,~$(\tilde{y}, \tilde{u})$ is a feasible solution of~\myref{FASC-V}($LB$) with objective function value~$p$, which finishes the proof of the first direction of the equivalence.  
	
	For the other direction, let~$(\tilde{y}, \tilde{u})$ be a feasible solution of \myref{FASC-V}($LB$) with objective function value~$p$. Such a solution exists by similar arguments as in Theorem~\ref{theo:d3_lift_conv}. Define~$\tilde{z}=LB$. It is easy to see that $\tilde{y}$ and~$\tilde{z}$ satisfy~\eqref{pac3vlb:sumy},~\eqref{pac3vlb:ybin} and~\eqref{pac3vlb:zbin}. 
	Next, we choose 
	\[
		A_i \in \arg\min_{A^\prime\in \mathcal{J}^\alpha} \left\{d_{iA^\prime}: \sum_{\substack{A\in \mathcal{J}^\alpha\\d_{iA} <d_{iA^\prime}}} \tilde{u}_{iA} < 1 \textrm{ and } \sum_{\substack{A\in \mathcal{J}^\alpha\\d_{iA}\leq d_{iA^\prime}}} \tilde{u}_{iA} \geq 1 \right\}
	\]
	and define
	\[
		\tilde{x}_{iA} = \left\{ \begin{array}{ll}
			\tilde{u}_{iA} & \textrm{if } d_{iA} < d_{iA_i} \\
			\tilde{u}_{iA} \sigma_i & \textrm{if } d_{iA} = d_{iA_i} \\
			0 & \textrm{if } d_{iA} > d_{iA_i} 
		\end{array}\right.
	\]
	with
	\[
		\sigma_i = \frac{\left( 1 - \sum_{\substack{A\in \mathcal{J}^\alpha\\d_{iA}< d_{iA_i}}} \tilde{u}_{iA} \right)}{\sum_{\substack{A\in \mathcal{J}^\alpha\\d_{iA}= d_{iA_i}}} \tilde{u}_{iA}}
	\]
	for all $i\in I$ and $A\in \mathcal{J}^\alpha$.
	Then, by similar arguments as in the proof of Theorem~\ref{theo:d3_lift_conv}, we can show that~$(\tilde{x}, \tilde{y}, \tilde{z})$ with $\tilde{x} = (\tilde{x}_{iA})_{i\in I, A\in \mathcal{J}^\alpha}$ satisfies constraints~\eqref{pac3vlb:sumx},~\eqref{pac3vlb:valid},~\eqref{pac3vlb:l_opt} and~\eqref{pac3vlb:xbin}. 
	Therefore,~$(\tilde{x}, \tilde{y}, \tilde{z})$ is a feasible solution of \myref{LF3-V}($LB$) with objective function value~$\tilde{z}=LB$. This implies~$\mathcal{L}_{3V}(LB) \leq LB$, which together with~$\mathcal{L}_{3V}(LB) \geq LB$ from Proposition~\ref{prop:lift_d3_v_r} yields~$\mathcal{L}_{3V}(LB) = LB$. This finishes the proof.	
\end{proof}

We note that if a lower bound $LB$ on $\nu(\text{\paCP})$ satisfies $\mathcal{L}_{3V}(LB) = LB$, then the same is true for all greater lower bounds $LB' > LB$ as well by Theorem~\ref{theo:d3v_lift_conv}. Therefore, we define the lower bound $LB^{\#}_{3V} =  \min \{LB\in \mathbb{R}: \mathcal{L}_{3V}(LB) = LB\}$ that can be efficiently computed by an iterative approach similar to that described in Section~\ref{sec:liftingD3basic} as the following theorem shows. 

\begin{theorem}\label{theo:polytimeLB3v}
	For a fixed value of~$\alpha$, the lower bound~$LB^{\#}_{3V}$ can be computed in polynomial time. 	
\end{theorem}

The proof of Theorem~\ref{theo:polytimeLB3v} can be done analogously to the one of 
Theorem~\ref{theo:polytime_LB3} by exploiting the fact that 
the LP \myref{FASC-V}($LB$) has $O(nm^\alpha)$ variables, 
and~$O(nm)$~constraints and is therefore omitted here.

We close this section by comparing this new lower bound~$LB^{\#}_{3V}$ based on the stronger LP-relaxation of \myref{F3-V} to the one based on the weaker LP-relaxation of \myref{F3} from the previous section. 

\begin{theorem}\label{theo:lift_bounds_d3}
	$LB^{\#}_{3} \leq LB^{\#}_{3V}$.
\end{theorem}

\begin{proof}
	Let $(\tilde{x}, \tilde{y}, \tilde{z})$ be an optimal solution of \myref{LF3-V}($LB^{\#}_{3V}$). By the definition of $LB^{\#}_{3V}$, we know $\tilde{z} = LB^{\#}_{3V}$. It is easy to see that $(\tilde{x}, \tilde{y}, \tilde{z})$ is a feasible solution of \myref{LF3}($LB^{\#}_{3V}$), hence $\mathcal{L}_3(LB^{\#}_{3V}) \leq \tilde{z} = LB^{\#}_{3V}$. Because of Proposition~\ref{prop:new_d3_r}, this already implies $\mathcal{L}_3(LB^{\#}_{3V}) = LB^{\#}_{3V}$.
	Then, $LB^{\#}_{3} \leq LB^{\#}_{3V}$ by the definition of $LB^{\#}_{3}$.
\end{proof}

Preliminary computations showed that there are instances, for which $LB^{\#}_{3} < LB^{\#}_{3V}$ holds. This is not surprising, as in Examples~\ref{ex:d1vsd3} and~\ref{ex:d1vsd3v} we saw an instance for which $\nu\myref{F3-R} < \nu\myref{F3-V-R}$ holds.
Thus, adding the lifted inequalities~\eqref{pac3:lift_ineq} iteratively 
to \myref{F3} and \myref{F3-V} does not  
close the potential gap in the bounds obtained from their LP-relaxations in general. 

\section{Lifting of formulation (F1)}\label{sec:lifting_D1}

Next, we want to apply the lifting procedure of \citet{gaar2022scaleable} 
to our MIP formulation~\myref{F1}.  
Unfortunately, this endeavor is much more challenging for~\myref{F1}, as we do 
not have the potential optimal objective function value as 
coefficients 
of 
single variables, such that we could directly apply the 
lifting as in the 
lifted inequalities~\eqref{pac3:lift_ineq} for formulations~\myref{F3} 
and~\myref{F3-V}.
To overcome this obstacle, we
define 
$$S_i(LB) = \left\{w_i\in \mathbb{R}^m: \sum_{j\in A} w_{ij} \leq \max \{LB, d_{iA}\} \textrm{ for all } A\in \mathcal{J}^\alpha\right\}$$
for $LB \in \mathbb{R}$.
This allows us to obtain 
valid inequalities, as the following lemma shows.  

\begin{lemma}\label{lem:d1_lift_ineq}
	Let~$LB$ be a lower bound on $\nu(\text{\paCP})$.
	Then,  
	\begin{equation}\label{pac1:lift_ineq}
		\sum_{j\in J} w_{ij} x_{ij} \leq z \qquad \forall i\in I, \forall 
		w_i\in S_i(LB) 
	\end{equation}
	are valid inequalities for \myref{F1}, which we denote 
	by \emph{lifted inequalities.}
\end{lemma}

\begin{proof}
	Let~$(\bar{x},\bar{y},\bar{z})$ be a feasible solution of \myref{F1}, let $i\in I$ and let~$w_{i}\in S_i(LB)$. 
	Define $A_i = \{j\in J: \bar{x}_{ij} = 1\}$, then $|A_i| = \alpha$ and $\bar{x}_{ij} = 0$ for all $j\in J\setminus A_i$ by~\eqref{pac1:sumx} and~\eqref{pac1:xbin}. 
	Thus, we have $\sum_{j\in J} w_{ij} \bar{x}_{ij} = \sum_{j\in A_i} w_{ij} \leq \max\{LB, d_{iA_i}\}$ since~$w_i \in S_i(LB)$ and $A_i\in \mathcal{J}^\alpha$.
	Additionally, $\max \{LB, d_{iA_i}\} \leq \bar{z}$ holds, because clearly $LB\leq \bar{z}$ and $d_{iA_i} = \sum_{j\in J} d_{ij} \bar{x}_{ij} \leq \bar{z}$ follows from~\eqref{pac1:sumdx}. 
	Therefore, $\sum_{j\in J} w_{ij} \bar{x}_{ij} \leq \bar{z}$ holds and thus~\eqref{pac1:lift_ineq} are valid inequalities for~\myref{F1}. 
\end{proof}
 
If we consider only optimal solutions, we can strengthen the valid inequalities~\eqref{pac1:lift_ineq} with 
the following idea.
If we are given an upper bound~$UB$ on $\nu(\text{\paCP})$ and a set $A\in 
\mathcal{J}^\alpha$ with $d_{iA} > UB$, then in any optimal solution for the 
\paCP no customer $i$ can be assigned to all
facilities $j\in A$, as such a solution cannot be optimal. 
Thus,  
we do not care if the condition in the definition of $S_i(LB)$ is satisfied for this specific $A$. 
This gives rise to the set
\[
	S_i(LB,UB) = \left\{w_i\in \mathbb{R}^m: \sum_{j\in A} w_{ij} \leq \max \{LB, d_{iA}\} \textrm{ for all } A\in \mathcal{J}^\alpha \textrm{ with } d_{iA} \leq UB \right\},
\] 
which is a superset of $S_i(LB)$, so $S_i(LB) \subseteq S_i(LB,UB)$ holds
for all $LB, UB \in \mathbb{R}$. 
The following lemma shows that~\eqref{pac1:lift_ineq} still hold for optimal solutions of \myref{F1}, even if $w_i$ comes from the superset $S_i(LB,UB)$. 

\begin{lemma}\label{lem:d1_lift_ineq_opt}
	Let~$LB$ and~$UB$ be lower and upper bounds on $\nu(\text{\paCP})$, respectively. 
	Then,  
	\begin{equation}\label{pac1:lift_ineq_opt}
		\sum_{j\in J} w_{ij} x_{ij} \leq z \qquad \forall i\in I, \forall 
		w_i\in S_i(LB,UB) 
	\end{equation}
	are optimality-preserving inequalities for \myref{F1}, which we denote 
	by \emph{extended lifted inequalities.}
	In particular, every optimal solution of \myref{F1} satisfies 
	\eqref{pac1:lift_ineq_opt}.
\end{lemma}

\begin{proof}
	Let~$(\bar{x},\bar{y},\bar{z})$ be an optimal solution of \myref{F1}, let $i\in I$ and let~$w_{i}\in S_i(LB,UB)$.
	Again, we define $A_i = \{j\in J: \bar{x}_{ij} = 1\}$ and 
	can derive $d_{iA_i} \leq UB$, as otherwise
	$UB < d_{iA_i} = \sum_{j\in A_i} d_{ij} = \sum_{j\in J} d_{ij} \bar{x}_{ij} \leq \bar{z}$
	contradicts the optimality of~$(\bar{x},\bar{y},\bar{z})$.
	Therefore, $A_i$ is one of the sets in $\mathcal{J}^\alpha$ for which the condition given in $S_i(LB,UB)$ is satisfied. 	
	Hence, $(\bar{x},\bar{y},\bar{z})$ satisfies \eqref{pac1:lift_ineq_opt} by the same arguments as in the proof of Lemma~\ref{lem:d1_lift_ineq}. 
\end{proof}

Note, that $S_i(LB) = S_i(LB, \infty)$. We will use the stronger set of 
inequalities~\eqref{pac1:lift_ineq_opt} in our implementation detailed in 
Section~\ref{sec:implementation}. However, for the theory described in the 
remainder of this section, we assume $UB = \infty$ and utilize the lifted
inequalities~\eqref{pac1:lift_ineq} analogously to how it is done the approach in 
Section~\ref{sec:lifting_D3}. Indeed, by adding these lifted inequalities 
to the 
LP-relaxation of \myref{F1}, we obtain a new lower bound on $\nu(\text{\paCP})$ 
that is possibly stronger in the following way.

\begin{proposition}\label{prop:new_d1_r}
	Let~$LB$ be a lower bound on $\nu(\text{\paCP})$ 
	and let
	\begin{subequations}
		\begin{alignat}{3}
			& \mathrm{\mytag{LF1}}(LB) \qquad
			& \mathcal{L}_1(LB) =  \min \quad \obj \phantom{iiii} \label{pac1lb:z}  \\ 
			&& \st~    \sum_{j \in \loc}y_{j}       & =     p                         &&                               \label{pac1lb:sumy} \\
			&&             \sum_{j \in \loc}x_{ij}       & =     \alpha                  && \forall i \in \cus                          \label{pac1lb:sumx} \\
			&& x_{ij}                 & \leq  y_j                        \qquad && \forall i\in \cus, \forall j\in \loc                   \label{pac1lb:xy} \\
			&& \sum_{j\in J} w_{ij} x_{ij} &\leq z &&\forall i\in I, \forall w_i\in S_i(LB) \label{pac1lb:lift}\\
			&        &               x_{ij} &\geq 0                   && \forall i\in \cus, \forall j\in \loc                    \label{pac1lb:xbin}     \\
			&        &               y_{j}                  & \leq 1  && \forall j \in \loc                           \label{pac1lb:ybin} \\
			&& z &\in \mathbb{R}.  \label{pac1lb:zbin}  
		\end{alignat}
	\end{subequations}
	Then, $\mathcal{L}_1(LB)$ is a lower bound on $\nu(\text{\paCP})$ and $\mathcal{L}_1(LB) \geq LB$ holds.
\end{proposition}

\begin{proof}		
	We can improve the LP-relaxation \myref{F1-R} of \myref{F1} by adding the 
	lifted inequalities~\eqref{pac1:lift_ineq} because of 
	Lemma~\ref{lem:d1_lift_ineq}. 
	Morover, the set $S_i(LB)$ of 
	coefficients of the lifted inequalities~\eqref{pac1:lift_ineq} contains the 
	coefficients defining the constraints~\eqref{pac1:sumdx}, and thus we can 
	omit including~\eqref{pac1:sumdx} explicitly. As a result, 
	$\mathcal{L}_1(LB)$ is a lower bound on $\nu(\text{\paCP})$.
	In order to show~$\mathcal{L}_1(LB) \geq LB$, let~$(\bar{x}, \bar{y}, 
	\bar{z})$ be a feasible solution of \myref{LF1}($LB$). Let~$i\in I$ 
	and define~$\bar{w}_{ij} = \frac{LB}{\alpha}$ for all~$j\in J$. 
	Then,~$\sum_{j\in A} \bar{w}_{ij} = \sum_{j\in A}\frac{LB}{\alpha} = LB 
	\leq \max \{LB, d_{iA}\}$ for all~$A\in \mathcal{J}^\alpha$, so~$w_{i} \in 
	S_i(LB)$. As a consequence, 
	$\bar{z}\geq \sum_{j\in J} \bar{w}_{ij} \bar{x}_{ij} = \frac{LB}{\alpha} \sum_{j\in J} \bar{x}_{ij} = LB$ because of~\eqref{pac1:sumx}, and~$\mathcal{L}(LB) \geq LB$ follows.
\end{proof}

We note that due to the lifted inequalities~\eqref{pac1lb:lift} the LP 
\myref{LF1}($LB$) has an infinite number of constraints. However, one can deal with 
these inequalities in a cutting plane-fashion, i.e., separate them when necessary. 
In our solution algorithm, which uses the (extended) lifted inequalities, we do 
so, see Section~\ref{sec:b_and_c} for details.

With the help of Proposition~\ref{prop:new_d1_r}, by defining $LB^{\#}_1 = \min \{LB \in \mathbb{R}: \mathcal{L}_1(LB) = LB\}$ we obtain a lower bound on $\nu(\text{\paCP})$. 
Next, we compare $LB^{\#}_1$ to $LB^{\#}_{3V}$ in the following two lemmata.

\begin{lemma} \label{lem:LB1_leq_LB3V}
	$LB^{\#}_1 \leq LB^{\#}_{3V}$.
\end{lemma}

\begin{proof}
	We have $\mathcal{L}_{3V}(LB^{\#}_{3V}) = LB^{\#}_{3V}$ by definition, so let $(\tilde{x},\tilde{y},\tilde{z})$ be an optimal solution of \myref{LF3-V}($LB^{\#}_{3V}$) with $\tilde{z}=LB^{\#}_{3V}$. 
	Define $\bar{x}_{ij} = \sum_{\substack{A\in \mathcal{J}^\alpha\\ j\in A}} \tilde{x}_{iA}$ for all $i\in I$ and $j\in J$, $\bar{y} = \tilde{y}$ and $\bar{z} = \tilde{z}$. It is easy to see that constraints~\eqref{pac1lb:sumy} to~\eqref{pac1lb:xy} and constraints~\eqref{pac1lb:xbin} to~\eqref{pac1lb:zbin} are satisfied by $(\bar{x}, \bar{y}, \bar{z})$. 
	Furthermore, for all $i\in I$ and $w_i\in S_i(LB)$, we have
	\[
	\sum_{j\in J} w_{ij} \bar{x}_{ij} = \sum_{j\in J} w_{ij} \sum_{\substack{A\in \mathcal{J}^\alpha \\j\in A}} \tilde{x}_{iA} = \sum_{A\in \mathcal{J}^\alpha}  \left( \sum_{j\in A} w_{ij} \right) \tilde{x}_{iA} \leq \sum_{A\in \mathcal{J}^\alpha} \max\{LB, d_{iA}\} \tilde{x}_{iA}   \leq \tilde{z} = \bar{z}
	\]
	because of~\eqref{pac3lb:l_opt}, 
	hence constraints~\eqref{pac1lb:lift} hold as well. 
	This shows that $(\bar{x}, \bar{y}, \bar{z})$ is a feasible solution of \myref{LF1}($LB^{\#}_{3V}$) and therefore 
	$\mathcal{L}_1(LB^{\#}_{3V}) \leq LB^{\#}_{3V}$ holds. Since $\mathcal{L}_1(LB^{\#}_{3V}) \geq LB^{\#}_{3V}$ by Proposition~\ref{prop:new_d1_r}, we have $\mathcal{L}_1(LB^{\#}_{3V}) = LB^{\#}_{3V}$ and thus $LB^{\#}_1 \leq LB^{\#}_{3V}$.
\end{proof}

\begin{lemma} \label{lem:LB1_geq_LB3V}
	$LB^{\#}_1 \geq LB^{\#}_{3V}$.	
\end{lemma}
\begin{proof}
	Let $(\bar{x},\bar{y},\bar{z})$ be an optimal solution of 
	\myref{LF1}($LB^{\#}_1$), so $\bar{z} = LB^{\#}_1$  holds because 
	$\mathcal{L}_1(LB^{\#}_1) = LB^{\#}_1$ by definition.	
	First, consider an arbitrary but fixed $i\in I$.
	Then, because of~\eqref{pac1lb:lift}, we have that 
	\begin{subequations}
					\label{eq:opt_w}
	\begin{alignat}{3}
		\bar{z} \geq \max~
		 &\sum_{j\in J} w_{ij} \bar{x}_{ij}\\ 
		 \st~ &\sum_{j\in A} w_{ij} \leq \max \{LB^{\#}_1, d_{iA}\} \qquad && \forall A\in \mathcal{J}^\alpha
	\end{alignat}
	\end{subequations}	
	holds.
	This
	is equivalent to
		\begin{subequations}
						\label{eq:opt_uv}
	\begin{alignat}{3}
\bar{z} \geq \max~ &u_i + \sum_{j\in J} v_{ij} \bar{x}_{ij}\\ 
\st~ &u_i + \sum_{j\in A} v_{ij} \leq \max \{LB^{\#}_1, d_{iA}\} \qquad && \forall A\in \mathcal{J}^\alpha,		\label{eq:opt_uv_con_max}
	\end{alignat}
	\end{subequations}	
	because each feasible solution $(\hat{u}_i,\hat{v}_{i})$ of the LP on the 
	right-hand side of~\eqref{eq:opt_uv} can be transformed into a feasible solution 
	$\hat{w}_i$ of the LP on the right-hand side of~\eqref{eq:opt_w} with 
	$\hat{w}_{ij} = \frac{1}{\alpha}\hat{u}_i + \hat{v}_{ij}$ for all $j\in J$ with the same objective function value
	\begin{align*}
		\sum_{j\in J} \hat{w}_{ij} \bar{x}_{ij} = 
		\frac{1}{\alpha}\hat{u}_i \sum_{j\in J} \bar{x}_{ij}+  \sum_{j\in J} \hat{v}_{ij} \bar{x}_{ij}=
		\hat{u}_i + \sum_{j\in J} \hat{v}_{ij} \bar{x}_{ij}
	\end{align*}
	because of~\eqref{pac1lb:sumx}. Furthermore, each feasible solution $\tilde{w}_i$ 
	of the LP in~\eqref{eq:opt_w} can be transformed into a feasible solution 
	$(\tilde{u}_i, \tilde{v}_i)$ of the LP in~\eqref{eq:opt_uv} with 
	$\tilde{u}_i = 0$ and $\tilde{v}_{ij} = \tilde{w}_{ij}$ for all $j\in J$, 
	again with the same objective function value. 
	Thus, the LPs in~\eqref{eq:opt_w} and~\eqref{eq:opt_uv} are indeed equivalent.
	
	The LP in~\eqref{eq:opt_uv} is feasible, because $u'_i = 0$ and 
	$v'_{ij} = \frac{1}{\alpha}LB^{\#}_1$ for all $j\in J$ is a feasible solution and clearly it is bounded by $\bar{z}$. 
	\iffalse	
	Furthermore, 
	it is bounded due to analogous arguments as the ones in the proof of 
	Lemma~\ref{lem:d1>=d3-v} for the boundedness of ($D^i$).  In particular, for a 
	feasible solution $(\hat{u}_i,\hat{v}_{i})$ of the LP in~\eqref{eq:opt_uv} we let 
	$\hat{A} \in \arg\max \left\{\sum_{j\in A} \hat{v}_{ij}: A\in \mathcal{J}^\alpha \right\}$ 
	and let $\hat{x}_{i} = (\hat{x}_{ij})_{j\in J}$ with $\hat{x}_{ij} = 1$ for 
	all~$j\in \hat{A}$ and~$\hat{x}_{ij} = 0$ for all $j\in J\setminus \hat{A}$. 
	Hence, $\sum_{j\in J} \hat{x}_{ij} = \alpha$ and
	therefore $\sum_{j\in \hat{A}} \hat{v}_{ij} = 
	\sum_{j\in J} \hat{x}_{ij} \hat{v}_{ij} 
	\geq \sum_{j\in J} \bar{x}_{ij} \hat{v}_{ij}$ holds, because on the left and 
	right hand-side of the inequality the weight~$\alpha$ is distributed to the 
	largest~$\hat{v}_{ij}$ and any~$\hat{v}_{ij}$, respectively.
	Then, $\max\{LB^{\#}_1, d_{iA}\} \geq \hat{u}_i + \sum_{j\in \hat{A}} \hat{v}_{ij} \geq \hat{u}_i + \sum_{j\in J} \bar{x}_{ij}\hat{v}_{ij}$ holds by~\eqref{eq:opt_uv_con_max} 
	and thus the LP in~\eqref{eq:opt_uv} is bounded by $\max_{i'\in I}\max_{A \in \mathcal{J}^\alpha} \{ d_{i'A} \}$.
	\fi
	As a consequence, the LP in~\eqref{eq:opt_uv} has an optimal solution, and the optimal objective function value of it and its dual LP coincide by strong duality. When introducing dual variables $x_{iA}$ for the LP in~\eqref{eq:opt_uv} this boils down to 		
	\begin{subequations}
		\label{eq:opt_uv_dual}
		\begin{alignat}{3}
		\bar{z} \geq \min~  \sum_{A\in \mathcal{J}^\alpha } \max \{&LB^{\#}_1, d_{iA}\} x_{iA}\\ 
	\st~ \sum_{A\in \mathcal{J}^\alpha} x_{iA} &= 1 \\		
		\sum_{\substack{A\in \mathcal{J}^\alpha:j\in A}} x_{iA} &= \bar{x}_{ij}
\qquad &&\forall j\in J \label{eq:opt_uv_dual:sumxiA} \\
			 x_{iA} &\geq 0 \qquad&& \forall A\in \mathcal{J}^\alpha.
	\end{alignat}
	\end{subequations}	
	Now let $\tilde{x}_i = (\tilde{x}_{iA})_{A\in \mathcal{J}^\alpha}$ be an optimal solution of the optimization problem in~\eqref{eq:opt_uv_dual}
	and let $\tilde{x} = (\tilde{x}_i)_{i \in I}$. Furthermore, let $\tilde{y} = \bar{y}$ and $\tilde{z} = \bar{z} = LB^{\#}_1$. Then  $(\tilde{x},\tilde{y},\tilde{z})$ is a feasible solution of \myref{F3-V}($LB^{\#}_{1}$), because 	
	\eqref{pac3vlb:valid} follows from~\eqref{eq:opt_uv_dual:sumxiA} and~\eqref{pac1lb:xy},
	\eqref{pac3vlb:l_opt} is a consequence of~\eqref{eq:opt_uv_dual},
	and all of 
\eqref{pac3vlb:sumy},
\eqref{pac3vlb:sumx} and 
\eqref{pac3vlb:xbin}-\eqref{pac3vlb:zbin} are easy to see. 	
	Therefore, $\mathcal{L}_{3V}(LB^{\#}_{1}) \leq LB^{\#}_{1}$. Since $\mathcal{L}_{3V}(LB^{\#}_{1}) \geq LB^{\#}_{1}$ by Proposition~\ref{prop:lift_d3_v_r}, we have $\mathcal{L}_{3V}(LB^{\#}_{1}) = LB^{\#}_{1}$ and thus $LB^{\#}_{3V} \leq LB^{\#}_{1}$.
\end{proof}

We close this section by the following theorem stating that the two best lower 
bounds based on our lifted inequalities for~\myref{F1} 
and~\myref{F3-V} coincide and are better than the best lower bound based 
on our 
lifted inequalities for~\myref{F3}.

\begin{theorem}
	$LB^{\#}_1 = LB^{\#}_{3V} \geq LB^{\#}_{3}$.
\end{theorem}
\begin{proof}
	This is a direct consequence of Lemma~\ref{lem:LB1_leq_LB3V}, Lemma~\ref{lem:LB1_geq_LB3V} and Theorem~\ref{theo:lift_bounds_d3}.
\end{proof}

\section{Solution algorithm}\label{sec:implementation}

In this section, we describe our solution algorithm for the \paCP, which is a 
branch-and-cut algorithm (\BC) based on our MIP formulation~\myref{F1}. 
We use CPLEX as MIP solver and enhance it with various ingredients such as 
variable fixing, optimality-preserving inequalities and lifted inequalities, as 
well as branching priorities. Implementation details of these enhancements, 
including the separation scheme for the inequalities, are provided in Section 
\ref{sec:b_and_c}. Moreover, we also use a starting heuristic and a primal 
heuristic, which are described in Section \ref{sec:heuristics}. In our 
computational study in Section~\ref{sec:computational} we investigate various 
settings that combine these ingredients. Note, that we also did some 
initial experiments with a solution algorithm based on our MIP formulation~\myref{F3-V}, 
however due to the exponential number of variables, the performance of this 
algorithm was quite poor even for small-scale instances, thus for the sake of 
brevity we do not discuss such an approach further in this work.

\subsection{Branch-and-cut algorithm}\label{sec:b_and_c}

To implement our enhancements, we use the \texttt{UserCutCallback} of CPLEX, 
which is called after fractional solutions are obtained when solving the 
LP-relaxation of a node of the \BC tree,   
and the\linebreak \texttt{LazyConstraintCallback} of CPLEX, which is called for each integer 
solution that is found during the \BC\ (i.e., integer solutions of 
LP-relaxations and solutions obtained by the internal heuristics of CPLEX).
In the following, we describe the details of our variable fixing, separation schemes and the used branching priorities. 

\paragraph{Initial variable fixing}
At the initialization, we fix assignment variables~$x_{ij} = 0$ 
based on the remoteness equalities~\eqref{pac1:optpres_farthest} in any case 
and based 
on the upper bound equalities~\eqref{pac1:optpres2} if an upper bound~$UB$ on 
$\nu(\text{\paCP})$ is known. 
We also omit the linking constraints~\eqref{pac1:xy} 
for all~$i\in I$ and $j\in J$ for which~$x_{ij} = 0$ was fixed, as they are redundant.

\paragraph{Overall separation scheme}
We separate (in)equalities within the \texttt{UserCutCallback} in 
four steps in the order 
\begin{enumerate}[itemsep=0cm]
	\item Upper bound equalities~\eqref{pac1:optpres2} \label{cb1}
	\item Linking constraints~\eqref{pac1:xy} \label{cb2}
	\item Simple upper bound inequalities~\eqref{pac1:optpres3} 
	and general upper bound inequalities~\eqref{pac1:optpres_all} 
	\label{cb3} 
	\item Extended lifted inequalities~\eqref{pac1:lift_ineq_opt},
	\label{cb4}
\end{enumerate}
and in settings where we separate the linking 
constraints~\eqref{pac1:xy}, we additionally separate~\eqref{pac1:xy} within the \texttt{LazyConstraintCallback}, as they 
are necessary for the correct modeling of the \paCP.

In each iteration of the \texttt{UserCutCallback}, we ensure that at most one 
inequality is added per customer~$i\in I$ by ignoring this customer for 
the subsequent separation procedures in this round after an 
(in)equality involving $i$ was added, i.e., no inequalities involving $i$ are 
considered in subsequent separation procedures in this round.
In order to achieve a high node-throughput within the \BC tree, we separate the 
extended lifted inequalities~\eqref{pac1:lift_ineq_opt} only at the root 
node, as they have the most time consuming separation procedure. We also 
restrict the total number of linking constraints, simple upper 
bound inequalities and general upper bound inequalities added in one separation 
round 
in a non-root 
node of the \BC tree to at most \texttt{maxNumCutsTree}.

To allow for tailing-off control of our separation scheme, we restrict the number of separation rounds in each node by \texttt{maxNumSepRoot} if it is the root node, and \texttt{maxNumSepTree} if it is another node in the \BC tree. In addition, we also prohibit further separation rounds within a node, if the lower bound is not improved by more than 0.0001 for~\texttt{maxNoImprovements} consecutive separation rounds.

The following paragraphs specify the separation schemes for the different (in)equality types further.

\paragraph{Separation step \ref{cb1} (Upper bound equalities 
\eqref{pac1:optpres2})}

Let $UB$ be the currently best upper bound.  
If a new improved upper bound~$UB^\prime < UB$ is found during the \BC, we 
separate the upper bound equalities~\eqref{pac1:optpres2} for this new 
bound~$UB^\prime$  by enumeration in the subsequent \texttt{UserCutCallback}.

\paragraph{Separation step \ref{cb2} (Linking constraints~\eqref{pac1:xy})}

In settings where we separate the linking 
constraints~\eqref{pac1:xy}, we do not add all linking constraints at 
the 
initialization, but only add a fixed 
number, specified by the parameter 
\texttt{numInitialCuts}, of them for each customer~$i\in I$.
We do that by 
ordering the potential facility locations~$j\in J$ according to their 
distance to~$i$ in 
increasing order for each~$i\in I$ and adding the linking 
constraint~\eqref{pac1:xy} for the first \texttt{numInitialCuts} facilities 
in that order, i.e., for the \texttt{numInitialCuts} closest facilities 
to~$i$.
For separating~\eqref{pac1:xy} within the \texttt{UserCutCallback} and 
\texttt{LazyConstraintCallback},\sloppy~ we go through all~$i\in I$ in random 
order. 
For each~$i\in I$, we again order the facilities~$j\in J$ by their distance to~$i$ 
in increasing order. Finally, we add~\eqref{pac1:xy} for the first facility for 
which the current solution violates it.

\paragraph{Separation step \ref{cb3} (Simple upper bound 
inequalities~\eqref{pac1:optpres3} 
	and general upper bound inequalities~\eqref{pac1:optpres_all})}
Let $UB$ be the currently best upper bound. 
If a new improved upper bound~$UB^\prime < UB$ is found during the \BC, we 
separate the simple upper bound 
inequalities~\eqref{pac1:optpres3} and the general upper bound 
inequalities~\eqref{pac1:optpres_all} for this new 
bound~$UB^\prime$ in the subsequent \texttt{UserCutCallback}. This separation  
is done by enumeration.
We start with separating~\eqref{pac1:optpres3} (which are a special case 
of~\eqref{pac1:optpres_all}) for~$\beta \in \{2, \ldots, \alpha\}$, as the 
condition for these inequalities is easier to check. The separation of~\eqref{pac1:optpres_all} 
for~$\beta \in \{2, \ldots, \alpha\}$ is then done afterwards, which means it is 
called less often, because, as described above, we do not consider inequalities 
involving a customer~$i$ for which we have already added any violated inequality 
in this round, and we additionally have the limit \texttt{maxNumCutsTree} in 
non-root nodes. We do not consider~\eqref{pac1:optpres3} 
and~\eqref{pac1:optpres_all} for~$\beta = 1$ because those inequalities correspond 
to the previously separated upper bound 
equalities~\eqref{pac1:optpres2}.  

\paragraph{Separation step \ref{cb4} (Extended lifted 
inequalities~\eqref{pac1:lift_ineq_opt})}
Let~$(x^*, y^*, z^*)$ be the current fractional solution of the relaxation,~$UB$ be the current upper bound, and~$LB'$ be the current lower bound on~$\nu(\text{\paCP})$.
Before separating the extended lifted 
inequalities~\eqref{pac1:lift_ineq_opt}, we improve the lower 
bound~$LB'$ by increasing it to the next $\alpha$-distance in $D^\alpha$, 
denoted by~$LB$, as described in the proof of Theorem~\ref{theo:polytime_LB3} and 
add the inequality~$z\geq LB$ to the model. 
For a given~$i\in I$, we then separate the extended lifted 
inequalities~\eqref{pac1:lift_ineq_opt} as follows. 
First, we restrict the set $S_i(LB,UB)$ 
to the non-negative vectors.
Next, for a given $x_i^*$, we consider the support $\supp(x_i^*)=\{j \in J: 
x_{ij}>0\}$. 
We observe that for all~$j \in J \setminus \supp(x_i^*)$ the variable~$w_{ij}$ 
does not contribute to the sum $\sum_{j\in J} w_{ij} x_{ij}^*$ on the 
left-hand side of~\eqref{pac1:lift_ineq_opt} and can therefore be omitted in the 
separation. We thus consider the restricted set
{\fontsize{9.7pt}{9.7pt}\selectfont
\begin{align*}
	S^+_i(LB,UB) = \left\{ w_i\in \mathbb{R}^{|\supp(x^*_i)|}_{\geq 0}: \sum_{j\in A \cap \supp(x^*_i)} w_{ij} \leq \max\{LB, d_{iA}\}
	\forall A\in \mathcal{J}^\alpha: d_{iA} \leq UB, A\cap \supp(x^*_i)\neq \emptyset \right\}
\end{align*}
}
for separation.
The separation problem for a customer $i \in I$ consists of solving the LP 
\begin{equation}
\max_{w_i \in S_i^+(LB,UB) } \quad \sum_{j\in \supp(x_i^*)} w_{ij} x_{ij}^*  \label{wi:obj}.
\end{equation}

Note, that in general, the LP~\eqref{wi:obj} can be unbounded if there is a 
facility~$j\in \supp(x^*_i)$ such that $d_{iA} > UB$ for all $A\in 
\mathcal{J}^\alpha$ with $j\in A$. 
However, 
such a facility~$j$ also fulfills the condition of the variable 
fixing equalities~\eqref{pac1:optpres2} in separation step~\ref{cb1}. We 
thus have already added a variable fixing 
equality~\eqref{pac1:optpres2} involving
customer $i$ in this separation round in step~\ref{cb1} and do not 
consider it here anymore.
We use the primal simplex algorithm of CPLEX to solve \eqref{wi:obj}, as this turned out to be beneficial in preliminary computations.

An optimal solution $w_i^*$ of \eqref{wi:obj} gives an 
inequality~\eqref{pac1:lift_ineq_opt} with maximum left-hand-side for the current 
$x^*_i$ for any value of  $w_{ij}$ for the remaining $j \in J \setminus 
\supp(x^*_i)$. So if $\sum_{j \in \supp(x^*_i)} w^*_{ij} x^*_{ij}\leq 
z^*$ holds, we do not obtain a violated inequality~\eqref{pac1:lift_ineq_opt} 
for this customer~$i$ and hence we continue the separation 
with the next customer. However, if $\sum_{j \in \supp(x^*_i)} w^*_{ij} 
x^*_{ij}> z^*$ holds, the inequality~\eqref{pac1:lift_ineq_opt} induced by $w^*_i$ 
is violated by the current solution $x^*$ for any completion of $w_i^*$ and 
we thus can add it. In particular, given such an $w^*_i$, we construct a 
coefficient vector $\bar w_i$ for~\eqref{pac1:lift_ineq_opt}, which we 
initialize with $\bar w_{ij}=w^*_{ij}$ for $j \in \supp(x_i^*)$ and $\bar 
w_{ij}=0$ for $j \in J \setminus \supp(x_i^*)$. We then try to increase the values 
of $\bar w_{ij}$ for $j \in J \setminus \supp(x^*_i)$ in a greedy fashion as 
follows:
Let $r=|J\setminus \supp(x_i^*)|$ and let $j_1$, \dots, $j_r$ be the 
facilities in $J\setminus \supp(x_i^*)$ sorted by their distance to~$i$, 
so~$J\setminus \supp(x_i^*) = \{j_1, \ldots, j_r\}$ with~$d_{ij_1} \leq \ldots 
\leq d_{ij_r}$. Then, we recursively update~$\bar{w}_{ij_\ell}$ starting 
with~$\ell = 1$ and setting 
\[
\bar{w}_{ij_\ell} = \min_{\substack{B\subseteq J\\ |B| = \alpha-1}} \max \{LB, 
d_{iB} + d_{ij_\ell}\} - \sum_{j'\in B} \bar{w}_{ij'}.
\]
The idea behind this approach is that large values for these remaining 
entries of $\bar w_{ij_\ell}$ result in stronger inequalities, thus we set them as large as possible such that the resulting $\bar{w}_i$ still lies in~$S^+_i(LB,UB)$.
Finally, we add the inequality~$\sum_{j\in J} \bar{w}_{ij} x_{ij} \leq z$ to the 
model using the CPLEX option \texttt{purgeable}, as it may become redundant in 
later iterations considering better lower bounds, and with this option CPLEX can 
remove the inequality again, if it does not consider it useful.

As solving \eqref{wi:obj} can become time consuming for larger instances, we do 
not consider the separation of the lifted 
inequalities~\eqref{pac1:lift_ineq_opt} for all customers 
$i \in \cus$, but only for at most \texttt{numLiftedCustomers}. 
These customers are selected among all customers, for which no variable 
fixing, linking constraint or optimality-preserving inequality was added 
in steps~\ref{cb1} to~\ref{cb3} during the current separation round.
Let 
$I^{sep}$ be the set of these customers.
For each $i\in I^{sep}$ we calculate the following score:
We sort the potential facility locations by their distance to $i$, i.e. 
$J=\{d_{ij_1}, \ldots, d_{ij_m}\}$ with $d_{ij_1}\leq \ldots \leq d_{ij_m}$. 
Let $1\leq\ell\leq m$ be such that $\sum_{k<\ell} y^*_{ij_k} < \alpha$ 
and $\sum_{k\leq \ell} y^*_{ij_k} \geq \alpha$. Then, we define
\[
	\bar{x}_{ij_k} = \left\{ \begin{array}{ll}
		y^*_{j_k} & \textrm{ if } k < \ell \\
		\alpha - \sum_{k < \ell} y^*_{j_k} & \textrm{ if } k = \ell \\
		0 & \textrm{ if } k \geq \ell,
	\end{array}\right.
\]
which represents the closest assignment from $i$ to potential facility 
locations under the current $y^*$. Note, that~$x^*$ may violate these closest 
assignment constraints.
We calculate~$\bar{x}_{ij}$ for all $i \in I^{sep}$ and $j \in J$ choose 
those (at most) \texttt{numLiftedCustomers}~$i\in I^{sep}$ that have the 
largest sums~$\sum_{j\in J} d_{ij} \bar{x}_{ij}$. Using~$\bar{x}_{ij}$ turned out 
to be beneficial in preliminary computations compared to calculating the score 
using $x^*$. A potential explanation for this could be that due to $x^*$ 
not necessarily fulfilling closest assignment constraints, 
directly focusing in the $y$-variables results in a better choice of customers to 
consider for separation, as a high score indicates that even assigning a 
customer to the nearest (fractionally) open facilities results in a large 
$\alpha$-distance for this customer.

\paragraph{Branching} 
The branching is done preferably on the~$y$-variables. This is implemented by setting the branching priorities of the $y$-variables to 100 and the branching priorities of the~$x$- and~$z$-variables to zero. 

\subsection{Heuristics}\label{sec:heuristics}

\paragraph{Starting heuristic}
In order to obtain a starting solution, we use a greedy starting heuristic. 
Let~$P$ be a (partial) solution of the \paCP. First, we randomly pick a 
location~$j\in J$ to open a facility, i.e.,~$P=\{j\}$. Next, we iteratively add 
facilities to~$P$ by the following criterion. For all~$i\in I$, we calculate 
the~$\alpha^\prime$-distance~$d_{\alpha^\prime}(i,P)$ from~$i$ to~$P$, 
where~$\alpha^\prime =\min\{\alpha, |P|\}$. We choose some $i'\in 
I\setminus P$ with the largest~$\alpha^\prime$-distance to~$P$ randomly, 
i.e., a random~$i' \in \arg \max\{d_{\alpha^\prime} (i,P): i\in I\setminus 
P\}$. 
If~$I=J$, we add~$i'$ to~$P$. If~$I\neq J$, we add the closest facility~$j\in 
J\setminus P$ to~$i'$ to~$P$. We stop, once we reach~$|P| = p$. This 
heuristic is run \texttt{numStartHeurRuns} times, with different random starting 
locations~$j\in J$. 

\paragraph{Local search}
We utilize a local search to further improve the solutions found by the 
starting heuristic. 
Let~$P$ be a feasible solution of the \paCP. We search for a better solution 
in the neighborhood~$N(P)$ of $P$ consisting of all solutions that differ 
from~$P$ in exactly one facility. In particular, we check for each open 
facility~$j \in P$ whether there is another facility~$j^\prime \in J\setminus P$ 
such that 
the solution $P^\prime = (P\setminus \{j\}) \cup \{j^\prime\}$ obtained by
replacing~$j$ by~$j^\prime$ in~$P$ improves the objective function value. 
To this end, we sort the customers~$i\in I$ in decreasing order of  
$d_\alpha(i,P)$, i.e., their current~$\alpha$-distance to~$P$. 
Note, that in this ordering, $d_\alpha(i_1,P)$ for the first customer $i_1$ is 
exactly the objective function value $f_\alpha(P)$ of the current 
solution~$P$.
We iterate through them, and 
as soon as a customer $i\in I$ with $d_\alpha(i,P')$ larger or 
equal to $f_\alpha(P)$ is found, we can reject the currently considered 
switch of $j$ and $j'$ and continue our search with a different solution 
in~$N(P)$, i.e., a different $j^\prime$. If no such $i$ exists, the 
solution~$P^\prime$ is a strictly better solution than~$P$, so we update~$P$ and 
set it to $P^\prime$. We then repeat the search with the new~$P$. 
We stop when we find a solution~$P$ that can not be improved by replacing any~$j 
\in P$ with any~$j'\in J\setminus P$. 

This local search heuristic is applied to all the solutions obtained by the starting heuristic. 

\paragraph{Primal heuristic}
We implement a primal heuristic within the \texttt{HeuristicCallback} of CPLEX. In 
this callback, we have the current LP solution $(x^*,y^*,z^*)$ available and 
construct a feasible solution using~$y^*$ by choosing the~$p$ facilities $j$
with the largest values~$y_j^*$ to open. Ties are broken by index.

\section{Computational study}\label{sec:computational}

We implemented our solution algorithm in C++ using CPLEX 22.1.1 as MIP solver. In 
all our runs, we set the branching priorities as described in 
Section~\ref{sec:b_and_c} and the \texttt{relativeMIPGap} parameter of CPLEX 
to~$0.0$.
All other CPLEX parameters are left at their default values. The runs were made on 
a single core of an Xeon E5-2630 machine with 8 cores, 2.4GHz and 64GB of RAM. We 
set a time limit of 1800 seconds.

In our solution algorithm, we use the parameters \texttt{numStartHeurRuns} 
= 10, \texttt{numInitialCuts} = 100, \texttt{maxNumSepRoot} = 150, 
\texttt{maxNoImprovements} = 5, \texttt{maxNumCutsTree} = 50, 
\texttt{maxNumSepTree} = 5, and \texttt{numLiftedCustomers} = 20. These values 
were determined with preliminary computations.

\subsection{Instances}\label{sec:inst_values}

In our computational study, we use the two sets of benchmark instances that were 
(partially) used for the~\pCP in 
\citet{elloumi2004,calik2013double,chen2009,contardo2019scalable} and 
\citet{gaar2022scaleable}, for the~\apCP in 
\citet{mousavi2023exploiting,gaar2023exact} and \citet{chagas2024parallel}, for 
the~\pnCP in \citet{albareda2015centers,lopez2019grasp,londe2021evolutionary} and 
\citet{ristic2023auxiliary}, and for the~\pSCP in \citet{ristic2023solving}.

The first instance set \pmed is based on the OR-library introduced by \citet{beasley1990or}.
We consider all~40 instances, each consisting of a graph and a 
number~$p$. In these instances, we have a single set of customer locations and 
potential facility 
locations, i.e.,~$I=J=V$, where~$V$ are the vertices 
of the graph. To obtain the distances between locations $i,j\in V$, the 
shortest path distance between $i$ and $j$ in the underlying graph 
is used.
For the instances in \pmed, all such distances 
are integer. 

The second instance set \tsplib was originally introduced for the traveling 
salesperson problem in \cite{reinelt1991tsplib}. Each \tsplib instance is 
given by a set of locations $V$ and we again have $I=J=V$.
 The names of the instances contain the 
number of locations defining this instance, e.g., instance 
\texttt{eil101} 
contains~$101$ locations. Each location is given as a 
two-dimensional 
coordinate and the 
distance between two locations is calculated as the Euclidean distance. 
We consider a subset of 8 instances of the \tsplib instances with at least $48$ and up to $439$ locations.
We test~$p\in\{10, 20, 30, 40, 50, 60, 70, 80, 90, 
100\}$ but we do not use
 values of $p$
that are greater or equal to $10+|I|/2$. So we test $p\in 
\{10,20,30,40,50,60,70\}$ for \texttt{bier127} for example.

For both instances sets, we consider $\alpha=2$.

\subsection{Analysis of the branch-and-cut algorithms components}\label{sec:settings}

We analyze the different components of our \BC described in Section~\ref{sec:implementation} by comparing the settings
\begin{itemize}[noitemsep]
	\item \texttt{1}, i.e., directly solving~\myref{F1} without any of our enhancements,
	\item \texttt{1H}, i.e., solving~\myref{F1} with all our heuristics
	as well as the variable fixing based on upper bound 
	equalities~\eqref{pac1:optpres2} 
	and remoteness equalities~\eqref{pac1:optpres_farthest} (step~\ref{cb1} 
	of the separation 
	scheme), 
	\item \texttt{1HS}, i.e., setting~\texttt{1H} with the separation of the 
	linking constraints~\eqref{pac1:xy} (up to step~\ref{cb2} of the 
	separation scheme),
	\item \texttt{1HSL}, i.e., setting~\texttt{1HS} with separation of the 
	simple upper bound inequalities~\eqref{pac1:optpres3} 
	and general upper bound inequalities~\eqref{pac1:optpres_all}, and the 
	extended lifted 
	inequalities~\eqref{pac1:lift_ineq_opt} (all steps of the separation 
	scheme).
\end{itemize}

Figure~\ref{fig:plot_pmed} shows cumulative runtime and optimality 
gap plots 
for the \pmed instances. The optimality 
gap for one run is defined as $(UB-LB)/UB$
where $UB$ and $LB$ are the best upper and lower 
bound obtained with this run, respectively. Figure~\ref{fig:plot_tsp}
shows
the same for the \tsplib instances.

\begin{figure}[!ht]
	\centering
	\begin{subfigure}{0.45\textwidth}
		\includegraphics[width=\linewidth]{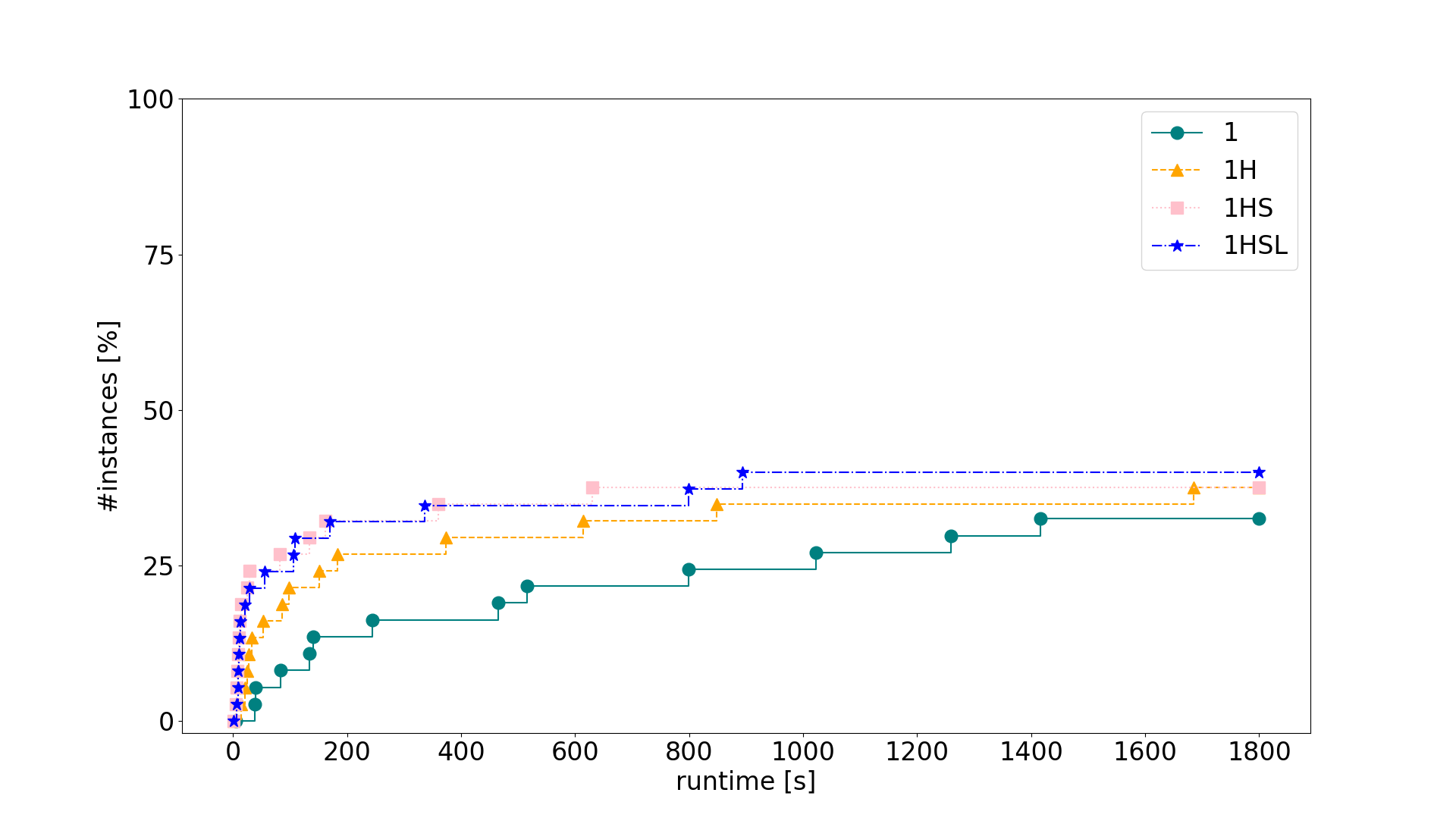}
		\caption{runtime for the \pmed instances}
		\label{fig:run_pmed}
	\end{subfigure}
	\begin{subfigure}{0.45\textwidth}
		\centering
		\includegraphics[width=\linewidth]{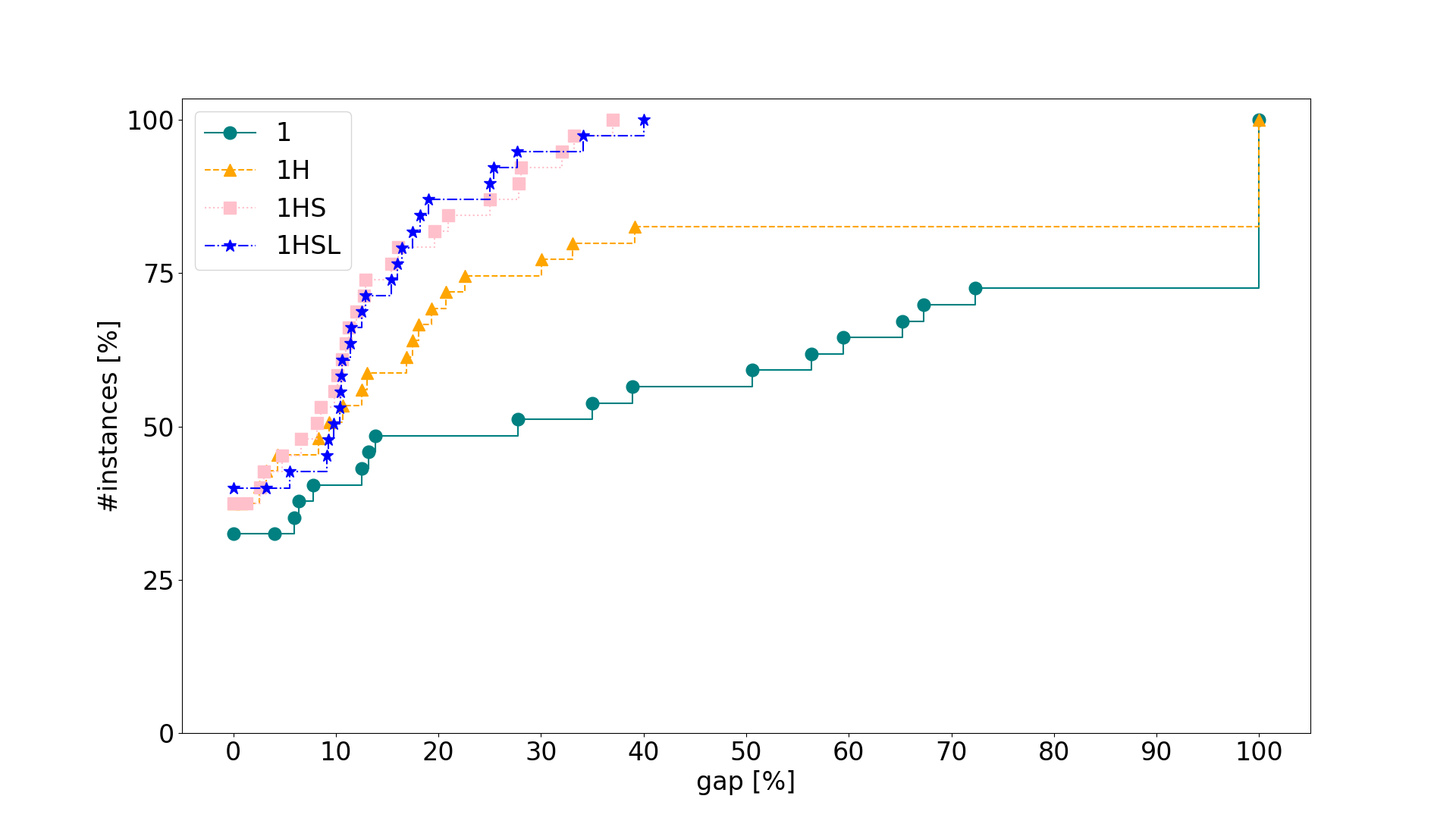}
		\caption{optimality gap for the \pmed instances}
		\label{fig:gap_pmed}
	\end{subfigure}
	\caption{Comparison of our \BC settings for the \pmed instances}
	\label{fig:plot_pmed}
\end{figure}

\begin{figure}[!ht]
	\centering
		\begin{subfigure}{0.45\textwidth}
		\includegraphics[width=\linewidth]{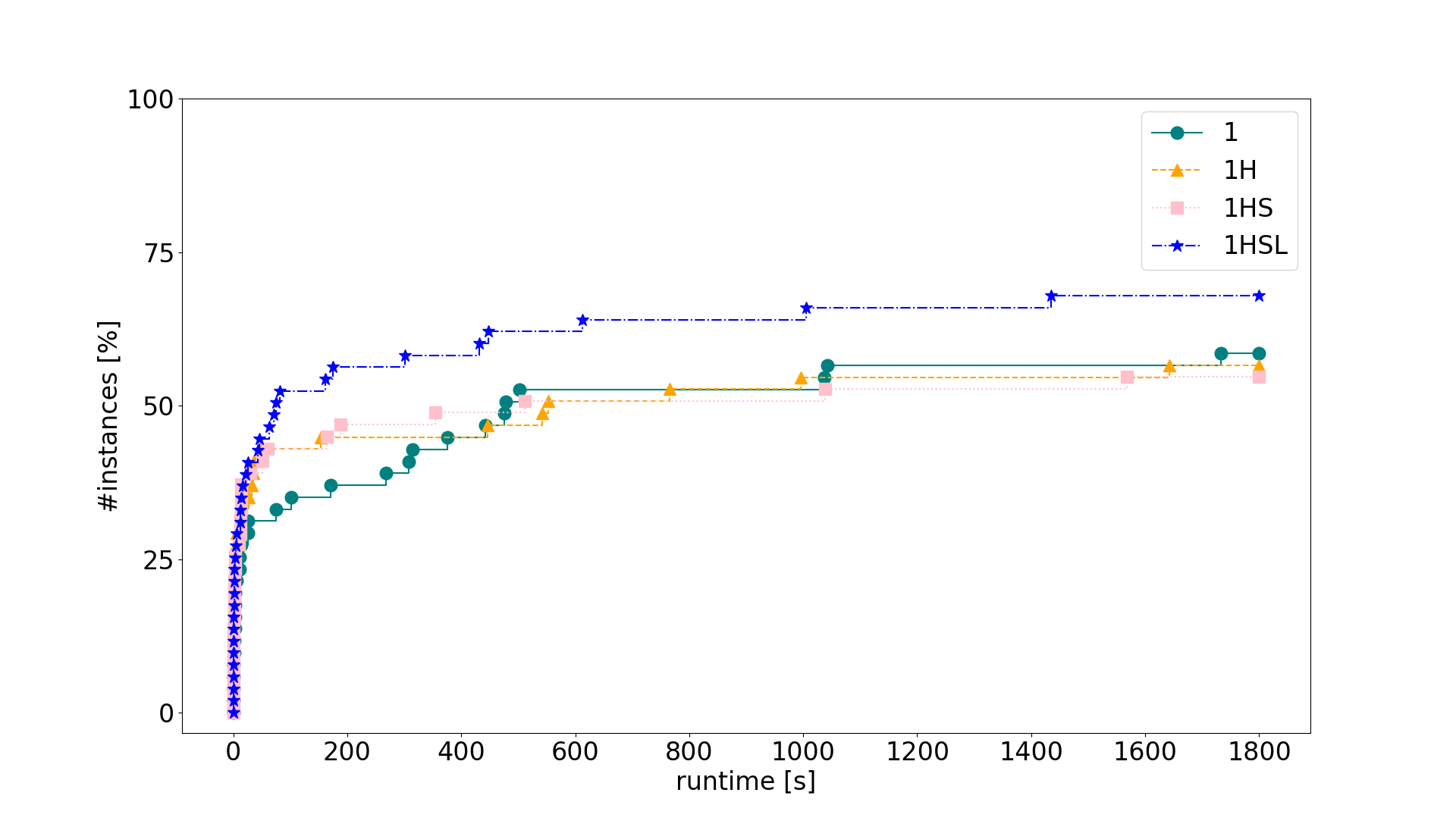}
		\caption{runtime for the \tsplib instances}
		\label{fig:run_tsp}
	\end{subfigure}
	\begin{subfigure}{0.45\textwidth}
		\centering
		\includegraphics[width=\linewidth]{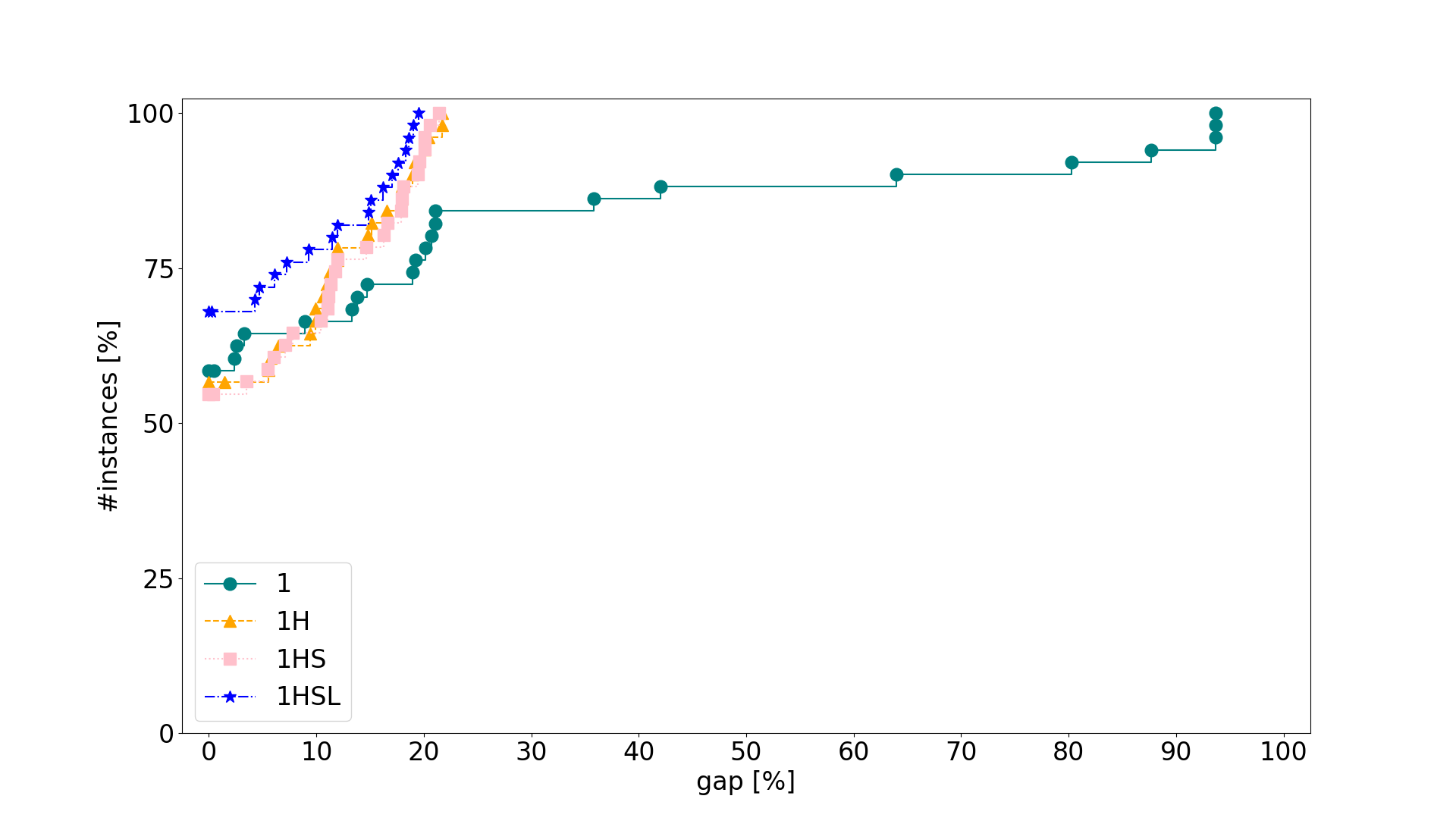}
		\caption{optimality gap for the \tsplib instances}
		\label{fig:gap_tsp}
	\end{subfigure}
	\caption{Comparison of our \BC settings for the \tsplib instances}
	\label{fig:plot_tsp}
\end{figure}

We see that for both instance sets, setting \texttt{1}, i.e., directly 
solving~\myref{F1} without any enhancement, has the worst performance with only 32.5\% of the \pmed instances and 58.5\% of the \tsplib instances solved 
to optimality within the imposed time limit, and 42.5\% of the \pmed 
instances and 11.3\% of the \tsplib instances with a remaining optimality gap of 
over 50\%. 

Adding heuristics and variable fixing with setting \texttt{1H} yields the 
highest improvement in the performance for the \pmed instances and some 
improvement for the \tsplib instances. This can be explained by the fact that 
fixing variables reduces the size of the MIP formulation which makes it 
easier to solve it. This effect is larger for instances with a higher 
reduction potential, i.e., instances with a high number of locations, of 
which the instance set \pmed contains more. 

For the \pmed instances, separating the linking constraints, i.e., setting 
\texttt{1HS}, has a positive effect. Unfortunately, this positive effect is 
neglectable for the \tsplib instances. 
This is likely due to the different sizes of the instances, as 
the instance set \pmed contains larger instances and thus not adding all the 
linking constraints at initialization but separating them has a bigger effect on 
the size of the LP-relaxations. 

Adding the optimality-preserving inequalities and the lifted inequalities, i.e., 
setting \texttt{1HSL}, has a limited effect for the \pmed instances but 
demonstrates a significant improvement for the \tsplib instances, where between 5 and 7 
additional 
instances could be solved with this setting compared to the others. This shows 
that these inequalities do not only strengthen formulation~\myref{F1} 
theoretically, 
but also improve the computational performance. The limited effect on the \pmed 
instances may be explained by the larger instance sizes in this set, which makes 
both solving the LP-relaxations and also solving the separation-LPs for the lifted 
inequalities more time consuming. 
Therefore, adding all enhancements yields the best overall performance with 40.0\% of the \pmed instances and 67.9\% of the \tsplib instances solved to optimality within the time limit and all \pmed instances solved to a 40\% optimality gap and all \tsplib instances solved to an optimality gap of at most 20\%.

In addition, we also consider the quality of root bounds, i.e., the lower 
bound at the end of the root node, for the different settings. 
To this end, we compare the root bound $RB$ of each instance and setting with the 
best upper bound $UB^*$ found for this instance by any setting. The root gap is 
then computed as $(UB^*-RB)/UB^*$. 
We note that the root bounds $RB$ in this calculation are taken from MIP-solving 
runs, in which CPLEX applies internal preprocessing and cut generating procedures.
In Figures~\ref{fig:root_pmed} and~\ref{fig:root_tsp} we 
provide plots of the 
cumulative root gaps for the instances sets \pmed and \tsplib, respectively.
There, we see that adding our 
optimality-preserving and lifted inequalities significantly improve the root 
bounds. 

\begin{figure}[!ht]
	\centering
		\begin{subfigure}{0.45\textwidth}
			\includegraphics[width=\linewidth]{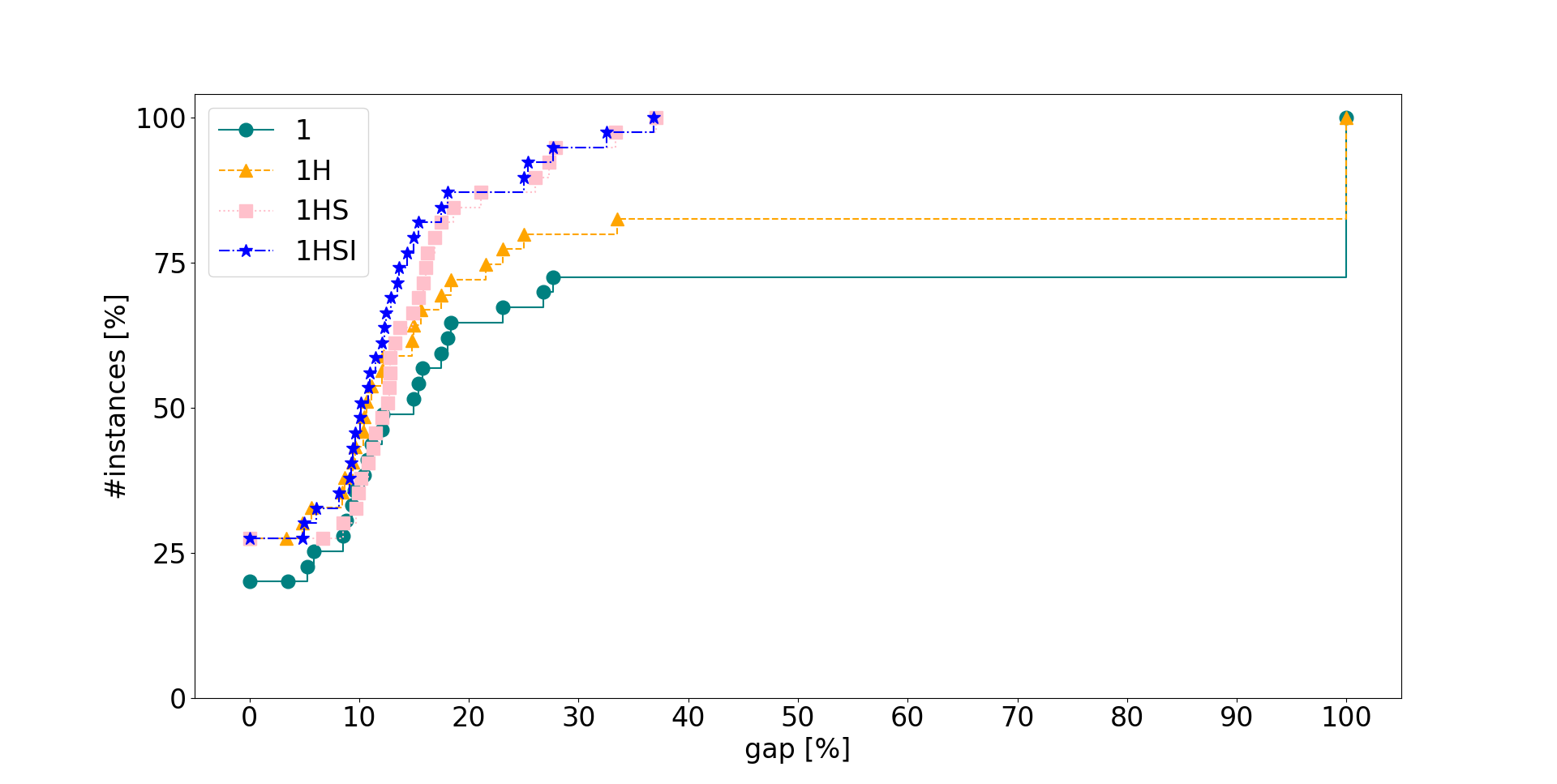}
			\caption{root gap for the \pmed instances}
			\label{fig:root_pmed}
		\end{subfigure}
	\begin{subfigure}{0.45\textwidth}
		\includegraphics[width=\linewidth]{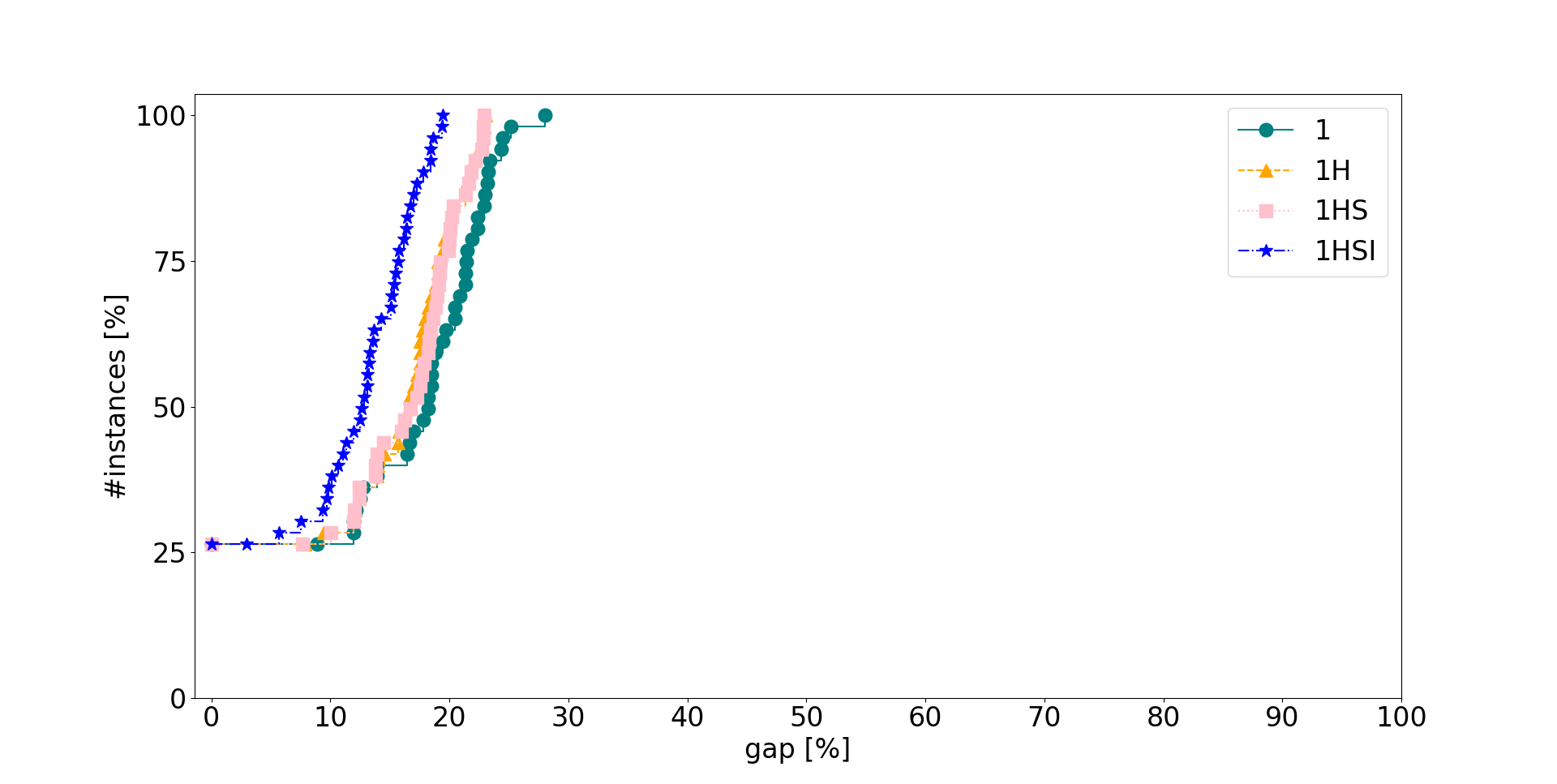}
		\caption{root gap for the \tsplib instances}
		\label{fig:root_tsp}
	\end{subfigure}
	\caption{Comparison of our \BC settings with respect to the root gap for the \pmed and \tsplib instances}
	\label{fig:plot_rootgap}
\end{figure}

These observations are consistent with the results in \citet{gaar2023exact} with the exception of the effect of adding the lifted inequalities. Although this addition has a positive effect also for the~\paCP, the improvement of performance is significantly higher for the~\apCP. 
A reason for this lies in the structure of the lifted inequalities. For the~\apCP, a lifted inequality based on a stronger lower bound supersede a lifted inequality based on a weaker lower bound. 
However, this dominance of single inequalities is not given for the~\paCP. Here, only subsets of lifted inequalities can supersede lifted inequalities based on weaker bounds, which leads to more relevant cuts and a higher complexity. 

\subsection{Detailed results}

We now provide detailed results and compare our 
enhanced branch-and-cut algorithm, i.e., setting \texttt{1HSL}, with the basic 
approach of using CPLEX without any enhancement, i.e., setting~\texttt{1}. 
Moreover, we also include a comparison with the 
heuristic of \citet{ristic2023solving} whenever possible, i.e, for the 
instance set \pmed. 

\begin{table}[h!tb]
	\caption{Results for \texttt{pmed} for $\alpha=2$}	
	\label{tab:runtime_pmed}
	\centering
	\setlength{\tabcolsep}{5.5pt}
	\renewcommand{\arraystretch}{1.1}
	\begin{tabular}{lrr|rrrr|rrrr|rr}\multicolumn{3}{c}{} & \multicolumn{4}{|c}{\texttt{1}} & \multicolumn{4}{|c}{\texttt{1HSL}} & \multicolumn{2}{|c}{VNS} \\\hline
		\multicolumn{1}{l}{instance} & \multicolumn{1}{c}{$n$} & \multicolumn{1}{c}{$p$} & \multicolumn{1}{|c}{$UB$} & \multicolumn{1}{c}{$LB$} & \multicolumn{1}{c}{$t[s]$} & \multicolumn{1}{c}{$\#BB$}& \multicolumn{1}{|c}{$UB$} & \multicolumn{1}{c}{$LB$} & \multicolumn{1}{c}{$t[s]$} & \multicolumn{1}{c}{$\#BB$} & \multicolumn{1}{|c}{$UB$} & \multicolumn{1}{c}{$t[s]$} \\ \hline 
		pmed1 & 100 & 5 & \textbf{268} & \textbf{268.00 }& 141.3 & 1833 & \textbf{268} & \textbf{268.00} & \textbf{56.1} & 947 & 268 & 0.2 \\
		pmed2 & 100 & 10 & \textbf{220} & \textbf{220.00 }& 1023.0 & 7580 &\textbf{ 220} & \textbf{220.00} & \textbf{108.8} & 1974 & 220 & 14.2 \\
		pmed3 & 100 & 10 & \textbf{208} & \textbf{208.00} & 465.9 & 7926 & \textbf{208} & \textbf{208.00} &\textbf{ 170.2} & 4391 & 208 & 1.7 \\
		pmed4 & 100 & 20 & \textbf{163} & \textbf{163.00} & \textbf{515.9} & 11233 & \textbf{163} & \textbf{163.00} & 893.2 & 19036 & 163 & 0.4 \\
		pmed5 & 100 & 33 & \textbf{110} & \textbf{110.00} & \textbf{38.6} & 610 & \textbf{110} & \textbf{110.00} & 106.5 & 6775 & 110 & 0.1 \\
		pmed6 & 200 & 5 & \textbf{180} & \textbf{180.00} & 1259.2 & 1049 & \textbf{180} & \textbf{180.00} & \textbf{799.8} & 5941 & 180 & 0.6 \\
		pmed7 & 200 & 10 & \textbf{145} & \textbf{135.74} & TL & 1373 & \textbf{145} & 129.91 & TL & 9461 & 143 & 11.2 \\
		pmed8 & 200 & 20 & 131 & \textbf{114.64} & TL & 1087 & \textbf{128} & 
		107.00 & TL & 3001 & 122 & 9.3 \\
		pmed9 & 200 & 40 & \textbf{87} & \textbf{80.24} & TL & 2313 & 91 & 74.44 & TL & 5984 & 85 & 0.5 \\
		pmed10 & 200 & 67 & \textbf{70} &\textbf{ 70.00} & 5.6 & 0 & \textbf{70} & \textbf{70.00} & \textbf{0.6 }& 0 & 70 & 0.1 \\
		pmed11 & 300 & 5 & 126 & 120.96 & TL & 487 & \textbf{125} & \textbf{121.00} & TL & 2366 & 125 & 0.8 \\
		pmed12 & 300 & 10 & 122 & \textbf{105.06} & TL & 400 & \textbf{114} & 103.47 & TL & 1924 & 112 & 3.8 \\
		pmed13 & 300 & 30 & \textbf{83} & \textbf{72.03} & TL & 183 & 84 & 68.00 & TL & 1044 & 78 & 3.6 \\
		pmed14 & 300 & 60 & \textbf{60} & \textbf{60.00} & 1416.4 & 1101 & 67 & \textbf{60.00} & TL & 2420 & 60 & 2.6 \\
		pmed15 & 300 & 100 & \textbf{44} & \textbf{44.00} & \textbf{39.1} & 0 & \textbf{44} & \textbf{44.00} & 335.7 & 2857 & 44 & 0.3 \\
		pmed16 & 400 & 5 & 101 & \textbf{95.00} & TL & 0 & \textbf{99} & 93.53 & TL & 1932 & 98 & 1.5 \\
		pmed17 & 400 & 10 & 108 & \textbf{78.00 }& TL & 0 & \textbf{88} & 77.00 & TL & 1029 & 83 & 8.3 \\
		pmed18 & 400 & 40 & 84 &\textbf{ 54.61} & TL & 0 & \textbf{71} & 53.00 & TL & 748 & 62 & 136.4 \\
		pmed19 & 400 & 80 & 123 & \textbf{34.00} & TL & 0 & \textbf{47} & \textbf{34.00 }& TL & 600 & 42 & 25.6 \\
		pmed20 & 400 & 133 & 40 & \textbf{40.00} & 84.1 & 0 & \textbf{40} & \textbf{40.00} & \textbf{9.6} & 0 & 40 & 0.7 \\
		
		pmed21 & 500 & 5 & 126 & 77.00 & TL & 0 & \textbf{88} & \textbf{79.41} & TL & 1458 & 85 & 8.3 \\
		pmed22 & 500 & 10 & 174 & 70.57 & TL & 0 & \textbf{87} &\textbf{ 73.08} & TL & 794 & 80 & 57.0 \\
		pmed23 & 500 & 50 & 118 & 41.00 & TL & 0 & \textbf{56} & \textbf{42.00} & TL & 17 & 49 & 162.2 \\
		pmed24 & 500 & 100 & 101 & \textbf{33.00 }& TL & 0 & \textbf{40} & \textbf{33.00} & TL & 299 & 35 & 11.9 \\
		pmed25 & 500 & 167 & \textbf{44} & \textbf{44.00} & 134.7 & 0 & \textbf{44} & \textbf{44.00} & \textbf{6.1} & 0 & 44 & 1.5 \\
		pmed26 & 600 & 5 & 120 & - & TL & 0 & \textbf{82} & \textbf{73.33} & TL & 1092 & 80 & 100.9 \\
		pmed27 & 600 & 10 & 139 & 60.66 & TL & 0 & \textbf{70} & \textbf{62.00} & TL & 749 & 67 & 66.1 \\
		pmed28 & 600 & 60 & 161 & - & TL & 0 & \textbf{57} & \textbf{57.00} &\textbf{ 13.5} & 0 & 57 & 2.6 \\
		pmed29 & 600 & 120 & \textbf{36} & \textbf{36.00 }& 799.2 & 180 & \textbf{36} & \textbf{36.00} &\textbf{ 29.2} & 0 & 36 & 3.0 \\
		pmed30 & 600 & 200 & \textbf{40} & \textbf{40.00} & 244.0 & 0 & \textbf{40} & \textbf{40.00} & \textbf{8.7 }& 0 & 40 & 2.6 \\
		pmed31 & 700 & 5 & 109 & - & TL & 0 & \textbf{66} & \textbf{59.05 }& TL & 578 & 64 & 6.5 \\
		pmed32 & 700 & 10 & 219 & - & TL & 0 & \textbf{72} & \textbf{72.00} & \textbf{10.5} & 0 & 72 & 4.3 \\
		pmed33 & 700 & 70 & 118 & - & TL & 0 & \textbf{44} & \textbf{29.00} & TL & 1 & 35 & 64.2 \\
		pmed34 & 700 & 140 & 83 & \textbf{41.00} & TL & 0 & \textbf{41} & \textbf{41.00} & \textbf{11.8} & 0 & 41 & 4.4 \\
		pmed35 & 800 & 5 & 89 & - & TL & 0 & \textbf{66} & \textbf{60.00 }& TL & 504 & 64 & 127.7 \\
		pmed36 & 800 & 10 & 161 & - & TL & 0 & \textbf{61} & \textbf{54.00 }& TL & 313 & 58 & 112.0 \\
		pmed37 & 800 & 80 & 119 & - & TL & 0 & \textbf{39} & \textbf{33.00} & TL & 104 & 33 & 208.8 \\
		pmed38 & 900 & 5 & 114 & - & TL & 0 & \textbf{62} & \textbf{54.00} & TL & 343 & 61 & 48.8 \\
		pmed39 & 900 & 10 & 194 & - & TL & 0 & \textbf{74} & \textbf{74.00} & \textbf{21.5} & 0 & 74 & 9.0 \\
		pmed40 & 900 & 90 & 90 & - & TL & 0 & \textbf{40} & \textbf{24.00 }& TL & 0 & 29 & 29.6 \\
		\hline
	\end{tabular}
\end{table}
 
Table~\ref{tab:runtime_pmed} shows the results for the \pmed instances, 
including the runtime in seconds in the columns $t[s]$, with $TL$ 
indicating that the time limit is reached. For the heuristic of 
\citep{ristic2023solving}, we report the average runtime given in 
\citep{ristic2023solving} as runtime, because the heuristic is run multiple 
times in 
their computational study. The objective function values of the best integer 
solutions obtained
are given in the columns $UB$ and the columns $LB$ contain the 
best found lower bounds.
Moreover, 
columns $\#BB$ report the number of nodes in the B\&C tree. 
Bold entries in the table indicate the 
best values for the runtime, the upper bound and the lower bound among the two 
exact approaches. 

We see that for the smaller instances with up to~400 locations the 
performance of~\texttt{1} and~\texttt{1HSL} is comparable. However, starting from 
instance \texttt{pmed19}, the setting~\texttt{1HSL} consistently yields better 
lower bounds and runtimes than setting~\texttt{1}, while providing an at 
least as good upper bound.
An explanation for this could be that for small instances, solving the 
compact MIP formulation directly with CPLEX is more efficient and the 
advantage of fixing variables and separating cuts is less significant than in 
larger instances.   
For the instances with at least~600~locations, CPLEX has trouble solving 
the 
LP-relaxation within the time limit, which is why only few lower bounds are 
reported for these instances. The setting \texttt{1HSL}, however, yields lower 
bounds 
even for the instances with~900~locations since separating the linking 
constraints 
reduces the size of the LP-relaxation extremely. The fact that \texttt{1HSL} also 
finds better upper bounds is due to our starting heuristic which is efficient even 
for large instances.
The heuristic of \citet{ristic2023solving} finds the best solutions among all 
three approaches for all~40~instances, and 
for 17 of them, our exact algorithm proves optimality.  

\begin{table}[h!tb]
	\caption{Results for \texttt{tsplib} for~$\alpha=2$}
	\centering
	\label{tab:runtime_tsp}
	\renewcommand{\arraystretch}{1.2}
	\begin{tabular}{lr|rrrr|rrrr}\multicolumn{2}{c}{} & \multicolumn{4}{|c}{\texttt{1}} & \multicolumn{4}{|c}{\texttt{1HSL}} \\\hline
		\multicolumn{1}{l}{instance} & \multicolumn{1}{c}{$p$} & \multicolumn{1}{|c}{$UB$} & \multicolumn{1}{c}{$LB$} & \multicolumn{1}{c}{$t[s]$} & \multicolumn{1}{c}{$\#BB$}& \multicolumn{1}{|c}{$UB$} & \multicolumn{1}{c}{$LB$} & \multicolumn{1}{c}{$t[s]$} & \multicolumn{1}{c}{$\#BB$} \\ \hline 
		att48 & 10 & \textbf{2827.72} & \textbf{2827.72} & 3.4 & 436 & \textbf{2827.72} & \textbf{2827.72} & \textbf{1.6 }& 191 \\
		att48 & 20 & \textbf{1654.69} & \textbf{1654.69} & 2.6 & 406 & \textbf{1654.69} & \textbf{1654.69} & \textbf{1.3} & 247 \\
		att48 & 30 & \textbf{1203.18} & \textbf{1203.18} & 0.3 & 0 & \textbf{1203.18} & \textbf{1203.18} & \textbf{0.1} & 0 \\
		st70 & 10 & \textbf{48.24} & \textbf{48.24} & 25.7 & 2033 & \textbf{48.24} & \textbf{48.24} & \textbf{22.2} & 1955 \\
		st70 & 20 & \textbf{30.59} & \textbf{30.59} & 25.2 & 4661 & \textbf{30.59} & \textbf{30.59} & \textbf{11.9} & 4711 \\
		st70 & 30 & \textbf{22.88} & \textbf{22.88} & \textbf{3.5} & 1254 & \textbf{22.88} & \textbf{22.88} & 12.3 & 7463 \\
		st70 & 40 & \textbf{19.70} & \textbf{19.70} & 0.5 & 0 & \textbf{19.70} & \textbf{19.70} & \textbf{0.1} & 0 \\
		rd100 & 10 & \textbf{484.87} & \textbf{484.87} & 478.2 & 11046 & \textbf{484.87} & \textbf{484.87} & \textbf{75.3} & 3327 \\
		rd100 & 20 & \textbf{325.05} & \textbf{325.05} & \textbf{101.6} & 8126 & \textbf{325.05} & \textbf{325.05} & 160.9 & 9050 \\
		rd100 & 30 & \textbf{264.83} & \textbf{264.83} & 1733.0 & 95245 & \textbf{264.83} & \textbf{264.83} & \textbf{432.2} & 33791 \\
		rd100 & 40 & \textbf{211.48} & \textbf{211.48} & 75.7 & 9713 & \textbf{211.48} & \textbf{211.48} & \textbf{13.8} & 2109 \\
		rd100 & 50 & \textbf{174.70 }&\textbf{ 174.70 }& 11.7 & 1312 &\textbf{ 174.70} &\textbf{ 174.70} & \textbf{5.9} & 1618 \\
		eil101 & 10 & \textbf{34.09} & \textbf{34.09 }& 315.1 & 8062 & \textbf{34.09} & \textbf{34.09} & \textbf{175.3} & 3111 \\
		eil101 & 20 & 23.09 & 22.55 & TL & 30197 & \textbf{22.66} & \textbf{22.66} &\textbf{ 448.1} & 15950 \\
		eil101 & 30 & \textbf{18.30} & \textbf{18.30} & 1036.7 & 69831 & \textbf{18.30} & \textbf{18.30} &\textbf{ 82.3} & 5521 \\
		eil101 & 40 & \textbf{16.02} &\textbf{ 16.02} & 308.1 & 19788 & \textbf{16.02} & \textbf{16.02} & \textbf{71.0} & 5452 \\
		eil101 & 50 & \textbf{14.47} & \textbf{14.47} & 11.1 & 2730 & \textbf{14.47} & \textbf{14.47} & \textbf{25.4} & 5000 \\
		eil101 & 60 & \textbf{12.73} & \textbf{12.73} & 3.2 & 1130 & \textbf{12.73} & \textbf{12.73} &\textbf{ 0.7} & 28 \\
		bier127 & 10 & \textbf{7717.43} & \textbf{7717.43} & 442.3 & 1469 &\textbf{ 7717.43} & \textbf{7717.43} & \textbf{46.3} & 277 \\
		bier127 & 20 & \textbf{6078.67 }& \textbf{6078.67} & 5.6 & 0 & \textbf{6078.67} &\textbf{ 6078.67} & \textbf{0.2} & 0 \\
		bier127 & 30 &\textbf{ 6078.67} & \textbf{6078.67} & 2.0 & 0 & \textbf{6078.67} & \textbf{6078.67} & \textbf{0.2} & 0 \\
		bier127 & 40 & \textbf{6078.67} & \textbf{6078.67} & 1.2 & 0 & \textbf{6078.67} & \textbf{6078.67} & \textbf{0.2} & 0 \\
		bier127 & 50 &\textbf{ 6078.67} & \textbf{6078.67} & 1.1 & 0 & \textbf{6078.67} & \textbf{6078.67} & \textbf{0.2} & 0 \\
		bier127 & 60 &\textbf{ 6078.67} & \textbf{6078.67} & 1.0 & 0 & \textbf{6078.67} & \textbf{6078.67} & \textbf{0.2} & 0 \\
		bier127 & 70 &\textbf{ 6078.67} & \textbf{6078.67} & 0.9 & 0 & \textbf{6078.67} & \textbf{6078.67} & \textbf{0.3} & 0 \\
		ch150 & 10 & 325.74 & 296.74 & TL & 15122 & \textbf{324.87} & \textbf{301.47} & TL & 17954 \\
		ch150 & 20 & \textbf{222.38} & 189.57 & TL & 26589 & 222.63 & \textbf{197.07} & TL & 28722 \\
		ch150 & 30 & 178.99 & 155.17 & TL & 46156 & \textbf{173.46} & \textbf{162.89} & TL & 29864 \\
		ch150 & 40 & \textbf{148.53} & \textbf{148.53} & 1042.3 & 25985 & \textbf{148.53} & \textbf{148.53} & \textbf{1005.8} & 20655 \\
		ch150 & 50 & \textbf{130.62} & \textbf{130.62} & \textbf{501.7} & 20720 & \textbf{130.62} & 130.62 & 613.2 & 41493 \\
		ch150 & 60 & \textbf{116.02} & \textbf{112.24} & TL & 120834 & 117.17 & 111.68 & TL & 166317 \\
		ch150 & 70 &\textbf{ 106.52} & \textbf{106.52} & \textbf{170.6} & 24266 & \textbf{106.52} & \textbf{106.52} & 301.2 & 46373 \\
		ch150 & 80 & \textbf{95.14} & 95.14 & \textbf{13.3} & 1075 & \textbf{95.14} & \textbf{95.14} & 16.5 & 2629 \\
		\hline
	\end{tabular}
\end{table}

\begin{table}[h!tb]
	\caption{Results for \texttt{tsplib} for $\alpha=2$ continued}	
	\label{tab:runtime_tsp2}
	\centering
	\renewcommand{\arraystretch}{1.35}
	\begin{tabular}{lr|rrrr|rrrr}\multicolumn{2}{c}{} & \multicolumn{4}{|c}{\texttt{1}} & \multicolumn{4}{|c}{\texttt{1HSL}} \\\hline
		\multicolumn{1}{l}{instance} & \multicolumn{1}{c}{$p$} & \multicolumn{1}{|c}{$UB$} & \multicolumn{1}{c}{$LB$} & $t[s]$ & \multicolumn{1}{c}{$\#BB$}& \multicolumn{1}{|c}{$UB$} & \multicolumn{1}{c}{$LB$} & \multicolumn{1}{c}{$t[s]$} & \multicolumn{1}{c}{$\#BB$} \\ \hline 
		rat195 & 10 & 101.27 & 87.29 & TL & 9510 & \textbf{97.48} & \textbf{88.40} & TL & 7205 \\
		rat195 & 20 & 69.06 & 55.16 & TL & 15933 & \textbf{66.63} & \textbf{56.59} & TL & 11200 \\
		rat195 & 30 & 54.14 & 42.91 & TL & 22283 & \textbf{53.88 }& \textbf{43.87} & TL & 17800 \\
		rat195 & 40 & \textbf{45.78} & 36.13 & TL & 27317 & 46.28 & \textbf{37.24} & TL & 31898 \\
		rat195 & 50 & \textbf{40.04} & 32.35 & TL & 52637 & 40.62 & \textbf{33.47} & TL & 38500 \\
		rat195 & 60 & \textbf{36.50} & 28.81 & TL & 36456 & \textbf{36.50} & \textbf{31.08} & TL & 44713 \\
		rat195 & 70 & 33.63 & 27.26 & TL & 37921 & \textbf{32.41} & \textbf{31.04} & TL & 49957 \\
		rat195 & 80 & \textbf{30.42} & 30.27 & TL & 60263 & \textbf{30.42} & \textbf{30.42} & \textbf{1434.5} & 71278 \\
		rat195 & 90 & \textbf{28.61} & 27.86 & TL & 74364 & \textbf{28.61} & \textbf{28.54} & TL & 107564 \\
		rat195 & 100 & \textbf{27.41} & \textbf{27.41} & 267.9 & 15067 & \textbf{27.41} & \textbf{27.41} & \textbf{63.6} & 4834 \\
		pr439 & 10 & 7118.44 & \textbf{4124.62 }& TL & 2 & \textbf{4970.02} & 4122.21 & TL & 1089 \\
		pr439 & 20 & 7016.72 & 2530.26 & TL & 10 & \textbf{3140.93} & \textbf{2542.89} & TL & 556 \\
		pr439 & 30 & 2865.09 & 1837.91 & TL & 450 & \textbf{2292.31} & \textbf{1920.39} & TL & 496 \\
		pr439 & 40 & 11578.53 & 1425.73 & TL & 0 &\textbf{ 1850.44 }& \textbf{1511.29} & TL & 542 \\
		pr439 & 50 & 21664.40 & \textbf{1364.73} & TL & 0 & \textbf{1550.00} & \textbf{1364.73} & TL & 1707 \\
		pr439 & 60 & 21664.40 & \textbf{1364.73} & TL & 0 & \textbf{1364.73} & \textbf{1364.73} & \textbf{43.1} & 0 \\
		pr439 & 70 & 21664.40 & \textbf{1364.73 }& TL & 0 & \textbf{1364.73} & \textbf{1364.73} & \textbf{4.0} & 0 \\
		pr439 & 80 & 6930.02 & \textbf{1364.73} & TL & 0 & \textbf{1364.73} &\textbf{ 1364.73} & \textbf{2.1} & 0 \\
		pr439 & 90 & \textbf{1364.73} & \textbf{1364.73} & 476.0 & 0 & \textbf{1364.73} & \textbf{1364.73} & \textbf{1.7} & 0 \\
		pr439 & 100 & \textbf{1364.73} & 1\textbf{364.73} & 376.1 & 0 & \textbf{1364.73 }& \textbf{1364.73} & \textbf{2.5} & 0 \\
		\hline
	\end{tabular}
\end{table}

Tables~\ref{tab:runtime_tsp} and~\ref{tab:runtime_tsp2} show the results for the 
\tsplib instances and settings \texttt{1} and \texttt{1HSL}, reporting the same 
values for them as in Table~\ref{tab:runtime_pmed}. 
Although nearly all instances with at most $|I|=|J|=127$ locations were 
solved to optimality by both settings, the best runtimes are obtained by 
\texttt{1HSL}.  
For the larger instances, of which only some could be solved to optimality, the 
setting~\texttt{1HSL} provides better lower bounds. This is likely due to the 
addition of our upper bound and remoteness equalities, simple and general upper bound inequalities and extended lifted inequalities, as they yield 
significantly improved root bounds, 
which are then further improved in the \BC tree.
Moreover, the importance of our 
heuristics is reflected by the fact that for the largest instance, namely 
\texttt{pr439}, the obtained upper bounds for \texttt{1HSL} are much better 
than those for~\texttt{1}.

\section{Conclusions}\label{sec:conclusion}

In this work, we introduce the~$p$-$\alpha$-closest-center problem (\paCP), which 
is a generalization of the classical (discrete) $p$-center problem 
and the $p$-second-center problem. To better set the \paCP in context, 
we first prove relationships between the optimal objective function values for 
different versions of the $p$-center problem including the 
$\alpha$-neighbor-$p$-center problem and the $p$-next-center problem.

As our main contribution, we introduce 
the very first mixed-integer programming (MIP) formulations of the \paCP, which 
are 
also the very first for the recently introduced $p$-second-center problem, and  
present several valid inequalities and so-called optimality-preserving 
inequalities, i.e., inequalities that may cut off feasible solutions, but that do 
not change the optimal objective function value. 
We conduct a polyhedral study on 
our four MIP formulations and study their semi-relaxations, in which only 
some of the 
binary variables are relaxed.
For three of our MIP formulations, we strengthen their linear programming relaxations by 
introducing lifted inequalities, i.e., inequalities that incorporate a lower 
bound. 
We characterize the best lower bounds that are obtainable by iterative procedures 
based on these lifted inequalities similar to the work of 
\citet{gaar2022scaleable} and \citet{gaar2023exact} for other $p$-center problem 
variants.

We develop a branch-and-cut algorithm with various ingredients such as starting 
and primal heuristics, variable fixing, and separation of (in)equalities 
which are based on our theoretical results.
We test the effects of these ingredients in a computational 
study.
Out of 93 considered instances from the literature on $p$-center problem 
variants, we solve 52 to optimality. 
We also compare our solution algorithm to the heuristic of 
\citet{ristic2023solving} on their instance set and prove optimality of their 
solutions for~17 instances.

There are various avenues for further work. For example, one could try to extend 
other formulations of the $p$-center problem to the \paCP, such as the 
formulations presented in 
\citet{elloumi2004,calik2013double,ales2018,gaar2022scaleable}. Moreover, it could 
also be interesting to try to extend set cover-based approaches (see, e.g, 
\citet{chen2009,contardo2019scalable}) from the $p$-center problem to the \paCP. 
The development of further heuristics could also be a fruitful research direction, 
in particular to be able to tackle larger-scale instances of the \paCP.

\FloatBarrier

\ifArXiV
\bibliographystyle{plainnat}
\else
\bibliographystyle{elsarticle-harv}
\fi

\bibliography{biblio}

\end{document}